\documentclass[12pt]{article}

\usepackage{a4wide}
\usepackage[french,english]{babel}
\usepackage[utf8]{inputenc}
\usepackage[T1]{fontenc}
\usepackage{amsthm}
\usepackage{amssymb}
\usepackage{amsmath}
\usepackage{stmaryrd}
\usepackage{makeidx}
\usepackage{graphicx}
\usepackage{caption}
\theoremstyle{plain}
\usepackage[colorlinks=true,breaklinks=true,linkcolor=black]{hyperref} 
\usepackage{amsfonts}
\usepackage{appendix}
\usepackage[normalem]{ulem}
\usepackage{comment}
\usepackage{enumerate}

\newtheorem{Theorem}{Theorem}[section] 
\newtheorem{Theoremintro}{Theorem}

\newtheorem{Proposition}[Theorem]{Proposition}
\newtheorem{Corollary}[Theorem]{Corollary}
\newtheorem{Definition}[Theorem]{Definition}
\newtheorem{Notation}[Theorem]{Notation}

\newtheorem{Lemma}[Theorem]{Lemma}

\newtheorem{Remark}[Theorem]{Remark}
\newtheorem{Definition/Proposition}[Theorem]{Definition/Proposition}
\newtheorem*{Theorem*}{Theorem}
\usepackage{dsfont}
\date{}
\makeindex
\usepackage{mathrsfs}
\usepackage[all]{xy}
\theoremstyle{definition}

\title{Grading of affinized Weyl semi-groups of Kac-Moody type}
\author{ Paul \textsc{Philippe} \\ Université Jean Monnet, Saint-Étienne, France\\ Institut Camille Jordan, UMR 5208 CNRS \\ paul.philippe@univ-st-etienne.fr}

\makeatletter
\renewcommand\theequation%
{\thesection.\arabic{equation}}
\@addtoreset{equation}{section}
\makeatother

\makeatletter \@addtoreset{figure}{section}\makeatother

\newcommand{\R}{\mathbb{R}}

\newcommand{\N}{\mathbb{N}}
\newcommand{\Z}{\mathbb{Z}}

\newcommand{\htt}{\mathrm{ht}}

\newcommand{\efface}[1]{}

\newcommand{\pr}{\mathrm{proj}}

\newcommand{\cC}{\mathcal{C}}

\newcommand{\cP}{\mathcal{P}}

\newcommand{\bx}{\mathbf{x}}

\newcommand{\by}{\mathbf{y}}

\newcommand{\bz}{\mathbf{z}}

\renewcommand{\phi}{\varphi}
\renewcommand{\emptyset}{\varnothing}

\renewcommand{\tilde}[1]{\widetilde{#1}}

\makeatletter
\def\Ddots{\mathinner{\mkern1mu\raise\p@
\vbox{\kern7\p@\hbox{.}}\mkern2mu
\raise4\p@\hbox{.}\mkern2mu\raise7\p@\hbox{.}\mkern1mu}}
\makeatother

\renewcommand{\hom}{\operatorname{Hom}}

\DeclareMathOperator{\sgn}{sgn}

\newcommand{\Inv}{\mathrm{Inv}}

\usepackage[pagewise]{lineno}

\hypersetup{
pdfauthor={P. Philippe},
pdftitle={Bruhat order},
}

\begin{document}
\setcounter{tocdepth}{2}
\bibliographystyle{alpha}
\maketitle
\begin{abstract}
    For any Kac-Moody root datum $\mathcal D$, D. Muthiah and D. Orr have defined a partial order on the semidirect product $W^a_+$ of the integral Tits cone with the vectorial Weyl group of $\mathcal D$, and a compatible length function. We classify covers for this order and show that this length function defines a $\mathbb Z$-grading of $W^a_+$, generalizing the case of affine ADE root systems and giving a positive answer to a conjecture of Muthiah and Orr.
\end{abstract}

\tableofcontents

\section*{Introduction}
\subsection*{Motivations}
\subsubsection*{Reductive groups over p-adic fields}
Let $\mathbf G$ be a split reductive group scheme with the data of a Borel subgroup $\mathbf B$ containing a maximal torus $\mathbf T$. Let $W=N_{\mathbf G}(\mathbf T)/\mathbf T$ be its vectorial Weyl group and $Y$ be its coweight lattice: $Y=\hom(\mathbb G_m,\mathbf T)$.
The action of $W$ on $\mathbf T$ induces an action of $W$ on $Y$ and allows to form the semidirect product $W^a=Y\rtimes W$. This group called the extended affinized Weyl group of $\mathbf G$, appears naturally in the geometry and the representation theory of $\mathbf G$ over discretely valued fields. A foundational work in this regard was done by N. Iwahori and H. Matsumoto in 1965 (\cite{iwahori1965bruhat}), when they exhibited a Bruhat decomposition of $\mathbf G(\mathbb Q_p)$ indexed by $W^a$.

Let $\mathcal K$ be a non-archimedean local field with ring of integers $\mathcal O_\mathcal K \subset \mathcal K$, uniformizer $\pi\in \mathcal O_\mathcal K$ and residue field $\mathds k_\mathcal K=\mathcal O_\mathcal K/\pi$.
Let $G=\mathbf G(\mathcal K)$, let $K=\mathbf G(\mathcal O_\mathcal K)$ its integral points and let $I$ be its Iwahori subgroup, defined as $I=\{g\in K \mid g\in \mathbf B(\mathds k_\mathcal K) \mod \pi \}.$ The extended affinized Weyl group can be understood as $N_G(\mathbf T(\mathcal K))/\mathbf T(\mathcal O_\mathcal K)$, so admits a lift in $G$. Then, $G$ admits a decomposition in $I$-double cosets indexed by $W^a$, the Iwahori-Matsumoto-Bruhat decomposition: \begin{equation}G=\bigsqcup\limits_{\pi^\lambda w \in W^a}I \pi^\lambda w I.\end{equation}
The group $W^a$ is a finite extension of a Coxeter group and thus admits a Bruhat order which arises from the geometry of the homogeneous space $G/I$: for any $\pi^\lambda w \in W^a$, $I \pi^\lambda w I$ is a subvariety of pure dimension $\ell(\pi^\lambda w)$ in $G/I$, and its closure admits a disjoint decomposition in $I$ orbits: \begin{equation}\label{eq: IM dec}\overline{I \pi^\lambda w I}=\bigsqcup\limits_{\pi^\mu v \leq \pi^\lambda w} I \pi^\mu v I\end{equation} which extends the Iwahori-Matsumoto decomposition. The connection between the geometry of $G/I$ and the combinatorial structure of $W^a$ is deeper. In particular, $R$-Kazhdan-Lusztig polynomials introduced by Kazhdan and Lusztig (\cite{kazhdan1980schubert}), defined as the number of points of certain intersections in $G/I$, are also given by a recursive formula based on the Bruhat order and the Bruhat length of $W^a$.

These polynomials appear in many topics around reductive groups over local fields, we aim to develop analogous polynomials when $\mathbf G$ is replaced by a general Kac-Moody group. 
\subsubsection*{Extension to Kac-Moody groups}
Replace $\mathbf G$ by a general split Kac-Moody group. Kac-Moody group functors are entirely defined by the underlying Kac-Moody root datum $\mathcal D$, as defined in \cite[\S 2]{remy2002groupes}, and reductive groups correspond to root data of finite type. Then the Iwahori-Matsumoto decomposition no longer holds on $G=\mathbf G(\mathcal K)$. However there is a partial Iwahori-Matsumoto decomposition: there exists a subsemigroup $G^+$ of $G$ such that:
\begin{equation}G^+=\bigsqcup\limits_{\pi^\lambda w\in W^a_+} I \pi^\lambda w I. \end{equation}
The indexing set for this decomposition $W^a_+$ is a subsemigroup of $W^a=Y\rtimes W$, and it appears naturally in other related contexts, for example when trying to construct an Iwahori-Hecke algebra for $G$ \cite{braverman2016iwahori,bardy2016iwahori}. Let us briefly explain how $W^a_+$ is defined:

Let $\Phi$ be the real root system of the root datum $\mathcal D$. It is an infinite set (unless $\mathcal D$ is reductive) of linear forms on $Y$ coming with a subset of positive roots $\Phi_+ \subset \Phi$ such that $\Phi=\Phi_+ \sqcup -\Phi_+$. Let $Y^{++}=\{\lambda \in Y \mid \forall \alpha \in \Phi_+,\; \alpha(\lambda)\geq 0 \}$ and $Y^+=W\cdot Y^{++}$. Then $W^a_+$ is defined as $Y^+\rtimes W$. In the reductive case, $Y^+$ coincides with $Y$ and thus $W^a_+=W^a$. 
However, $W^a$ can no longer be conceived as a finite extension of a Coxeter system, hence there is a priori no Bruhat order on $W^a_+$, let alone on $W^a$. A well-behaved topology on $G^+/I$ would allow to define an order on $W^a_+$ through the analog of decomposition~\eqref{eq: IM dec}, but $G^+/I$ does not seem to have a natural variety, nor even an ind-variety structure.

\subsubsection*{An order and two lengths on $W^a_+$}
In appendix of their article on the construction of an Iwahori-Hecke algebra for $G$ an affine Kac-Moody group over a p-adic field \cite[Appendix B2]{braverman2016iwahori}, A. Braverman, D. Kazhdan and M. Patnaik propose the definition of a preorder on $W^a_+$ which would replace the Bruhat order of $W^a$ and they conjecture that it is a partial order. In \cite{muthiah2018iwahori}, D. Muthiah extends the definition of this preorder to any Kac-Moody group $G$, defines a $\mathbb Z \oplus \varepsilon \mathbb Z$-valued length compatible with this preorder and hence shows that it is an order. In a joint work \cite{muthiah2019bruhat}, D. Muthiah and D. Orr then show that this length can be evaluated at $\varepsilon = 1$ to obtain a $\mathbb Z$-valued length strictly compatible with the order on $W^a_+$.

In order to build a Kazhdan-Lusztig theory of p-adic Kac-Moody groups, we want to understand how close this order is to the Bruhat order of an affine Coxeter group, which properties still hold and which do not. The definition of a $\mathbb Z$-length is already a significant step, but many important properties, which are known to hold for Bruhat orders, remain unknown in this context. Several were proved only for Kac-Moody root systems of affine simply laced type using the specific structure of affinized Weyl group of $W$ in this context.
\subsubsection*{Choice of vocabulary}The order on $W^a_+$ is often mentioned in the literature as "double affine Bruhat order" and the associated length as "double affine Bruhat length" because it is most studied in the case of $G$ a Kac-Moody group of affine type (in which case $W$ is an affine Weyl group). In this paper we refer to it as "the affinized Bruhat order" and "the affinized Bruhat length", denoted $\ell^a$, because we do not suppose that $W$ is an affine Weyl group. Note that, if $W$ is finite, then the affinized Bruhat length and order on $W^a_+$ are just the ones induced by its Coxeter group structure.  
\subsubsection*{Main result}
Our main result is a positive answer to \cite[Conjecture 1.5]{muthiah2019bruhat} in full generality:

For any partial order $\leq$ on a set $X$, we say that $\by$ covers $\bx$ if $\bx\neq \by$ and $\{\bz \in X \mid \bx\leq \bz\leq \by\}=\{\bx,\by\}$. A grading of $X$ is a length function $\ell$ on $X$ strictly compatible with $\leq$ and such that $\by$ covers $\bx$ if and only if $\bx \leq \by$ and $\ell(\by)-\ell(\bx)=1$. Gradings thus give an easy classification of covers and more generally of saturated chains in $X$. The Bruhat length for a Coxeter group equipped with the Bruhat order is the prototypical example of a grading.

In \cite{muthiah2019bruhat}, Muthiah and Orr prove that if $\Phi$ is of affine ADE type, the affinized Bruhat length gives a $\mathbb Z$-grading of $W^a_+$ for the affinized Bruhat order and conjecture this to be true in general. The main result of this paper is a positive answer to this conjecture:
\begin{Theoremintro}\label{Theorem : maintheoremintro}
    Let $\mathcal D$ be any Kac-Moody root datum. Then the affinized length $\ell^a$ on $W^a_+$ defines a $\mathbb Z$-grading of $W^a_+$ strictly compatible with the affinized Bruhat order.
    Otherwise said, let $\bx,\by \in W^a_+$ be such that $\bx\leq\by$. Then
    \begin{equation}\by \text{ covers } \bx \text{ if and only if } \ell^a(\by)-\ell^a(\bx)=1.\end{equation}
\end{Theoremintro}

Along the way, we obtain several geometric properties of covers for the affinized Bruhat order which we expect to be insightful even if the root datum is reductive (so $W$ is finite and $W^a_+$ is an affine Weyl group) as they only rely on the Coxeter structure of $W$. In particular, we obtain in Proposition~\ref{Proposition classification general} a classification of covers which generalize results obtained using quantum Bruhat graphs, in the reductive setting by T. Lam and M. Shimozono (\cite[Proposition 4.4]{lam2010quantum}) and F. Schremmer \cite[Proposition 4.5]{schremmer2023affine}, and in the affine simply laced setting by A. Welch (\cite[Theorem 2]{welch2022classification}).

\subsubsection*{Further directions}
In an upcoming joint work with A. Hébert, we prove that any element of $W^a_+$ admits a finite number of covers for the affinized Bruhat order. We use this finiteness in the context of masures to define $R$-Kazhdan-Lusztig polynomials, following Muthiah's strategy exposed in \cite{muthiah2019double} and the work on twin masures of N. Bardy-Panse, A. Hébert and G. Rousseau \cite{twinmasures}. Our understanding of covers is useful to compute these $R$-polynomials and we intend to use $R$-polynomials to define $P$-Kazhdan-Lusztig polynomials.

Another interesting (but quite long reach) question is the following: $W^a_+$ appears as the affinization of $W$, which may be taken as an affinized version of a finite Coxeter group. Can we iterate the affinization process, e.g to obtain a valid theory for reductive or Kac-Moody groups on valued fields of higher dimensions?

Finally, little is known on the preorder defined on the whole semidirect product $W^a$, it could be insightful to study it and to connect it to the failure of the full Iwahori-Matsumoto decomposition of $G$.
\subsection*{Organization of the paper}
\subsubsection*{Proof strategy}
 The global strategy is to construct a non-trivial chain from $\bx$ to $\by$ every time $\by\geq \bx$ verifies $\ell^a(\by)-\ell^a(\bx)>1$. Let $\pr^{Y^+}$ denote the projection $W^a_+=Y^+\rtimes W \rightarrow Y^+$. We distinguish two cases which depend on the form of $\bx$ and $\by$: The first case is when $\pr^{Y^+}(\by)$ lies in the orbit of $\pr^{Y^+}(\bx)$, we call such covers the vectorial covers. The other case is when $\pr^{Y^+}(\by)\notin W\cdot\pr^{Y^+}(\bx)$, we call such covers the properly affine covers.

If $\pr^{Y^+}(\by)\in W\cdot\pr^{Y^+}(\bx)$, we show that the affinized Bruhat order on $\{\bz \in W^a_+ \mid \bx \leq \bz \leq \by\}$ is, in some sense, a lift of several Bruhat-like orders on $W$. We are then able to construct chains between $\bx$ and $\by$ from chains in $W$, we deduce a classification of vectorial covers. The characterization of properly affine covers is, at first glance, more involved. Through a careful study of the relation between the vectorial chamber containing $\pr^{Y^+}(\bx)$ and the vectorial chamber containing $\pr^{Y^+}(\by)$, we show that the length difference $\ell^a(\by)-\ell^a(\bx)$ can be rewritten in a more workable form, making clear the conditions for which it is equal to one. Then the difficulty is to build, explicitly, a non-trivial chain every time one of these conditions is not satisfied. 
\subsubsection*{Organization} 
Section~\ref{section: preliminaries} consists of preliminaries. In Subsection~\ref{subsection: definitions} we formally define everything we mentioned in this introduction. In particular we give the definition of the affinized Bruhat order and the two affinized Bruhat lengths as they are given in \cite{muthiah2019bruhat}. To be more flexible, we chose to define the affinized Bruhat order on the whole affinized Weyl group $W^a=Y\rtimes W$, on which it may not be a preorder.

We show, amongst other preliminary results, that we indeed recover the affinized Bruhat order on $W^a_+$ from this preorder in Subsection~\ref{subsection: prelimary results}.

We also give, in Subsection~\ref{subsection: geometric interpretation}, a geometric interpretation of $W^a_+$ and its affinized Bruhat order, which is to be compared with the interpretation of the Bruhat order in the Coxeter complex of a Coxeter group. Even though it is not clearly mentioned in the rest of the paper, this geometric interpretation was very useful to construct chains and understand $W^a_+$.

\vspace{0.5 cm}

In Section~\ref{Section: constant dominance class}, we prove Theorem~\ref{Theorem : maintheoremintro} for vectorial covers. We define relative versions of the Bruhat order on $W$ in Subsection~\ref{subsection: 2.1} and we connect theses relative Bruhat orders to the affinized Bruhat length in Subsection~\ref{subsection 2.2}. This is enough to prove Theorem~\ref{Theorem : maintheoremintro} when $\pr^{Y^+}(\by)=\pr^{Y^+}(\bx)$ (see Theorem~\ref{Theorem: constantcoweight}). Using finer results on parabolic quotients in Subsection~\ref{subsection: 2.3}, we extend it to vectorial covers such that $\pr^{Y^+}(\by)\in W\cdot\pr^{Y^+}(\bx)\setminus \{\pr^{Y^+}(\bx)\}$ (see Theorem~\ref{non_reg_covers_same_orbit}). 

\vspace{0.5 cm}

In Section~\ref{section: varying dominant coweight}, we deal with properly affine covers. We first show in Subsection~\ref{subsection: 3.1} that these covers are of a very specific form. Namely, if $\bx=\pi^{v(\lambda)}w$ with $v,w\in W$ and $\lambda \in Y^{++}$, then $\by$ needs to be of the form $\pi^{v(\lambda+\beta^\vee)}s_{v(\beta)}w$ or $\pi^{vs_\beta(\lambda+\beta^\vee)}s_{v(\beta)}w$ for some $\beta \in \Phi_+$.

The strategy is then to get enough necessary conditions on $v,w,\lambda,\beta$ for $\by$ to cover $\bx$, in order to obtain a simplified expression for $\ell^a(\by)-\ell^a(\bx)$. Proposition~\ref{covers_minimal_galleries} gives a first result in this direction. In Subsection~\ref{subsection: 3.2} we fully exploit this strategy to obtain Expression~\eqref{eq: length difference generalcase} for the length difference.

Finally, in Subsections~\ref{subsection: almost dominance} and~\ref{subsection: quantum roots}, we construct various chains from $\bx$ to $\by$ to prove that the quantities appearing in Expression~\eqref{eq: length difference generalcase} need to be minimal when $\by$ covers $\bx$, which allows us to conclude the argument in Subsection~\ref{subsection: conclusion}. 

\subsubsection*{Acknowledgements}
I am especially grateful to Auguste Hébert for his help in understanding the affinized Bruhat order and his feedback throughout the conception of this paper, as well as to Stéphane Gaussent for his supervision and his careful rereading. I would also like to thank the anonymous referee for his help and corrections.

\newpage

\section{Preliminaries} \label{section: preliminaries}

\subsection{Definitions and notation} \label{subsection: definitions}

Let $\mathcal D = (A,X,Y,(\alpha_i)_{i \in I},(\alpha_i^\vee)_{i \in I})$ be a Kac-Moody root datum as defined in \cite[\S 8]{remy2002groupes}. It is a quintuplet such that:
\begin{enumerate}
    \item $I$ is a finite indexing set and $A=(a_{ij})_{(i,j)\in I\times I}$ is a generalized Cartan matrix.
    \item $X$ and $Y$ are two dual free $\mathbb Z$-modules of finite rank, we write $\langle \cdot,\cdot \rangle$ the duality bracket.
    \item $(\alpha_i)_{i\in I}$ (resp. $(\alpha_i^\vee)_{i \in I}$) is a family of linearly independent elements of $X$ (resp. $Y$): the simple roots (resp. simple coroots).
    \item For all $(i,j) \in I^2$ we have $\langle \alpha_i^\vee,\alpha_j \rangle = a_{ij}$.
\end{enumerate}

\subsubsection{Vectorial Weyl group} For every $i \in I$ set $s_i \in \operatorname{Aut}_\mathbb Z(X):  x \mapsto x- \langle \alpha_i^\vee, x \rangle \alpha_i$. The generated group $W=\langle s_i \mid i \in I \rangle$ is the \textbf{vectorial Weyl group} of the Kac-Moody root datum. 

The duality bracket $\langle \cdot,\cdot\rangle$ induces a contragredient action of $W$ on $Y$, explicitly $s_i(y)=y-\langle y,\alpha_i\rangle \alpha_i^\vee$. The bracket is then $W$-invariant.

The vectorial Weyl group $W$ is a Coxeter group with set of simple reflections $S=\{s_i\mid i\in I\}$, in particular it has a Bruhat order $<$ and a length function $\ell$ compatible with the Bruhat order. We refer to \cite{bjorner2005combinatorics} for general definitions and properties of Coxeter groups. A \textbf{reflection} in a Coxeter group is any element conjugated to a simple reflection.

\subsubsection{Real roots} Let $\Phi=W\cdot\{\alpha_i \mid i \in I \}$ be the set of real roots of $\mathcal D$, it is a root system in the classical sense, but possibly infinite.
In particular let $\Phi_+= \bigoplus\limits_{i\in I}\mathbb N \alpha_i \cap \Phi$ be the set of positive real roots, then $\Phi=\Phi_+ \sqcup -\Phi_+$, we write $\Phi_-=-\Phi_+$ the set of negative roots.

The set $\Phi^\vee=W\cdot\{\alpha_i^\vee \mid i \in I\}$ is the set of \textbf{coroots}, and its subset $\Phi^\vee_+=\bigoplus\limits_{i\in I}\mathbb N \alpha_i^\vee \cap \Phi^\vee$ is the set of \textbf{positive coroots}.

To each root $\beta$ corresponds a unique coroot $\beta^\vee$: if $\beta=w(\alpha_i)$ then $\beta^\vee=w(\alpha_i^\vee)$. This map $\beta\mapsto\beta^\vee$ is well defined, bijective between $\Phi$ and $\Phi^\vee$ and sends positive roots to positive coroots. Note that $\langle \beta^\vee,\beta\rangle = 2$ for all $\beta \in \Phi$.

Moreover to each root $\beta$ we associate a reflection $s_\beta \in W$: If $\beta=w(\pm\alpha_i)$ then $s_\beta:= ws_iw^{-1}$. Explicitly it is the map $Y\rightarrow Y$ defined by $s_\beta(x)=x-\langle \beta^\vee , x \rangle \beta$. For any $\beta\in\Phi$ we have $s_\beta=s_{-\beta}$ and the map $\beta \mapsto s_\beta$ forms a bijection between the set of positive roots and the set $\{ws_iw^{-1}\mid (w,i)\in W\times I\}$ of reflections of $W$. 

\subsubsection{Inversion sets}
For any $w \in W$, let $\Inv(w)=\Phi_+ \cap w^{-1}(\Phi_-)= \{ \alpha \in \Phi_+ \mid w(\alpha) \in \Phi_- \}$. Theses sets are strongly connected to the Bruhat order, as by \cite[1.3.13]{kumar2002kac}, for all $\alpha \in \Phi_+$
\begin{equation}
    \alpha \in \Inv(w) \iff ws_\alpha < w \iff s_\alpha w^{-1} < w^{-1}.\end{equation}
Moreover, they are related to the Bruhat length: $\ell(w)=|\Inv(w)|$ (\cite[1.3.14]{kumar2002kac}).

\subsubsection{Fundamental chamber and Tits cone}
We define the (closed) integral fundamental chamber by $Y^{++}=\{\lambda \in Y \mid \langle \lambda,\alpha_i \rangle \geq 0 \; \forall i \in I \}$. If $\lambda\in Y^{++}$, we say that it is a \textbf{dominant coweight}. Then, the integral Tits cone is $Y^{+}:=\bigcup\limits_{w\in W} w(Y^{++})$. It is a convex cone of $Y$, in particular it is a semigroup for the group operation of $Y$, and it is equal to $Y$ if and only if $W$ is finite, if and only if $\Phi$ is finite, if and only if $A$ is of finite type (see \cite[1.4.2]{kumar2002kac}).

The integral fundamental chamber $Y^{++}$ is a fundamental domain for the action of $W$ on $Y^{+}$, and for any $\lambda\in Y^+$ we define $\lambda^{++}$ to be the unique element of $Y^{++}$ in its $W$-orbit.

There is a height function on $Y^+$, defined as follows:

\begin{Definition}
Let $(\Lambda_i)_{i \in I}$ be a set of fundamental weights, that is to say $\langle \alpha_i^\vee,\Lambda_i \rangle = \delta_{ij}$ for any $i,j \in I$. We fix it once and for all. Let $\rho= \sum_{i \in I} \Lambda_i$. Then for any $\lambda \in Y$ define the \textbf{height} of $\lambda$ as: \begin{equation}\htt(\lambda)=\langle \lambda, \rho \rangle.    
\end{equation}
The height depends on the choice of fundamental weights, but its restriction to $Q^\vee=\bigoplus\limits_{i\in I} \mathbb Z \alpha_i^\vee$ does not: $\htt(\sum_{i \in I} n_i \alpha_i^\vee)=\sum_{i \in I} n_i$.
\end{Definition}

\subsubsection{Parabolic subgroups, minimal coset representatives}
For $\lambda \in Y^+$, let $\Phi_\lambda$ denote the set $\{\alpha \in \Phi \mid \langle \lambda, \alpha \rangle =0\}$ and $W_\lambda=\operatorname{Stab}_{W}(\lambda)$. We say that $\lambda$ is \textbf{regular} if $\Phi_\lambda=0$, or equivalently if $W_\lambda=1_{W}$. More generally we say that $\lambda$ is \textbf{spherical} if $W_\lambda$ is finite.

Let $v\in W$ be such that $\lambda =v\lambda^{++}$. Then $W_\lambda v = vW_{\lambda^{++}}$ and, since $\lambda^{++}$ is dominant, $W_{\lambda^{++}}$ is a standard parabolic subgroup, that is a group of the form $W_J = \langle s \mid s \in J\rangle$ where $J\subset S$ is a set of simple reflections. More precisely, $J=\{s\in S\mid s(\lambda^{++})=\lambda^{++}\}$.

By standard Coxeter group theory (see for instance \cite[Section 2.2]{bjorner2005combinatorics}), for any $u\in W$, the left coset $uW_{\lambda^{++}}=u
W_J$ has a unique representative of minimal length which we denote $u^J$, and one has a decomposition $u=u^J u_J$ with $u_J \in W_J$ such that: \begin{equation}\label{eq : additive parabolic length}\ell(u)=\ell(u^J)+\ell(u_J).\end{equation}
\begin{Notation}\label{Notation : paraboliccosets}

\begin{enumerate}
    \item For any $J\subset S$, we denote by $W^J$ the set of minimal length representatives for $W_J-$cosets in $W$: 
    \begin{equation}w\in W^J \iff \forall \tilde w \in W_J,\; \ell(w\tilde w)>\ell(w)\iff \forall s\in J,\; \ell(ws)>\ell(w).\end{equation}
    If $\lambda\in Y^{++}$ is such that $W_\lambda=W_J$, then we may use $W^\lambda$ as an alternative notation for $W^J$. 
    \item For any $\lambda \in Y^+$ (not necessarily dominant), we denote by $v^\lambda$ the minimal length element in $W$ which verifies $\lambda=v^\lambda \lambda^{++}$:
    \begin{equation}
        v^\lambda = \min\{v\in W \mid \lambda=v\lambda^{++}\} .\end{equation} In other words, for any $u \in W$ such that $\lambda = u\lambda^{++}$, we have $v^\lambda=u^J$, where $J$ is the set of simple reflections such that $W_J=W_{\lambda^{++}}$.
    \end{enumerate}
\end{Notation}

\subsubsection{Affinized Weyl semigroup}
The action of $W$ on $Y$ allows to form the semidirect product $Y\rtimes W$, which we denote $W^a$. We denote its elements by $\pi^\lambda w$ with $\lambda \in Y, w\in W$.

By definition, $Y^+\subset Y$ is stable by the action of $W$ on $Y$, therefore we can form $W^a_+=Y^+ \rtimes W$ which is a subsemigroup of $W^a$. This semigroup is called the affinized Weyl semigroup. In \cite{muthiah2019bruhat} Muthiah and Orr define a Bruhat order and an associated length function on $W^a_+$ which we aim to study in this article.

Denote by $\pr^{Y^+}:W^a_+\rightarrow Y^+$ the canonical projection, which sends $\pi^\lambda w$ onto $\lambda$. Moreover denote by $\pr^{Y^{++}}:W^a_+\rightarrow Y^{++}$ the projection to $Y^{++}$: $\pr^{Y^{++}}(\bx)=(\pr^{Y^+}(\bx))^{++}$. Let us call $\pr^{Y^+}(\bx)$ the \textbf{coweight} of $\bx$, and $\pr^{Y^{++}}(\bx)$ its \textbf{dominance class}.
\subsubsection{Affinized roots}
Let $\Phi^a=\Phi \times \mathbb Z$ be \textbf{the set of affinized roots} and denote by $\beta+n\pi$ the affinized root $(\beta,n)$. The affinized root $\beta+n\pi$ is said to be positive if $n>0$ or ($n=0$ and $\beta \in \Phi_+$) and we write $\Phi^a_+$ for the set of positive affinized roots. We have $\Phi^a = \Phi^a_+ \sqcup -\Phi^a_+$.

The semidirect product $W^a$ acts on $\Phi^a$ by:
\begin{equation}\label{eq : W^a_+ action}
    \pi^\lambda w(\beta+n\pi)=w(\beta)+(n+\langle \lambda,w(\beta)\rangle)\pi.
\end{equation}

For any $n \in \mathbb Z$, its sign is denoted $\sgn(n)\in \{-1,+1\}$, with the convention that $\sgn(0)=+1$. Note that $|n|=\sgn(n)n$. We also define the sign of an affinized root: $\sgn(\beta+n\pi)\in \{-1,+1\}$ and $\sgn(\beta+n\pi)=+1$ if and only if $\beta+n\pi \in \Phi_+^a$. 

For $n\in \Z$ and $\beta\in \Phi_+$, set:\begin{align}\label{eq : affine_roots}
&\beta[n]=\sgn(n) \beta+ |n|\pi\in \Phi_+^a \\
&s_{\beta[n]}=\pi^{n\beta^\vee} s_\beta.\end{align}
We also define $\beta[n]\in \Phi^a_+$ for $\beta\in\Phi_-$, by $\beta[n]=(-\beta)[-n]$, and $s_{\beta[n]}=s_{-\beta[-n]}=\pi^{n\beta^\vee}s_\beta$. The affinized root $\beta[n]$ is therefore the positive affinized root within the pair $\{\beta+n\pi, -(\beta+n\pi)\}$. Note that $s_{\beta[0]}$ is the vectorial reflection $s_\beta$.
\subsubsection{Bruhat order on $W^a_+$} Recall Braverman, Kazhdan and Patnaik's definition of the Bruhat order $<$ introduced in \cite[Section B. 2]{braverman2016iwahori}:
Let $\bx \in W^a_+$ and let $\beta[n] \in \Phi_+^a$ be such that $\bx s_{\beta[n]} \in W^a_+$. Then,
\begin{equation}\label{eq : Bruhat order}
    \bx<\bx s_{\beta[n]} \iff \sgn(\beta+n\pi)=\sgn(\bx(\beta+n\pi))\iff \bx(\beta[n]) \in \Phi_+^a.
\end{equation}
Explicitly, if $\bx=\pi^\lambda w \in W^a_+$, the right-hand side condition writes out:
\begin{align*}
    &\sgn(n)(n+\langle \lambda,w(\beta)\rangle) > 0 \\
    &\text{or} \; n=-\langle \lambda,w(\beta)\rangle \; \text{and} \; \sgn(n)w(\beta)>0.
\end{align*}
Then we extend this relation by transitivity, which makes it a preorder. Originally, Braverman, Kazhdan and Patnaik defined it only for affine vectorial Weyl groups, but the definition extends to any vectorial Weyl group and Muthiah showed in \cite{muthiah2018iwahori} that it is an order on $W^a_+$ in general.

\subsubsection{Extension to $W^a$}
As the whole semidirect product $W^a$ acts on $\Phi_+^a$, Formula~\eqref{eq : Bruhat order} makes sense for any $\bx \in W^a$, and in this paper we define $<$ on $W^a$, as the closure by transitivity of the relation defined through~\eqref{eq : Bruhat order} for $\bx \in W^a$. We show in the next section that if $\bx<\by$ and $\by \in W^a_+$, then $\bx \in W^a_+$. This ensures that the restriction of the $W^a$-preorder to $W^a_+$ coincides with Braverman, Kazhdan, Patnaik's order on $W^a_+$. However $<$ may not be an order on $W^a$.
\subsubsection{Bruhat order through a right action}

In this paper, we consider multiplication by reflections on the left. To switch between the right and left actions note that:
\begin{equation}\label{eq : lefttorightmultiplication}s_{\beta[n]}\pi^\lambda w=\pi^{s_\beta\lambda + n\beta^\vee}s_{\beta}w=\pi^{\lambda +(n-\langle\lambda,\beta\rangle) \beta^\vee}ws_{w^{-1}(\beta)}= \pi^\lambda w s_{w^{-1}(\beta)[n-\langle \lambda,\beta\rangle]}.\end{equation} 
In particular:
    \begin{equation}\label{eq : productdominanceclass}
     s_{\beta[0]}\pi^\lambda w = s_{\beta}\pi^\lambda w = \pi^{s_\beta \lambda}s_\beta w \text{ and }s_{\beta[\langle \lambda,\beta\rangle]}\pi^\lambda w=\pi^\lambda s_\beta w=\pi^\lambda w s_{w^{-1}(\beta)}.
    \end{equation}

Using Formula~\eqref{eq : lefttorightmultiplication}, the affinized Bruhat order can be recovered using a right action of $W^a$ on $\Phi^a_+$.
\begin{Proposition}
    Let $\pi^\lambda w \in W^a$ and $(\beta,n)\in \Phi \times \mathbb Z \setminus( \Phi_-\times \{0\})$. Then

\begin{equation}\label{eq: leftaction}
    s_{\beta[n]}\pi^\lambda w > \pi^\lambda w \iff \sgn(n)w^{-1}(\beta)+(|n|-\sgn(n)\langle \lambda,\beta\rangle)\pi \in \Phi^a_+.
\end{equation}
\end{Proposition}
\begin{Remark}
    The root appearing in the right-hand side of Equivalence~\eqref{eq: leftaction} is the affinized root $(\pi^\lambda w)^{-1}(\beta[n])$.
\end{Remark}
\begin{proof}
    Let $\pi^\lambda w \in W^a$ and $\beta+n\pi\in \Phi^a$. Then by Formulas~\eqref{eq : Bruhat order} and~\eqref{eq : lefttorightmultiplication}, \begin{equation*}
        s_{\beta[n]}\pi^\lambda w >\pi^\lambda w \iff \sgn(\beta+n\pi)= \sgn(w^{-1}(\beta)+(n-\langle \lambda,\beta\rangle)\pi). 
    \end{equation*} If $(\beta,n)\notin \Phi_-\times \{0\}$, then $\beta[n]=\sgn(n)(\beta+n\pi)$ so this is equivalent to $\sgn(n)(w^{-1}(\beta)+(n-\langle \lambda,\beta\rangle)\pi)$ being a positive affinized root, which is Equivalence~\eqref{eq: leftaction}.
\end{proof}
Note that Formula~\eqref{eq: leftaction} is no longer correct if $\beta \in \Phi_-$ and $n=0$, in which case it needs to be applied to $(-\beta)[0]$.
Applying reflections on the left is better suited for the geometric interpretation we will give in Subsection \ref{subsection: geometric interpretation}.

\subsubsection{Terminology on partially ordered sets}

For $p\leq q \in \Z$, we denote by $\llbracket p,q\rrbracket$ the set $\{r\in \mathbb Z \mid p\leq r\leq q\}$. If $p>q$, then $\llbracket p,q\rrbracket$ is another notation for $\llbracket q,p\rrbracket$. We also write $\rrbracket p,q\llbracket$ for $\llbracket p,q\rrbracket\setminus\{p,q\}$.

Let $(\cP,\leq)$ be a partially ordered set. 
For $\bx,\by\in \cP$, we say that $\bx$ and $\by$ are comparable if $\bx\leq \by$ or $\by\leq \bx$. We say that $\by$ \textbf{covers} $\bx$ if $\bx\neq \by$ and $\{\bz \mid \bx \leq \bz\leq \by\}=\{\bx,\by\}$, we then write $\bx\lhd \by$.
If $\cP=W^+_a$, covers $\bx \lhd \by$ such that $\pr^{Y^{++}}(\by)=\pr^{Y^{++}}(\bx)$ are called \textbf{vectorial covers} and covers which are not vectorial covers are called \textbf{properly affine covers}.

A \textbf{chain} from $\bx$ to $\by$ is a finite sequence $(\bx_0,\dots, \bx_n)$ such that $\bx_0=\bx$, $\bx_n=\by$ and $\bx_k\leq \bx_{k+1}$ for all $k\in \llbracket0,n-1\rrbracket$. If $\cP=W^a_+$ (resp. if $\cP$ is a Coxeter group), we add the condition that $\bx_{k+1}\bx_k^{-1}$ is an affinized reflection (resp. a reflection). A chain is \textbf{saturated} if $\bx_k\lhd \bx_{k+1}$ for all $k\in \llbracket0,n-1\rrbracket$. 
We say that a subset $\cC$ of $\cP$ is \textbf{convex} if, for all $\bx,\by\in \cC$ and $\bz\in \cP$,
\begin{equation}\bx\leq \bz\leq \by \implies \bz\in \cC.\end{equation}
Equivalently, $\cC$ is convex if and only if any chain from one element of $\cC$ to another is contained in $\cC$.

Let $\ell:\cP\rightarrow A$ be a function with values in a totally ordered set $A$, we say that it is \textbf{order-preserving} if, for all $\bx,\by\in \cP$,
\begin{equation}
    \bx \leq \by \implies \ell(\bx)\leq \ell(\by).
\end{equation}We say that $\ell$ is a \textbf{strictly compatible} ($A$-valued) length function if:
\begin{equation}\bx < \by \iff \bx,\by \text{ are comparable and } \ell(\bx)< \ell(\by).\end{equation}

We say that a strictly compatible length function $\ell$ defines a \textbf{grading} of $\cP$ if:
\begin{equation}\bx\lhd \by \iff \bx\leq \by \text{ and } \ell(\bx)\lhd\ell(\by).\end{equation}
In particular if $A=\Z$ or $A=\N$, this rewrites 
\begin{equation}\bx\lhd \by \iff \bx\leq \by \text{ and } \ell(\by)=\ell(\bx)+1.\end{equation}
For instance, the Bruhat length on a Coxeter group $W$ is strictly compatible with the Bruhat order, and defines a $\N$-valued grading of $W$. Muthiah and Orr associated length functions strictly compatible with the Bruhat order on $W^a_+$, generalizing the classical Bruhat length on Coxeter groups. We now formally introduce these lengths.

\subsubsection{Length functions on $W^a_+$}

\begin{Definition}
The \textbf{affinized length function} is the map $W^a_+ \rightarrow \mathbb Z \oplus \varepsilon \mathbb Z$ defined by:
\begin{equation} 
    \ell^a_\varepsilon(\pi^\lambda w)=2\htt(\lambda^{++})+\varepsilon(|\{\alpha \in \Inv(w^{-1}) \mid \langle \lambda,\alpha \rangle \geq 0 \}|-|\{\alpha \in \Inv(w^{-1}) \mid \langle \lambda,\alpha \rangle < 0 \}|).
\end{equation}
The \textbf{affinized length with integral values} is the affinized length function on which we set $\varepsilon = 1$:
  \begin{equation}  \ell^a(\pi^\lambda w)=2\htt(\lambda^{++})+(|\{\alpha \in \Inv(w^{-1}) \mid \langle \lambda,\alpha \rangle \geq 0 \}|-|\{\alpha \in \Inv(w^{-1}) \mid \langle \lambda,\alpha \rangle < 0 \}|).
\end{equation}
\end{Definition}

\begin{Theorem}\cite[Theorem 4.24]{muthiah2018iwahori} \cite[Theorem 3.6]{muthiah2019bruhat}\label{Theorem Muthiah Orr}

The affinized length function and the affinized length function with integral values are strictly compatible with the affinized Bruhat order on $W^a_+$. In other words, for any $\bx\in W^a_+$ and $\beta[n]\in \Phi_+^a$,
  \begin{equation}
    \bx s_{\beta[n]} > \bx \iff \ell^a_\varepsilon(\bx s_{\beta[n]}) > \ell^a_\varepsilon(\bx) \iff \ell^a(\bx s_{\beta[n]}) > \ell^a(\bx).
\end{equation}  
In particular the affinized Bruhat order is a partial order.
\end{Theorem}
In what follows we will mostly use $\ell^a$ and rarely mention $\ell^a_\varepsilon$, we now refer to $\ell^a$ as the affinized Bruhat length.

\subsection{Geometric interpretation} \label{subsection: geometric interpretation}

We introduced everything in a very algebraic way, but there is a strong geometric intuition behind root systems, vectorial Weyl groups and the vectorial Bruhat order, developed for instance in the context of buildings in \cite{ronan1989lectures}. There is also a geometrical interpretation of the Bruhat order on $W^a_+$ which we develop in this paragraph, it takes place in the standard apartment of the masure associated to a Kac-Moody group with underlying Kac-Moody datum $\mathcal D$.
\vspace{0.5 cm}

Let $V=Y\otimes_\mathbb Z \mathbb R$, the lattice $X$ embeds in its dual $V^\vee$ and the vectorial Weyl group $W$ acts naturally on it. Inside $V$ we have the (closed) fundamental chamber $C^v_f=\{v \in V \mid \langle v,\alpha_i\rangle \geq 0\}$ and the Tits cone $\mathcal T = W\cdot C^v_f$. A \textbf{vectorial chamber} is a set of the form $w\cdot C^v_f$ for $w \in W$. Since the interior of $C^v_f$ has trivial stabiliser in $W$, the set of chambers is in natural bijection with $W$ by $w\mapsto  C^v_w:=w \cdot C^v_f$.
\vspace{0.5 cm}

To each root $\beta \in \Phi_+$ let $M_\beta=\{x \in V \mid \langle x,\beta \rangle = 0 \}$, it is an hyperplane of $V$ and, if $\beta=w(\alpha_i)$ with $\alpha_i$ a simple root, then $C^v_w\cap C^v_{ws_i}\subset M_\beta \cap \mathcal T $.  The intersection $C^v_w\cap C^v_{ws_i}$ is called the panel of type $s_i$ of $w$.

We can put a structure of simplicial complex on $\mathcal T$, for which chambers are the cells of maximal rank and panels are the cells of maximal rank within non-chambers. This simplicial complex is a realization of the \textbf{Coxeter complex} of $(W,S)$. Each wall splits the Tits cone in two parts, and separate the set of vectorial chambers in two: say that $C^v_w$ is on the positive side of $M_\beta$ if $w^{-1}(\beta)>0$. In particular since $\beta$ is a positive root, the positive side is always the one which contain the dominant chamber.

Then the vectorial Bruhat order can be interpreted by:
    $s_\beta w>w$ if and only if, when we split $\mathcal T$ along $M_\beta$ the chambers $C^v_w$ and $C^v_f$ are in the same connected component of $\mathcal T$, that is to say $C^v_w$ is on the positive side of $M_\beta$. 
    
Moreover $\Inv(w^{-1})$ the inversion set of $w^{-1}$ can be interpreted as the set of walls separating the chamber $C^v_w=w\cdot C^v_f$ from the fundamental chamber $C^v_f$.
\vspace{0.5 cm}

In Figure~\ref{fig: figure1} we represent the Tits cone and its structure for a root system of rank $2$ with Cartan matrix  $\begin{pmatrix}
        2 & -3 \\ -2 & 2
    \end{pmatrix}$, which is of indefinite type. The Tits cone is colored in blue, and the vectorial chamber $C^v_w$ is labelled by $w$. It is an approximation since $W$ is infinite.

\begin{figure}[ht]

    \centering
    \includegraphics[width=\textwidth]{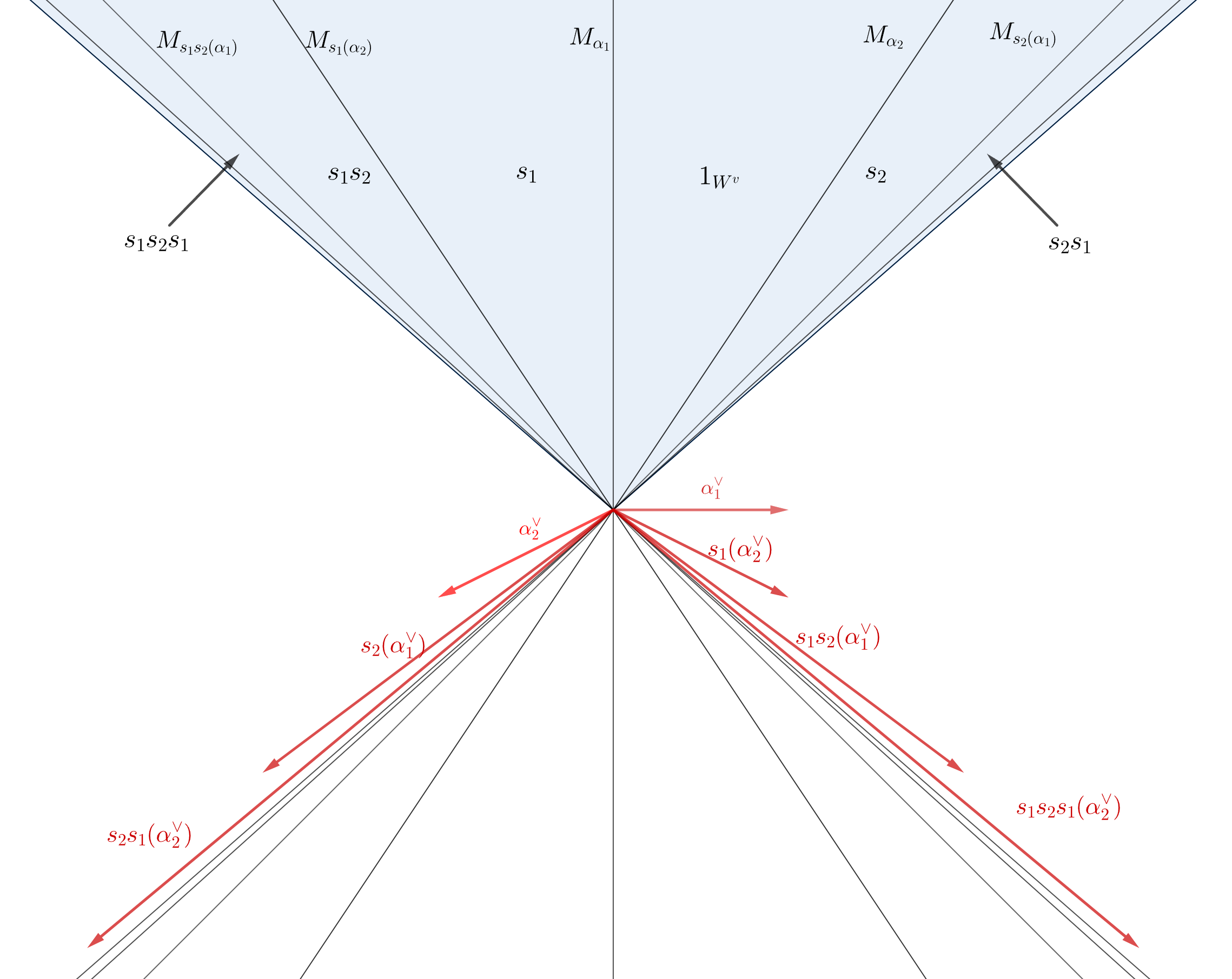}
    \caption{The Tits cone for a root system of Cartan Matrix $\begin{pmatrix}
        2 & -3 \\ -2 & 2
    \end{pmatrix}$.}
    \label{fig: figure1}

\end{figure}

\newpage

Let us now turn to the interpretation of the $W^a_+$-Bruhat order. Let $\mathbb A$ be a real affine space with direction $V$, we call $\mathbb A$ the (standard) affine apartment associated to $\mathcal D$. The tangent space of $\mathbb A$ is canonically isomorphic to $T\mathbb A = \mathbb A \times V$, with, for any $x \in \mathbb A$, $T_x\mathbb A = \{x\}\times V$.

The semigroup $W^a_+$ has an affine action on $\mathbb A$, given by $\pi^\lambda w(x)=-\lambda+w(x)$, it induces an action on the tangent space $T\mathbb A$ given by $\pi^\lambda w ((x,v))=(-\lambda+w(x),w(v))$.
To any positive affinized root $\beta[n] \in \Phi_+^a$ corresponds an affine hyperplane:
\begin{equation}
M_{\beta[n]}=\{x \in \mathbb A \mid \langle x,\beta\rangle +n=0 \},   
\end{equation}the \textbf{affine wall} associated to the affinized root $\beta[n]$. For any $x \in M_{\beta[n]}$ we have $T_xM_{\beta[n]}=\{x\} \times M_\beta \subset T_x\mathbb A$.

For any $\pi^\lambda w \in W^a_+$ let \begin{equation}
C_{\pi^\lambda w}=\{-\lambda\}\times C^v_w \subset T_{-\lambda} \mathbb A \subset T\mathbb A,    
\end{equation} we call it the \textbf{alcove of type $\pi^\lambda w$}. Mirroring the classical situation, $C_0=\{0\}\times C^v_f$ is a fundamental domain for the action of $W^a_+$ on $Y^+\times \mathcal T \subset T\mathbb A$ and $W^a_+$ acts on $\{C_{\bx}\mid \bx \in W^a_+\}$ simply transitively. Affine walls separate naturally the set of alcoves in two and we call the side containing $C_0$ the positive side.

Then the $W^a_+$-Bruhat order can be interpreted geometrically: \begin{equation}s_{\beta[n]}\pi^\lambda w >\pi^\lambda w \iff C_{\pi^{\lambda}w} \text{ is on the positive side of }M_{\beta[n]}.\end{equation}

We give an illustration of the affine apartment in Figure~\ref{fig: figure2} below.

\begin{figure}[ht]

    \centering
    \includegraphics[width=\textwidth]{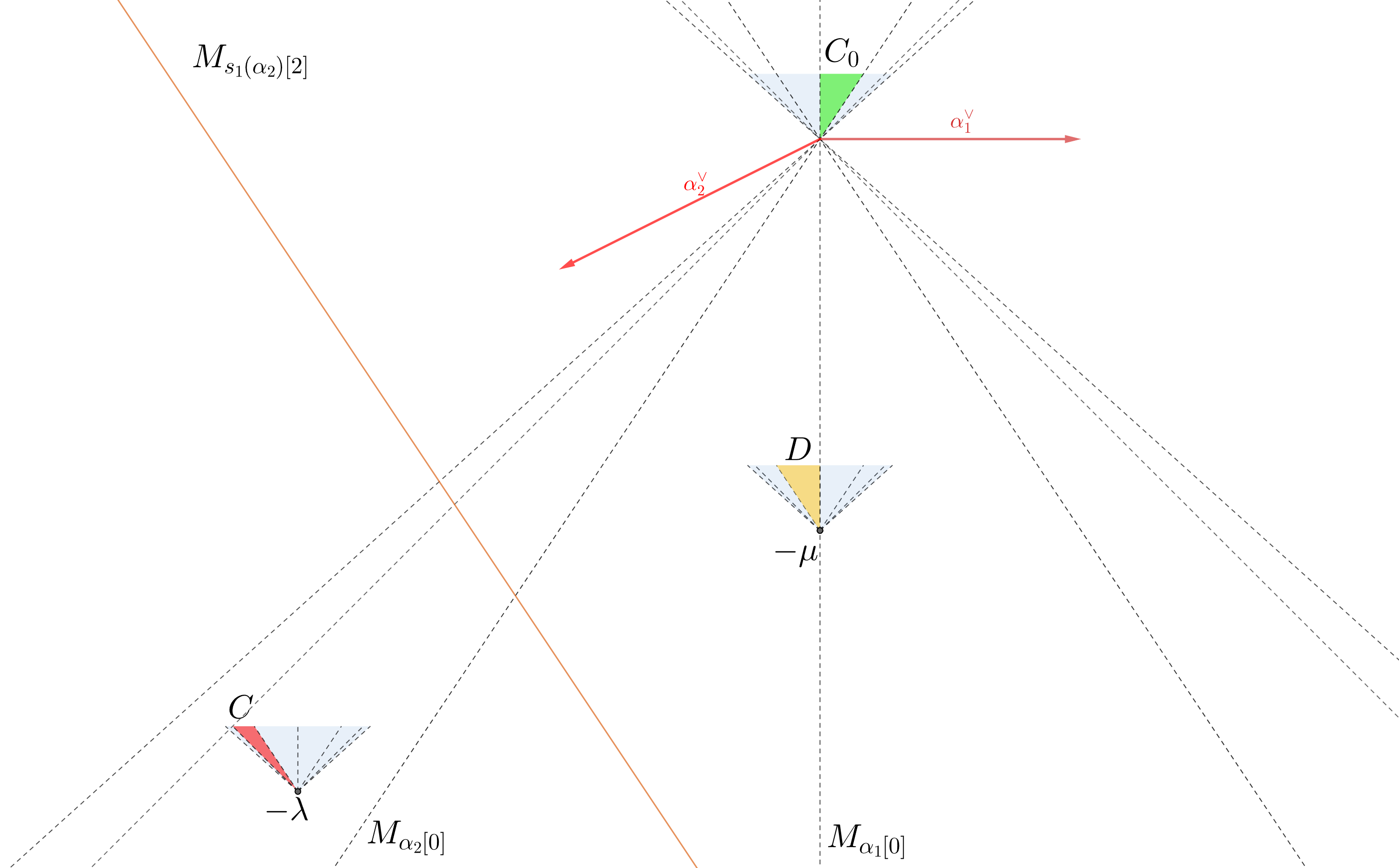}
    \caption{The affine apartment for a root system of Cartan Matrix $\begin{pmatrix}
        2 & -3 \\ -2 & 2
    \end{pmatrix}$.}
    \label{fig: figure2}
\end{figure}

\newpage

In Figure~\ref{fig: figure2} we represent the affine apartment for the same root datum as in Figure~\ref{fig: figure1}. The blue polygons represent the local Tits cones at three different points: the origin, $-\lambda\in -Y^+$ and $-\mu$, which is the image of $-\lambda$ by the reflection along the wall $M_{s_1(\alpha_2)[2]}$ (represented in green).

We have highlighted three alcoves: In green the alcove $C_0$; in red the alcove $C=C_{\pi^\lambda s_1 s_2}$ and in yellow $D=C_{\pi^{\mu} s_1}$ which is the image of $C$ by $s_{s_1(\alpha_2)[2]}$. We see that $D$ is on the same side of $M_{s_1(\alpha_2)[2]}$ as the fundamental alcove $C_{0}$, thus $\pi^\lambda s_1s_2=s_{s_1(\alpha_2)[2]}(\pi^{\mu} s_1)>\pi^{\mu} s_1$.

Note that $-\lambda$ lies in the negative vectorial chamber $-s_2 C^v_f$, that is to say that $s_2 \lambda$ is dominant. Therefore $\pi^\lambda s_2$ is the minimal length element of $\pi^\lambda W$. We will make this more explicit in Section~\ref{subsection 2.2}.

\subsubsection{Notation for segments}

For any two elements $x,y\in V=Y\otimes_\Z \R$, we introduce the following notation: 
\begin{align*}
[x,y]&=\{tx+(1-t)y\mid 0\leq t \leq 1 \} \\
]x,y[&=\{tx+(1-t)y\mid 0< t < 1 \}.\end{align*} Note in particular that, if $x\in Y$ and $y=x+n\beta^\vee$ for $n\in \Z$ and $\beta \in \Phi$, then for any $m \in \llbracket0,n\rrbracket$ we have $x+m\beta^\vee \in [x,y]\cap Y$.

 \subsection{Preliminary results} \label{subsection: prelimary results}

Since the affinized Bruhat order is generated on $W^a_+$ by the relations
$s_{\beta[n]} \bx>\bx \iff \ell^a(s_{\beta[n]}\bx)>\ell^a(\bx)$ for affinized roots $\beta[n]\in \Phi_+^a$, covers are always of this form. In the rest of the paper, we always apply affinized reflections on the left.
 \begin{Lemma}\label{Lemma : coweightconvexity}
     Let $\pi^{\lambda}w \in W^a$ and $\beta[n]\in \Phi_+^a$. Write $\pi^\mu w'$ for $s_{\beta[n]}\pi^\lambda w$ and suppose that  $(\pi^\lambda w)^{-1}(\beta[n])\in \Phi_+^a$. Then $\lambda \in [\mu,s_\beta \mu]$. In particular
     \begin{equation}
         \mu \in Y^+ \implies \lambda \in Y^+.
     \end{equation}
 \end{Lemma}
 \begin{proof}
     Explicitly: $$\pi^\mu w'= \pi^{n\beta^\vee}s_\beta.\pi^\lambda w=\pi^{s_\beta \lambda +n\beta^\vee}s_\beta w.$$
     Thus $\mu = s_\beta \lambda + n\beta^\vee=\lambda+(n-\langle \lambda,\beta\rangle)\beta^\vee$ and $s_\beta \mu = \lambda - n\beta^\vee$.

     Moreover, since $(\pi^\lambda w)^{-1}(\beta[n]) \in \Phi_+^a$, by Formula~\eqref{eq: leftaction}: $$|n|-\sgn(n)\langle \lambda,\beta\rangle = \sgn(n)(n-\langle \lambda,\beta\rangle) \geq 0.$$

     Therefore, unless $n-\langle \lambda,\beta\rangle=0$, $n$ and $n-\langle \lambda,\beta\rangle$ have same sign and thus, $\lambda = s_{\beta}\mu+n\beta^\vee=\mu-(n-\langle \lambda,\beta\rangle)\beta^\vee$ lies in $[s_{\beta}\mu,\mu]$. If $n-\langle \lambda,\beta\rangle=0$ then $\mu=\lambda$ and the result remains true.

    The Tits cone $\mathcal T$ is convex (\cite[Proposition 1.4.2c)]{kumar2002kac}) and $W$-stable, so if $\mu\in \mathcal T$, then $[\mu,s_\beta \mu]$ is contained in $\mathcal T$ for any $\beta\in\Phi$. Therefore in the situation above, if $\mu\in Y^+=\mathcal T\cap Y$, then $\lambda \in [\mu,s_\beta\mu]\cap Y \subset \mathcal T\cap Y=Y^+$, and thus $\mu \in Y^+ \implies \lambda \in Y^+$.         
 \end{proof}
 We directly obtain from Lemma~\ref{Lemma : coweightconvexity} the following result.
 \begin{Proposition}
     The affinized Bruhat order defined on $W^a_+$ coincides with the restriction of the preorder defined through~\eqref{eq : Bruhat order} on the whole semidirect product $W^a$.
 \end{Proposition}

\subsubsection{Properties of the height function}

 We give here a few elementary results on the height function, which will be useful in our study of the affinized Bruhat length. They are also used in \cite[Section 3]{muthiah2019bruhat}.
\begin{Proposition} \label{SubLemma: height}
    For any $w\in W$, \begin{equation}\rho-w^{-1}(\rho)=\sum\limits_{\gamma\in \Inv(w)}\gamma.\end{equation}
\end{Proposition}
\begin{proof}
    This is \cite[1.3.22 Corollary 3]{kumar2002kac}, we prove it by induction on the length of $w$.
    \begin{enumerate}
    \item If $w$ is a simple reflection $s_\alpha$ then $\Inv(s_\alpha)=\{\alpha\}$ and $\rho-s_\alpha(\rho)=\langle \alpha^\vee,\rho\rangle\alpha=\alpha$ since $\langle \alpha^\vee,\rho\rangle=1$ by definition of $\rho$.
    \item Suppose the result true for elements of length $n$, and suppose that $\ell(w)=n+1$ then write $w=w_1s_\alpha$ for $\alpha$ a simple root and $w_1$ an element of length $n$. Then $$\rho-w(\rho)=\rho-w_1(\rho)+w_1(\rho-s_\alpha(\rho))=\sum\limits_{\gamma\in \Inv(w_1^{-1})}\gamma +w_1(\alpha)$$ and since $\Inv(w^{-1})=\Inv(w_1^{-1})\sqcup\{w_1(\alpha)\}$ we get the result for $w$.
    
    \end{enumerate}
\end{proof}
 \begin{Corollary}\label{Lemma : height as sum}
    
    For any positive root $\beta\in \Phi_+$ we have
    \begin{equation} 2\htt(\beta^\vee)=\sum\limits_{\gamma \in \Inv(s_\beta)} \langle \beta^\vee,\gamma\rangle.\end{equation}Moreover all the terms in the sum are positive.
 \end{Corollary}
 \begin{proof}
     Let $\beta\in \Phi_+$ be a positive root. Note that $-s_\beta(\beta^\vee)=\beta^\vee$ and thus $\langle \beta^\vee,\rho\rangle=\langle -s_\beta(\beta^\vee),\rho\rangle=\langle \beta^\vee,-s_\beta(\rho)\rangle$. Therefore by Proposition~\ref{SubLemma: height},
    \begin{equation*}
        2\htt(\beta^\vee)=2\langle \beta^\vee,\rho\rangle=\langle \beta^\vee,\rho-s_\beta(\rho)\rangle=\sum\limits_{\gamma\in \Inv(s_\beta)}\langle \beta^\vee,\gamma\rangle.
    \end{equation*} Moreover for any $\gamma\in\Inv(s_\beta)$, by definition $\gamma\in\Phi_+$ and $s_\beta(\gamma)=\gamma-\langle\beta^\vee,\gamma\rangle\beta^\vee\in\Phi_-$ so, since $\beta$ is a positive root, the coefficient $\langle\beta^\vee,\gamma\rangle$ is necessarily positive. 
\end{proof}

\begin{Corollary} \label{Lemma : height of general coweight}
    Let $\mu \in Y^+$ and $u\in W$ be such that $\mu=u(\mu^{++})$. Then  \begin{equation} \htt(\mu^{++})=\htt(\mu)-\sum\limits_{\tau \in \Inv(u^{-1})}\langle \mu,\tau\rangle. \end{equation}
    Moreover the terms in this sum are non-positive and
    \begin{equation}\label{eq : heightdecreasing}
        \htt(\mu)\leq \htt(\mu^{++}).
    \end{equation}The inequality is strict unless $\mu$ is dominant.
\end{Corollary}
    \begin{proof}
        By definition $\htt(\mu^{++})=\langle u^{-1}(\mu),\rho\rangle = \langle \mu,u(\rho)\rangle$, and by Proposition~\ref{SubLemma: height}:
        \begin{align*}
            \htt(\mu^{++}) =\langle \mu,u(\rho)\rangle =\langle \mu,\rho-\sum\limits_{\tau\in\Inv(u^{-1})}\tau\rangle  = \htt(\mu)-\sum\limits_{\tau\in\Inv(u^{-1})}\langle \mu,\tau\rangle. 
        \end{align*}
        Moreover, for any $\tau \in \Phi$, we have $\langle \mu,\tau\rangle = \langle \mu^{++},u^{-1}(\tau)\rangle$, so $\tau \in \Inv(u^{-1})\implies \langle \mu,\tau\rangle \leq 0$ and the terms of the above sum are all non-positive, we deduce Formula~\eqref{eq : heightdecreasing}. If $\mu$ is not dominant, then there exists $\tau\in \Phi_+$ such that $\langle \mu,\tau\rangle <0$, and thus
        $$\htt(\mu)<\htt(s_\tau \mu)\leq \htt(\mu^{++}).$$
    \end{proof}

Amongst other things, Corollary~\ref{Lemma : height of general coweight} directly implies the following result, which was first indicated to the author by Hébert and Muthiah.
\begin{Lemma}\label{Lemma : heightdecreasing}
    Let $\lambda \in Y^+$ and $\beta\in \Phi_+$ such that $s_\beta \lambda\neq \lambda$. Suppose that $\mu \in ]\lambda,s_\beta\lambda[$. Then $\htt(\mu^{++})<\htt(\lambda^{++})$.
\end{Lemma}
\begin{proof}Note that we do not suppose $\mu \in Y$, the height function is extended to $V=Y\otimes_\mathbb Z \mathbb R$ linearly.
    Let $t\in ]0,1[$ be such that $\mu=t\lambda+(1-t)s_\beta\lambda$, let $v\in W$ be such that $\mu^{++}=v\mu$. Then $\htt(\mu^{++})=\htt(v\mu)=t\htt(v\lambda)+(1-t)\htt(vs_\beta\lambda)$. By Corollary~\ref{Lemma : height of general coweight}, $\htt(v\lambda)\leq \htt(\lambda^{++})$ and $\htt(vs_\beta \lambda)\leq \htt(\lambda^{++})$ and, since $s_\beta\lambda\neq \lambda$, at least one of the two inequality is strict. We deduce $\htt(\mu^{++})<\htt(\lambda^{++})$.
\end{proof}

\begin{Proposition}\label{Proposition : reflectionsamedominance}
    Let $\bx \in W^a_+$ and $\beta[n]\in \Phi_+^a$ such that $s_{\beta[n]}\bx\in W^a_+$. Then
    \begin{equation}\pr^{Y^{++}}(s_{\beta[n]}\bx) =\pr^{Y^{++}}(\bx) \iff n \in \{0, \langle \pr^{Y^+}(\bx),\beta\rangle \}. \end{equation}
\end{Proposition}
\begin{proof}
    To simplify notation, let $\lambda\in Y^+$ denote $\pr^{Y^+}(\bx)$. If $n\in \{0, \langle \lambda,\beta\rangle \}$ then by Formula~\eqref{eq : productdominanceclass}, $\pr^{Y^+}(s_{\beta[n]}\bx) \in \{s_\beta(\lambda),\lambda\}$ and therefore it has same dominance class. 
    
    Conversely, if $n \in \rrbracket 0, \langle \lambda,\beta\rangle \llbracket$ then $\pr^{Y^+}(s_{\beta[n]}\bx)=s_\beta(\lambda)+n\beta^\vee\in ]\lambda,s_\beta(\lambda)[$, and if $n\notin \llbracket 0,\langle \lambda,\beta\rangle\rrbracket$ then $\lambda\in ]s_\beta(\lambda)+n\beta^\vee,\lambda-n\beta^\vee[=]\pr^{Y^+}(s_{\beta[n]}\bx),s_\beta(\pr^{Y^+}(s_{\beta[n]}\bx))[$. 
    
    Either way by Lemma~\ref{Lemma : heightdecreasing}, $\htt(\pr^{Y^{++}}(s_{\beta[n]}\bx))\neq \htt(\pr^{Y^{++}}(\bx))$ and in particular $\pr^{Y^{++}}(s_{\beta[n]}\bx)\neq \pr^{Y^{++}}(\bx)$.
\end{proof}
\begin{Remark}\label{Remark: vectorial reflections}
    Note that if $n=\langle \lambda,\beta\rangle$, then by Formula~\eqref{eq : productdominanceclass}, $s_{\beta[n]}\pi^\lambda w = \pi^\lambda w s_{w^{-1}(\beta)}$. Therefore Proposition~\ref{Proposition : reflectionsamedominance} indicates that, if $\by=s_{\beta[n]}\bx$, then $\pr^{Y^{++}}(\by)=\pr^{Y^{++}}(\bx)$ if and only if $\by$ is obtained from $\bx$ by applying a vectorial reflection either on the left hand side or on the right hand side. This justifies the terminology for vectorial covers and properly affine covers.
\end{Remark}
\begin{Proposition}\label{prop: heightdecreasing}
Let $\bx,\by\in W^a_+$ and suppose that $\bx\leq \by$. Then
\begin{equation}\htt(\pr^{Y^{++}}(\bx))\leq \htt(\pr^{Y^{++}}(\by)),\end{equation}
with equality if and only if $\pr^{Y^{++}}(\bx)=\pr^{Y^{++}}(\by)$. 

In particular the function $\htt\circ \pr^{Y^{++}}: W^a_+\rightarrow \mathbb Z$ is order-preserving.
\end{Proposition}
\begin{proof}
    It is enough to prove it for cover relations, if $\by=s_{\beta[n]}\bx$ for some $\beta[n]\in \Phi_+^a$. In that case, by Lemma~\ref{Lemma : coweightconvexity} we have $\pr^{Y^+}(\bx)\in [\pr^{Y^+}(\by),s_\beta(\pr^{Y^+}(\by))]$. If $\pr^{Y^+}(\bx)\in \{\pr^{Y^+}(\by),s_\beta\pr^{Y^+}(\by)\}$ then they have the same dominance class: $\pr^{Y^{++}}(\bx)=\pr^{Y^{++}}(\by)$ and we obtain the equality case.
    
    Otherwise, $\pr^{Y^+}(\bx)\in ]\pr^{Y^+}(\by),s_\beta(\pr^{Y^+}(\by))[$, necessarily $s_\beta(\pr^{Y^+}(\by))\neq \pr^{Y^+}(\by)$ and by Lemma~\ref{Lemma : heightdecreasing} we deduce $\htt(\pr^{Y^{++}}(\bx))<\htt(\pr^{Y^{++}}(\by))$.
\end{proof}
\begin{Corollary}\label{Corollary : convexity of dominance class}
    For any $\lambda^{++}\in Y^{++}$, the set $\{\bx\in W^a_+\mid \pr^{Y^{++}}(\bx)=\lambda^{++}\}$ is convex for the affinized Bruhat order.
\end{Corollary}

\begin{proof}
    By Proposition~\ref{prop: heightdecreasing} the function $W^a_+\rightarrow \mathbb Z:\, \bx\mapsto \htt\circ \pr^{Y^{++}}$ is compatible with the affinized Bruhat order. Suppose that $\bx,\by\in W^a_+$ satisfy $\pr^{Y^++}(\bx)=\pr^{Y^{++}}(\by)$ and $\bx\leq\by$. Let $\bz\in W^a_+$ be such that $\bx\leq\bz\leq\by$. Then by Proposition~\ref{prop: heightdecreasing}, $\htt(\pr^{Y^{++}}(\bx))\leq\htt(\pr^{Y^{++}}(\bz))\leq\htt(\pr^{Y^{++}}(\by))=\htt(\pr^{Y^{++}}(\bx))$. By the equality case in Proposition~\ref{prop: heightdecreasing}, we deduce $\pr^{Y^{++}}(\bz)=\pr^{Y^{++}}(\bx)$. 
\end{proof}
\begin{Remark}
    Note that, for $\lambda\in Y^+$, the set $\{\bx\in W^a_+\mid \pr^{Y^{++}}(\bx)=\lambda^{++}\}$ is the double $W$-orbit of $\pi^{\lambda}$: \begin{equation}\{\bx\in W^a_+\mid \pr^{Y^{++}}(\bx)=\lambda^{++}\}=W\pi^\lambda W.\end{equation}

    We show in Section~\ref{Section: constant dominance class} that the right $W$-orbits $\pi^\lambda W$ are also convex for the affinized Bruhat order.  
\end{Remark}
We end this section with several metric properties of Coxeter groups, the results stated are proved in the context of Coxeter complexes and buildings in \cite{ronan1989lectures}.
\subsubsection{Metric properties of Coxeter groups} On any Coxeter group $(W_0,S_0)$ we define a map $d \;: \; W_0\times W_0 \rightarrow W_0$ by $d(v,w)=v^{-1}w$, called the \textbf{vectorial distance} of $W_0$, it is $W_0$-invariant: $d(uv,uw)=d(v,w)$ for any $u,v,w \in W_0$. We also define $d^\mathbb N = \ell \circ d$ where $\ell$ is the Bruhat length on $(W_0,S_0)$ (note that $\ell$ and $d^\mathbb N$ depend on the set of simple reflections $S_0$, but the vectorial distance does not). These maps have properties analogous to the standard distance axioms, which justify the name (see \cite[Chapter 3 \S 1]{ronan1989lectures}).

An unfolded gallery (resp. a gallery) in $W_0$ from $w$ to $v$ is a sequence $w=w_0,\dots,w_n=v$ such that $d^{\mathbb N}(w_i,w_{i+1})=1$ (resp. $d^{\mathbb N}(w_i,w_{i+1})\in \{0,1\}$) for all $i\in \llbracket0,n-1\rrbracket$. A gallery is said to be \text{minimal} if its length $n$ is equal to $d^\mathbb N(w_1,w_n)$, a minimal gallery is necessarily unfolded. We refer to \cite[Chapter 2]{ronan1989lectures} for properties of minimal galleries, but note that if $(w_0,\dots,w_n)$ is a minimal gallery then $d^\mathbb N(w_0,w_i)=i$ and thus $(w_0,\dots,w_i)$ is a minimal gallery from $w_0$ to $w_i$. Moreover since the distance is $W_0$ invariant, $(vw_0,\dots vw_n)$ is also a minimal gallery for any $v\in W_0$.
The next Lemma is a reformulation of \cite[Proposition 2.8]{ronan1989lectures}.
\begin{Lemma}\label{Lemma : non_minimal_distance}
Let $(W_0,S_0)$ be a Coxeter system and let $v_1,v_2,w \in W_0$ be such that $v_2$ is not on a minimal gallery from $v_1$ to $w$. Then there is a reflection $r\in W_0$ such that $d(v_1,rw)>d(v_1,w)$ and $d(v_2,rw)<d(v_2,w)$.\end{Lemma}

\begin{proof}

    If $v_2$ is not on a minimal gallery from $v_1$ to $w$, by \cite[Proposition 2.8]{ronan1989lectures} there exists a root $\alpha$ (seen as a half-apartment: $\alpha=\{u\in W_0 \mid \ell(u)<\ell(s_\alpha u)\}$) such that $v_1,w \in \alpha$ and $v_2\notin \alpha$. Then consider the folding along $\alpha$, defined by:
    \begin{equation*}
  \forall u\in W_0,\;\rho_\alpha (u) =
  \begin{cases}
    s_\alpha u & \text{if $u \notin \alpha$} \\
    u & \text{otherwise}.
  \end{cases}
\end{equation*}
It reduces the vectorial distance (see \cite[\S 2]{ronan1989lectures}) hence
\begin{align*}d(v_1,w)&=d(\rho_\alpha(v_1),\rho_\alpha(s_\alpha w))<d(v_1,s_\alpha w), \\ 
d(v_2,s_\alpha w)&=d(s_\alpha v_2,w)=d(\rho_\alpha(v_2),\rho_\alpha(w))<d(v_2,w).\end{align*}
\end{proof}
Recall that for $J\subset S$, $W_J$ is the subgroup generated by the set of simple reflections $J$. The Coxeter system $(W_J,J)$ is an example of Coxeter system for which we will use Lemma~\ref{Lemma : non_minimal_distance}. For any $w\in W$, the coset $wW_J$ is convex, in the sense that, if $w_1,w_2\in wW_J$ then any minimal gallery from $w_1$ to $w_2$ lies in $wW_J$ (see \cite[Lemma 2.10]{ronan1989lectures}).
\begin{Definition}\label{def: projparabolic}
    For any $J\subset S$ and $v,w\in W$, the \textbf{projection of $w$ on $vW_J$} is the unique element of $vW_J$ which reaches $\min\limits_{\tilde v\in vW_J}d^\mathbb N(w,\tilde v)$. It is denoted $\pr_{vW_J}(w)$. Moreover any minimal gallery from $v$ to an element of $wW_J$ goes through $\pr_{vW_J}(w)$ (see \cite[Theorem 2.10]{ronan1989lectures}).
\end{Definition}

\section{Restriction to constant dominance classes}\label{Section: constant dominance class}

In this section, we study the affinized Bruhat order restricted to a dominance class, that is to say, for a given $\lambda^{++} \in Y^{++}$, we study the restriction of the affinized Bruhat order to the subset $(\pr^{Y^{++}})^{-1}(\lambda^{++})=W\pi^{\lambda^{++}}W$. By Corollary~\ref{Corollary : convexity of dominance class} these are convex subsets for the affinized Bruhat order. We start by showing that, for any $\lambda \in Y^+$, the subset $\pi^\lambda W=(\pr^{Y^+})^{-1}(\lambda)$ of $(\pr^{Y^{++}})^{-1}(\lambda^{++})$ is also convex for the affinized Bruhat order.

\begin{Lemma}\label{Lemma : Minimal inversion set}
    Let $\lambda \in Y^+$, recall Notation~\ref{Notation : paraboliccosets} for $v^\lambda$. Then $\Inv((v^\lambda)^{-1})\cap \Phi_\lambda =\emptyset$. In particular for any $\beta\in \Phi_+$, \begin{equation}\label{eq : minimal inversion set}\htt(\lambda)<\htt(s_\beta \lambda)\iff \langle\lambda,\beta\rangle<0\iff s_\beta v^\lambda<v^\lambda.\end{equation}
\end{Lemma}
\begin{proof}
    
Let $\lambda \in Y^+$ and $\alpha \in \Inv((v^\lambda)^{-1}) \cap \Phi_\lambda$, then since $\alpha \in \Phi_\lambda$, $s_\alpha$ fixes $\lambda$, that is $s_\alpha \in W_\lambda$. Moreover $(v^\lambda)^{-1}(\alpha)<0$ so $s_\alpha v^\lambda < v^\lambda$, this contradicts the minimality of $v^\lambda$ (note that, as $W_\lambda v^\lambda =v^\lambda W_{\lambda^{++}}$, $v^\lambda$ is also the minimal representative for the right coset $W_\lambda v^\lambda$). Hence $\Inv((v^\lambda)^{-1})\cap\Phi_\lambda = \emptyset$, therefore any $\beta\in \Inv((v^\lambda)^{-1})$ verifies $\langle \lambda,\beta\rangle \neq 0$.

For $\beta\in \Phi_+$, $\langle \lambda,\beta\rangle =\langle \lambda^{++},(v^{\lambda})^{-1}(\beta)\rangle$. Since $\lambda^{++}$ is dominant, if this is negative then $\beta \in \Inv((v^{\lambda})^{-1})$, and since $\Inv((v^{\lambda})^{-1})$ and $\Phi_\lambda$ are disjoint, the converse is also true. Since $\beta\in \Inv((v^\lambda)^{-1})\iff s_\beta v^\lambda <v^\lambda$ we deduce the second equivalence in Formula~\eqref{eq : minimal inversion set}. Moreover $\htt(s_\beta \lambda)=\htt(\lambda)-\langle \lambda,\beta\rangle \htt(\beta^\vee)$ by linearity of the height function, and since $\htt(\beta^\vee)>0$, the first equivalence in Formula~\eqref{eq : minimal inversion set} is clear.
\end{proof}
\begin{Remark}
    The fact that $\Inv((v^\lambda)^{-1}) \cap \Phi_\lambda = \emptyset$ is visible geometrically in the Coxeter complex of $W$, in which $\Phi_\lambda$ is the set of walls containing $\lambda$, $\Inv(v^{-1})$ is the set of walls separating $C^v_f$ and $C^v_v$. The chamber $C^v_{v^\lambda}$ is the closest chamber from the fundamental chamber amongst the chambers containing $\lambda$ in their closure, in other words $v^\lambda = \pr_{W_\lambda}(1_{W})$. 
\end{Remark}

\begin{Proposition}\label{affine_to_relative0}

    Suppose that $\pi^\lambda w\in W^a_+$ and $r\in W$ is a reflection such that $r\lambda \neq \lambda$. Then
    \begin{equation}\label{eq : affinetorelative0}
        \pi^{r\lambda}rw>\pi^\lambda w \iff rv^\lambda<v^\lambda.
    \end{equation}
Moreover for any $\lambda^{++}\in Y^{++}$, the restriction of the function $\htt \circ \pr^{Y^+}$ to $(\pr^{Y^{++}})^{-1}(\lambda^{++})$ is order-preserving.
\end{Proposition}
\begin{proof}
 
Suppose that $r\in W$ is a reflection which does not fix $\lambda$. By definition there exists a positive root $\beta\in\Phi_+$ such that $r=s_\beta$ and, since $r$ does not fix $\lambda$, $\langle \lambda,\beta\rangle \neq0$. Note that $\pi^{r\lambda}rw=s_{\beta[0]}\pi^\lambda w$ so, using Formula~\eqref{eq: leftaction}, we have
    $$\pi^{r\lambda}rw>\pi^\lambda w\iff -\langle \lambda ,\beta\rangle > 0\iff \langle \lambda,\beta\rangle <0.$$
       By Lemma~\ref{Lemma : Minimal inversion set} this is equivalent to $rv^\lambda<v^\lambda$, and to $\htt(\lambda)<\htt(r\lambda)$. This is enough to obtain Equivalence~\eqref{eq : affinetorelative0}. Moreover by convexity of $(\pr^{Y^{++}})^{-1}(\lambda^{++})$ (see Corollary~\ref{Corollary : convexity of dominance class}) and by Proposition~\ref{Proposition : reflectionsamedominance} it also implies that $\htt\circ \pr^{Y^{+}}\,:\, (\pr^{Y^{++}})^{-1}(\lambda^{++})\rightarrow \mathbb Z$ is order-preserving.

\end{proof}
Note that the function $\htt\circ \pr^{Y^{+}}$ is not order-preserving on the whole semi-group $W^a_+$. For example suppose that $\lambda \in Y^{++}$ and $\beta\in \Phi_+$ are such that $\lambda+\beta^\vee$ is also dominant. Then we can check that $\pi^{s_\beta(\lambda)}<\pi^{s_\beta(\lambda)-\beta^\vee}s_\beta$ whereas $\htt(s_\beta(\lambda)-\beta^\vee)<\htt(s_\beta(\lambda))$.

Proposition~\ref{affine_to_relative0} implies convexity of left $W$-cosets:
\begin{Corollary}\label{Corollary : convexity of coweight}
    Let $\lambda\in Y^+$, then the set $\pi^\lambda W=\{\bx \in W^a_+\mid \pr^{Y^+}(\bx)=\lambda\}$ is convex for the affinized Bruhat order.
\end{Corollary}
\begin{proof}
    Let $\bx,\by\in \pi^\lambda W$ such that $\bx<\by$ and let $\bx=\bx_0<\bx_1<\dots \bx_n=\by$ be a chain from $\bx$ to $\by$, in particular for all $k\in \llbracket0,n-1\rrbracket$, let $\beta_k[n_k]\in \Phi_+^a$ be such that $\bx_{k+1}=s_{\beta_k[n_k]}\bx_k$. For $k\in \llbracket0,n\rrbracket$, write $\bx_k=\pi^{\lambda_k}w_k$ with $\lambda_k\in Y^+$, $w_k\in W$. By convexity of $(\pr^{Y^{++}})^{-1}(\lambda^{++})$, $\pr^{Y^{++}}$ is constant along the chain, therefore by Proposition~\ref{Proposition : reflectionsamedominance}, for all $k\in \llbracket0,n-1\rrbracket$ we have $\lambda_{k+1}\in \{\lambda_k,s_{\beta_k}(\lambda_k)\}$. From Proposition~\ref{affine_to_relative0} we deduce that $v^{\lambda_{k+1}}\leq v^{\lambda_k}$. Since $\lambda_0=\lambda_n=\lambda$, $v^{\lambda_0}=v^{\lambda_n}=v^\lambda$ and thus $v^{\lambda_k}=v^\lambda$, so $\lambda_k=\lambda$ for all $k\in \llbracket0,n\rrbracket$. Hence $\pi^\lambda W$ is convex.
\end{proof}
\subsection{Relative length on $W$}\label{subsection: 2.1}
We define a relative length and a relative Bruhat order on $W$, which naturally arises in the study of the affinized length $\ell^a$ on $W^a_+$. This connection was already observed by Muthiah and Orr in \cite{muthiah2018walk}.

\begin{Definition} 
For any $v,w \in W$ let: \begin{equation}\ell_v(w)=|\Inv(w^{-1})\setminus \Inv(v^{-1})|-|\Inv(w^{-1})\cap \Inv(v^{-1})|.\end{equation} This is a signed version of the Bruhat length, in particular $\ell_1=\ell$. 

We associate an order to $\ell_v$ by setting, for any element $w\in W$ and any reflection $r\in W$: $w<_v wr$ if and only if $\ell_v(w)<\ell_v(wr)$ and then let $<_v$ be the order generated by these relations, it is strictly compatible with $\ell_v$. In particular $<_1$ is the classical Bruhat order.
\end{Definition}
As does the Bruhat length, the lengths $\ell_v$ have a geometric interpretation in the Coxeter complex associated to $(W,S)$. For $M$ a  wall of the Coxeter complex and $w\in W$, let $\varepsilon_w(M)=-1$ if $M$ separates $C^v_f$ and $C^v_w$, and $\varepsilon_w(M)=+1$ otherwise. Then \begin{equation}\ell_v(w)=\sum\limits_{M\in \varepsilon_{w}^{-1}(-1)} \varepsilon_{v^{-1}}(M).\end{equation}

We will use this relative length to give an alternative definition of the affinized length. Let us first give an explicit formula for $\ell_v$ depending only on the classical length $\ell=\ell_1$.

\begin{Lemma}\label{length_recursion}
If $sv>v$ with $v \in W$ and $s$ a simple reflection then for any $w\in W$, $\ell_{sv}(w)=\ell_v(sw)-1$.
\end{Lemma}

\begin{proof}
For any $w\in W$, the map $\gamma \mapsto s\gamma$ defines a bijection: \begin{equation*}\Inv(w^{-1})\setminus \{\alpha_s\} \cong \Inv(w^{-1}s) \setminus \{\alpha_s\}.\end{equation*}
Moreover because $sv>v$, $\alpha_s \in \Inv(v^{-1}s)$ and $\alpha_s \notin \Inv(v^{-1})$.

Therefore \begin{equation*}|\Inv(w^{-1})\cap \Inv(v^{-1}s) \setminus \{\alpha_s\}|=|\Inv(w^{-1}s) \cap \Inv(v^{-1})|\end{equation*} and \begin{equation*}|\Inv(w^{-1}) \setminus \Inv(v^{-1}s)|=|\Inv(w^{-1}s)\setminus (\Inv(v^{-1}) \cup \{\alpha_s\})|.\end{equation*}

\begin{enumerate}
    \item If $\alpha_s \in \Inv(w^{-1})$ then $\alpha_s \notin \Inv(w^{-1}s)$ and $\ell_{sv}(w)=|\Inv(w^{-1}s) \setminus \Inv(v^{-1})|-(|\Inv(w^{-1}s)\cap\Inv(v^{-1})|+1)=\ell_v(sw)-1$.
    \item If $\alpha_s \notin \Inv(w^{-1})$ then $\alpha_s \in \Inv(w^{-1}s)$ and $\ell_{sv}(w)=(|\Inv(w^{-1}s) \setminus \Inv(v^{-1})|-1)-|\Inv(w^{-1}s)\cap\Inv(v^{-1})|=\ell_v(sw)-1$.
\end{enumerate}
\end{proof}
\begin{Proposition}\label{lengthexplicit}
For all $v,w \in W$ the relative length $\ell_v(w)$ is given by
\begin{equation}
    \ell_v(w)=\ell(v^{-1}w)-\ell(v).
\end{equation}
\end{Proposition}

\begin{proof}
Since $\ell=\ell_1$, we take a reduced expression for $v$ and apply Lemma~\ref{length_recursion} recursively to get the result.
\end{proof}
\begin{Corollary}\label{Corollary : relativegrading}
    For any $v\in W$, the relative length $\ell_v$ is a grading of $(W,<_v)$.
\end{Corollary}
\begin{proof}
    Let $v,w,w'\in W$. By proposition~\ref{lengthexplicit}, $\ell_v(w')-\ell_v(w)=\ell(v^{-1}w')-\ell(v^{-1}w)$ and $w'$ covers $w$ for $<_v$ if and only if $v^{-1}w'$ covers $v^{-1}w$ for the (standard) Bruhat order. Since the Bruhat length is a grading of $(W,<_1)$ (see \cite[Theorem 2.2.6]{bjorner2005combinatorics}), $v^{-1}w'$ covers $v^{-1}w$ if and only if $\ell(v^{-1}w')-\ell(v^{-1}w)=1$ and $v^{-1}w'=v^{-1}wr$ for some reflection $r\in W$. Hence $w'$ covers $w$ for $<_v$ if and only if $\ell_v(w')-\ell_v(w)=1$ and $w'=wr$ for some reflection $r\in W$: $\ell_v$ is a grading of $(W,<_v)$.
\end{proof}
The order $<_v$ also has a geometric interpretation which will be important later on, it is given by the following corollary. \begin{Corollary}\label{ordregeom}
For any root $\alpha\in \Phi$ and elements $w,v\in W$, we have that $w<_v s_\alpha w$ if and only if, in the Coxeter complex of $W$, $C^v_w$ and $C^v_v$ are on the same side of the wall $M_\alpha$. 
\end{Corollary}
\begin{proof}
    We have $\ell_v(s_\alpha w)-\ell_v(w)=\ell(v^{-1}s_\alpha w)-\ell(v^{-1}w)=d^\mathbb N(v,s_\alpha w)-d^\mathbb N(v,w)$. By the definition of the Coxeter complex this is positive if and only if $C^v_v$ and $C^v_w$ are on the same side of the wall $M_\alpha$. 
\end{proof}
Therefore $<_v$ can be interpreted as a shift of the classical Bruhat order, corresponding geometrically to taking $C^v_v$ as fundamental chamber in the Coxeter complex.

\subsection{Relation with the affinized Bruhat length}\label{subsection 2.2}

We relate the affinized Bruhat order and the relative order defined in subsection~\ref{subsection: 2.1}, we start by an alternative expression for the affinized Bruhat length.
\begin{Proposition}\label{Prop : affine_length_expression0}
For any coweight $\lambda=v\lambda^{++} \in Y^+$, for any $w \in W$,

\begin{equation}\label{eq : relativelength inversionset}
  |\{\alpha \in \Inv(w^{-1}) \mid \langle \lambda,\alpha \rangle \geq 0 \}|-|\{\alpha \in \Inv(w^{-1}) \mid \langle \lambda,\alpha \rangle < 0 \}|  =\ell_{v^\lambda}(w).
\end{equation}
Therefore
\begin{align}
    \ell_\varepsilon^a(\pi^\lambda w)&=2\htt(\lambda^{++})+\varepsilon \ell_{v^\lambda}(w), \label{eq: affinelengthexpressionepsilon} \\ \ell^a(\pi^\lambda w)&=2\htt(\lambda^{++})+\ell_{v^\lambda}(w).\label{eq: affinelengthexpression}
\end{align}
\end{Proposition}
\begin{proof}
For $\lambda \in Y^+$ and $v \in W$ such that $\lambda = v\lambda^{++}$, then, $\alpha \in \Phi_+$ verifies $\langle \lambda,\alpha \rangle \geq 0$ if and only if $\alpha \in \Phi_\lambda \cup (\Phi_+\setminus \Inv(v^{-1})) $, so:

$$\{\alpha \in \Inv(w^{-1}) \mid \langle \lambda,\alpha \rangle \geq 0 \}=\big(\Inv(w^{-1})\setminus \Inv(v^{-1})\big) \bigsqcup \big(\Inv(w^{-1})\cap \Inv(v^{-1}) \cap \Phi_\lambda\big)$$  $$\Inv(w^{-1})\cap \Inv(v^{-1}) = \{\alpha \in \Inv(w^{-1}) \mid \langle \lambda,\alpha \rangle < 0 \} \bigsqcup \big(\Inv(w^{-1})\cap \Inv(v^{-1}) \cap \Phi_\lambda\big).$$ Therefore
\begin{equation*}
|\{\alpha \in \Inv(w^{-1}) \mid \langle \lambda,\alpha \rangle \geq 0 \}|-|\{\alpha \in \Inv(w^{-1}) \mid \langle \lambda,\alpha \rangle < 0 \}| = \ell_v(w)+2|\Inv(w^{-1})\cap\Inv(v^{-1})\cap\Phi_\lambda|\end{equation*}
By Lemma~\ref{Lemma : Minimal inversion set} we deduce Formula~\eqref{eq : relativelength inversionset}.

\end{proof}
 \begin{Remark}
 By combining Lemma~\ref{Lemma : height of general coweight} with Proposition~\ref{Prop : affine_length_expression0}, we obtain the formulas already given by Muthiah and Orr in \cite[Proposition 3.10]{muthiah2019bruhat}.
 \end{Remark}

\begin{Corollary}\label{affine_to_relative}Let $\lambda \in Y^+$ and $w\in W$. Suppose that $\pi^\mu w' = s_{\beta[n]} \pi^\lambda w$ for some affinized root $\beta[n] \in \Phi^a$ such that $\mu^{++}=\lambda^{++}$. Then
        \begin{equation}\label{eq: affinetorelative1}
        \pi^\mu w' > \pi^\lambda w \iff \ell_{v^\mu}(w')>\ell_{v^\lambda}(w).
    \end{equation}
Moreover for any $\lambda \in Y^+$ and $w,w' \in W$,
    \begin{equation}\label{eq: affinetorelative2}
        \pi^\lambda w < \pi^\lambda w' \iff w <_{v^\lambda} w'.
    \end{equation}
    In particular, $\pi^\lambda v^\lambda$ is the minimal element of $\pi^\lambda W$.

\end{Corollary}

\begin{proof}
 Equivalence~\eqref{eq: affinetorelative1} is a direct consequence of Formula~\eqref{eq: affinelengthexpression} and strict compatibility of the affinized Bruhat length and the affinized Bruhat order. It implies by iteration that a chain for the relative order $<_{v^\lambda}$ from $w$ to $w'$ lifts to a chain for the affinized Bruhat order from $\pi^\lambda w$ to $\pi^\lambda w'$. Conversely, by Corollary~\ref{Corollary : convexity of coweight} $\pr^{Y^+}$ is constant along any chain from $\pi^\lambda w$ to $\pi^\lambda w'$, and therefore the projection on $W$ of a chain from $\pi^\lambda w$ to $\pi^\lambda w'$ induces a chain for the relative Bruhat order $<_v$ from $w$ to $w'$.
\end{proof}

We deduce a partial version of Theorem~\ref{Theorem : maintheoremintro}, for vectorial covers with constant coweight.
\begin{Theorem}\label{Theorem: constantcoweight}
    Let $\bx,\by\in W^a_+$ be such that $\pr^{Y^+}(\bx)=\pr^{Y^+}(\by)$ and $\bx\leq \by$. Then
    \begin{equation}\label{eq: coverlengthdifference0}
        \bx \lhd \by \iff \ell^a(\by)=\ell^a(\bx)+1.
    \end{equation}
    More precisely, if $\bx=\pi^\lambda w$ then $\by=\pi^\lambda rw$ for some reflection $r\in W$ such that $rw$ covers $w$ for the relative Bruhat order $<_{v^\lambda}$.
\end{Theorem}
\begin{proof}
    By Equivalence~\eqref{eq: affinetorelative2}, $\pi^\lambda w'$ covers $\pi^\lambda w$ if and only if $w'$ covers $w$ for the relative Bruhat order $<_{v^\lambda}$. By Corollary~\ref{Corollary : relativegrading}, this is equivalent to $\ell_{v^\lambda}(w')=\ell_{v^\lambda}(w)+1$,. Therefore by Formula~\eqref{eq: affinelengthexpression} we deduce that $\bx \lhd \by \implies \ell^a(\by)=\ell^a(\bx)+1$. The converse is immediate by strict compatibility of the affinized Bruhat length (Theorem~\ref{Theorem Muthiah Orr}).
\end{proof}
\subsection{Vectorial covers with non-constant coweight}\label{subsection: 2.3}
The aim of this section is to prove Theorem~\ref{Theorem : maintheoremintro} for vectorial covers with non-constant coweight.

Beforehand, we need a few results on parabolic decomposition. The first lemma is an adaptation of a standard result on minimal coset representatives (see \cite[Theorem 2.5.5]{bjorner2005combinatorics}), and the second is proved by P-E. Chaput, L. Fresse and T. Gobet in \cite{chaput2021parametrization}.
\begin{Lemma}\label{Lemma : good parabolic chains}
    Let $J$ be a subset of $S$, recall Notation~\ref{Notation : paraboliccosets} for $W^J$. Let $v$ be an element of $W^J$ and $u$ be any element of $W$ such that $u<v$. Then, there is a saturated chain for the Bruhat order:
    \begin{equation}u=u_0\lhd u_1 \lhd \dots \lhd u_N=v\end{equation} such that, for any $i\in \llbracket1,N\rrbracket$, $u_{i-1}^{-1}u_i$ does not belong to $W_J$.
\end{Lemma}
\begin{proof}
    If $v$ covers $u$, it is clear since $u<v$ is a saturated chain, and as $v$ is a minimal coset representative, $u^{-1}v\notin W_J$. By induction it thus suffices, for a general pair $(u,v)$, to construct $u_1\in W$ such that $u_1$ covers $u$, $u^{-1}u_1\notin W_J$ and $u_1<v$, the rest of the chain is obtained by induction.
    Let $s_1\dots s_n$ be a reduced expression of $v$. Since $u<v$, there exists a reduced expression of $u$ obtained from $s_1\dots s_n$ by deleting letters $s_{i_1},\dots,s_{i_N}$. Choose one such that $i_N$ is minimal. Then let $t\in W$ be the reflection defined by $t=s_n\ldots s_{i_N+1}s_{i_N}s_{i_N+1}\ldots s_n$. We show that $u_1=ut$ satisfies the desired properties.
    \begin{enumerate}
        \item By construction, an expression of $ut$ is obtained from $s_1\dots s_n$ by deleting the $N-1$ letters $s_{i_1},\dots,s_{i_{N-1}}$. Therefore $ut<v$.
        \item Since an expression of $vt$ is obtained from $s_1\dots s_n$ by deleting $s_{i_N}$, $vt<v$, and since $v$ is the minimal coset representative of $vW_J$, $t$ does not belong to $W_J$.
        \item It remains to show that $ut$ covers $u$. By the first point, we have that $\ell(ut)\leq \ell(u)+1$, so it suffices to show that $ut\not<u$. Suppose by contradiction that $ut<u$. Then, by the strong exchange property, an expression of $ut$ is obtained from $u$ by deleting one letter $s_p$ of the reduced expression $s_1\ldots \check{s}_{i_1}\ldots \check{s}_{i_N}\ldots s_n$ (where $\check{s}_i$ denotes a letter $s_i$ taken away from the expression $s_1\dots s_n$).
        \begin{enumerate}
            \item Suppose that $p>i_N$. Then $t$ can also be written $s_n\ldots s_{p+1}s_ps_{p+1}\ldots s_n$, and $v=(vt)t=s_1\ldots \check{s_{i_N}}\ldots \check{s_p}\ldots s_n$, which contradicts the hypothesis that $s_1\dots s_n$ is reduced.
            \item Suppose now that there is $d\leq N-1$ such that $i_d<p<i_{d+1}$ (with the convention that $i_0=-1$). Then $t=s_n\ldots \check{s}_{i_N}\ldots \check{s}_{i_{d+1}}\ldots s_p\ldots \check{s}_{i_{d+1}}\ldots \check{s}_{i_N}\ldots s_n$, and $u=(ut)t$ can be written from $s_1\dots s_n$ by deleting the terms of indices $i_1,\dots,i_{N-1}$ and $p<i_N$, but not $i_N$. This contradicts the minimality of $i_N$.
        \end{enumerate}
    \end{enumerate}
    
\end{proof}
\begin{Definition}
    For $v,w\in W$, we write $v\leq_R w$ for:
    \begin{equation}
        v\leq_R w \iff \ell(w)=\ell(v)+\ell(wv^{-1}).
    \end{equation}
    
\end{Definition}
\begin{Remark}
    The relation $\leq_R$ is called the weak Bruhat order and it is related to minimal galleries: $v\leq_R w$ if and only if there is a minimal gallery from $1$ to $w^{-1}$ going through $v^{-1}$.
\end{Remark}
Recall that for $J\subset S$ and $x\in W$, $(x^J,x_J)$ denotes the unique pair of $W^J\times W_J$ such that $x=x^J.x_J$.
\begin{Lemma}\label{Lemma : chaput}\cite[Lemma 8.11]{chaput2021parametrization}
Let $J\subset S$ be a subset of simple reflections. Let $u$ be an element of $W$ and $t$ be a reflection of $W\setminus W_J$ such that $ut$ covers $u$. Then $(ut)_J\leq_R u_J$. In other words $((ut)_J)^{-1}$ lies on a minimal gallery from $1$ to $(u_J)^{-1}$.     
\end{Lemma}

\begin{Theorem}\label{non_reg_covers_same_orbit}
Let $\bx,\by\in W^a_+$ be such that $\pr^{Y^{++}}(\by)=\pr^{Y^{++}}(\bx)$ and $\bx\leq\by$. Then \begin{equation}\label{eq: non_reg_covers_same_orbit}\bx \lhd \by \iff \ell^a(\by)=\ell^a(\bx)+1.\end{equation}

More precisely, write $\bx=\pi^\lambda w$. Let $J$ be the set of simple reflections stabilizing $\lambda^{++}$ and let $v\in W^J$ be such that $\lambda =v\lambda^{++}$ (so $v=v^\lambda$ with Notation~\ref{Notation : paraboliccosets}). Then, if $\pr^{Y^+}(\by)\neq \pr^{Y^+}(\bx)$, there exists a unique reflection $r\in W$ such that:
\begin{enumerate}
    \item The reflection $r$ does not stabilize $\lambda$ and $\by=\pi^{r\lambda}rw$.
    \item For the Bruhat order on $W$, $v$ covers $rv$.
    \item Set $u=rv$, so $u_J\in W_J$ denotes the element $(rv)_J$ and $rvu_J^{-1}\in W^J$. Then $vu_J^{-1}$ is on a minimal gallery from $v$ to $w$ in $W$.

\end{enumerate}
\end{Theorem}

\begin{proof}
If $\pr^{Y^+}(\by)=\pr^{Y^+}(\bx)$ then Equivalence~\ref{eq: non_reg_covers_same_orbit} is given by Theorem~\ref{Theorem: constantcoweight}. Moreover if $\ell^a(\by)=\ell^a(\bx)+1$ and $\by\geq \bx$ then, by strict compatibility of $\ell^a$, $\by$ covers $\bx$. We are thus reduced to prove that, if $\by$ covers $\bx$ with $\pr^{Y^+}(\by)\neq \pr^{Y^+}(\bx)$ and $\pr^{Y^{++}}(\by)=\pr^{Y^{++}}(\bx)$, then $\ell^a(\by)=\ell^a(\bx)+1$.

Write $\bx =\pi^\lambda w$ and $v=v^\lambda$. By definition of the affinized Bruhat order, if $\bx \lhd \by$ then $\by$ is of the form $s_{\beta[n]}\bx$ for some $\beta[n]\in\Phi_+^a$.

Let $\by=s_{\beta[n]}\bx\in W^a_+$ with $\pr^{Y^{++}}(\by)=\lambda^{++}$ and $\pr^{Y^+}(\by)\neq \lambda$, in particular $n\neq \langle \lambda,\beta\rangle$. By Proposition~\ref{Proposition : reflectionsamedominance}, $n=0$ so $\by=\pi^{r\lambda}rw$ for the reflection $r=s_\beta$ which does not stabilize $\lambda$. Let us write $u=rv$, note that $u^J=v^{r\lambda}$. By Formula~\eqref{eq: affinelengthexpression}, \begin{equation}\label{eq: long0}\ell^a(\by)-\ell^a(\bx)=\ell_{u^J}(rw)-\ell_v(w).\end{equation}

By definition, $rv=u^Ju_J$ with, by Equality~\eqref{eq : additive parabolic length}, $\ell(rv)=\ell(u_J)+\ell(u^J)$. We compute
\begin{align}\label{eq:long1} 
    \ell_{u^J}(rw)-\ell_v(w)&= \ell((rv(u_J)^{-1})^{-1}rw)-\ell(v^{-1}w)+ \ell(v)-\ell(u^J) \nonumber\\
    &=\ell(u_J v^{-1}w)-\ell(v^{-1}w)+ \ell(v)+\ell(u_J)-\ell(u) \nonumber\\
    &= (\ell(v)-\ell(u))+(\ell(u_J) - (d^\mathbb N(v,w)-d^\mathbb N(vu_J^{-1},w))).
\end{align}
From Equations~\eqref{eq: long0} and~\eqref{eq:long1}, we deduce
\begin{equation}\label{eq:long}
    \ell^a(\by)-\ell^a(\bx)=(\ell(v)-\ell(u))+(\ell(u_J) - (d^\mathbb N(v,w)-d^\mathbb N(vu_J^{-1},w))).
\end{equation}
In Expression~\eqref{eq:long}, by the triangle inequality and since $d^\mathbb N(v,vu_J^{-1})=\ell(u_J)$, the second term $\ell(u_J) - (d^\mathbb N(v,w)-d^\mathbb N(vu_J^{-1},w))$ 
is non-negative, and it is equal to $0$ if and only if $d^\mathbb N(v,w)=d^\mathbb N(v,vu_J^{-1})+d^\mathbb N(vu_J^{-1},w)$, so if and only if $vu_J^{-1}$ is on a minimal gallery from $v$ to $w$. 

Recall Definition~\ref{def: projparabolic} of $\pr_{vW_J}(w)$. Since $vu_J^{-1}$ lies in $vW_J$, a minimal gallery from $vu_J^{-1}$ to $w$ goes through $\pr_{vW_J}(w)$. Thus $\ell(u_J) - (d^\mathbb N(v,w)-d^\mathbb N(vu_J^{-1},w))$ is equal to zero if and only if $u_J^{-1}$ is on a minimal gallery from $1$ to $v^{-1}\pr_{vW_J}(w)$ in $W_J$.

Let us first suppose that this is not the case, we want to deduce that $\by$ does not cover $\bx$. We thus want to produce a non trivial chain from $\pi^\lambda w$ to $\pi^{r\lambda}rw$. By Lemma~\ref{Lemma : non_minimal_distance}, there is a reflection $t\in W_J$ such that $d^{W_J}(1,tv^{-1}\pr_{vW_J}(w))>d^{W_J}(1,v^{-1}\pr_{vW_J}(w))$ and $d^{W_J}(u_J^{-1},tv^{-1}\pr_{vW_J}(w))<d^{W_J}(u_J^{-1},v^{-1}\pr_{vW_J}(w))$. In $W$, this implies $d(vt,w)>d(v,w)$ and $d(vtu_J^{-1},w)<d(vu_J^{-1},w)$.

Let $\tilde w = vtv^{-1} w$. We compute:
\begin{equation*}
    \begin{split}
        \ell_v(\tilde w)-\ell_v(w)=d^\N(vt,w)-d^\N(v,w)>0 \\
        \ell_{u^J}(rw)-\ell_{u^J}(r\tilde w)=d^\N(vu_J^{-1},w)-d^\N(vtu_J^{-1},w)> 0\\
        \ell_{u^J}(r\tilde w)-\ell_v(\tilde w) = \ell(v)-\ell(u)+\ell(u_J) - (d^\mathbb N(v,\tilde w)-d^\mathbb N(vu,\tilde w))>0.  
    \end{split}
\end{equation*}
Hence by Lemma~\ref{affine_to_relative0} and Corollary~\ref{affine_to_relative}, \begin{equation}\pi^\lambda w < \pi^\lambda \tilde w < \pi^{r\lambda} r\tilde w < \pi^{r\lambda} rw.\end{equation}

Suppose now that $vu_J^{-1}$ is on a minimal gallery from $v$ to $w$, then by Formula~\eqref{eq:long}, $\ell^a(\by)-\ell^a(\bx)=\ell(v)-\ell(rv)$. If this is equals to $1$ then $\by$ covers $\bx$; it remains to show the converse. Suppose that $\ell(v)-\ell(rv)=N>1$. Let \begin{equation}rv=u_0\lhd u_1\lhd \dots\lhd u_N=v\end{equation} be a saturated chain in $W$ from $rv$ to $v$ given by Lemma~\ref{Lemma : good parabolic chains} and, for $i\in \llbracket 1,N\rrbracket$ let $\beta_i \in \Phi_+$ be such that $u_i=s_{\beta_i}u_{i-1}$, so $u_i=s_{\beta_i}\dots s_{\beta_1} u\geq rv$. Note in particular that
\begin{equation}\label{eq: lengthofu_i}
    \ell(u_i)=\ell(rv)+i=\ell(v)-N+i.
\end{equation}
Let us show that it induces a chain for the affinized Bruhat order \begin{equation}\label{eq: lengthuij6}\pi^\lambda w =s_{\beta_N[0]}\dots s_{\beta_1[0]} \pi^{r\lambda}rw<s_{\beta_{N-1}[0]}\dots s_{\beta_1[0]}\pi^{r\lambda}rw<\dots<\pi^{r\lambda}rw.\end{equation}
Since $s_{\beta_i[0]}\dots s_{\beta_1[0]}\pi^{r\lambda}rw=\pi^{u_i\lambda^{++}}s_{\beta_i}\dots s_{\beta_1} rw$, by Formula~\eqref{eq: affinelengthexpression} it is enough to verify \begin{equation}\label{eq: lengthuij5}\forall i \in \llbracket 0,n\rrbracket,\;\ell_{u_i^J}(s_{\beta_i}\dots s_{\beta_1} rw)=\ell_{v}(w)+N-i.\end{equation}
We compute
\begin{align}\label{eq: lengthuij1}
    \ell_{u_i^J}(s_{\beta_i}\dots s_{\beta_1} rw)&=\ell((u_i (u_i)_J^{-1})^{-1}s_{\beta_i}\dots s_{\beta_1}rw)-\ell(u_i^J) \nonumber\\
    &=\ell((u_i)_Jv^{-1}w)-\ell(u_i^J).
\end{align}Since the saturated chain $u_0<u_1<\ldots<u_N$ is obtained from Lemma~\ref{Lemma : good parabolic chains}, $u_i$ covers $u_{i-1}$ such that the reflection $u_{i-1}^{-1}u_i$ does not belong to $W_J$, so by Lemma~\ref{Lemma : chaput}, $(u_i)_J\leq_R (u_{i-1})_J$ and by iteration we have $(u_i)_J\leq_R (u_0)_J=u_J$. Otherwise said, $(u_i)_J^{-1}$ is on a minimal gallery from $1$ to $u_J^{-1}$, therefore $v(u_i)_J^{-1}$ is on a minimal gallery from $v$ to $vu_J^{-1}$, hence on a minimal gallery from $v$ to $w$. We deduce \begin{equation}\label{eq: lengthuij2}\ell((u_i)_Jv^{-1}w)=\ell(v^{-1}w)-\ell((u_i)_J).\end{equation} Combining Formulas~\eqref{eq: lengthuij1} and~\eqref{eq: lengthuij2} we obtain
\begin{equation}\label{eq: lengthuij3}
    \ell_{u_i^J}(s_{\beta_i}\dots s_{\beta_1} rw)=\ell(v^{-1}w)-\ell((u_i)_J)-\ell(u_i^J).
\end{equation}Moreover, \begin{equation}\label{eq: lengthuij4}
    \ell((u_i)_J)+\ell(u_i^J)=\ell(u_i)=\ell(v)-(N-i)
\end{equation} by Formulas~\eqref{eq : additive parabolic length} and~\eqref{eq: lengthofu_i}. Combining Formulas~\eqref{eq: lengthuij3} and~\eqref{eq: lengthuij4}, we deduce Formula~\eqref{eq: lengthuij5}. This concludes the proof.
\end{proof}

\section{Properly affine covers} \label{section: varying dominant coweight}

\subsection{A few properties of properly affine covers}\label{subsection: 3.1}

We now turn to the case of covers $\pi^\lambda w <\pi^\mu w'$ in $W^a_+$ with $\mu^{++}\neq \lambda^{++}$. Such covers are called properly affine covers.

By Formula~\eqref{eq: leftaction}, if $\pi^\mu s_{\beta} w = s_{\beta[n]}\pi^\lambda w >\pi^\lambda w$ with $\beta[n]\in\Phi_+^a$, then $n\in \mathbb Z \setminus \rrbracket0,\langle\lambda,\beta\rangle\llbracket$. Conversely, if $n\in \mathbb Z \setminus \llbracket0,\langle\lambda,\beta\rangle\rrbracket$ then $s_{\beta[n]}\pi^\lambda w >\pi^\lambda w$ however, $s_{\beta[n]}\pi^\lambda w$ may not be in $W^a_+$ as $\lambda+\mathbb Z \beta^\vee \not\subset Y^+$. The limit cases $n\in\{0,\langle\lambda,\beta\rangle\}$ correspond to $\lambda^{++}=\mu^{++}$ dealt with in the previous section. 

We first show that properly affine covers occur only for minimal $n$, in the following sense.
\begin{Proposition}\label{n_is_small} Let $\lambda \in Y^+$ and $w\in W$, let $\beta \in \Phi$ and $n\in \mathbb Z$. Let us denote $\sigma=\sgn(\langle \lambda,\beta\rangle)\in\{1,-1\}$. If $\pi^\mu w'=s_{\beta[n]}\pi^\lambda w \rhd \pi^\lambda w$ is a cover with $\lambda^{++}\neq\mu^{++}$, then $n\in\{-\sigma,\langle \lambda,\beta\rangle+\sigma\}$. 
\end{Proposition}

\begin{proof}
    For any $\nu \in Y^+$ if we identify the Coxeter complex of $W$ with the positive Tits cone $\mathcal T\subset \mathbb A$, $C^v_{v^\nu}$ is the closest vectorial chamber, from the fundamental chamber, containing $\nu$ in its closure. All the elements of $\lambda+\sigma \Z_{>0}\beta^\vee$ are on the same side of $M_\beta$, hence by Corollary~\ref{ordregeom}, for any two such $\nu,\nu'\in \lambda + \sigma \mathbb Z_{>0} \beta^\vee$ and any $w\in W$,
    \begin{equation}\label{eq: n_is_small1}w<_{v^\nu} s_\beta w \iff w<_{v^{\nu'}} s_\beta w.\end{equation} 

    Suppose first that $n\in \langle \lambda,\beta\rangle + \sigma \mathbb Z_{>1}$ and let $\mu=\lambda + (n-\langle \lambda,\beta\rangle) \beta^\vee$, Then
    \begin{enumerate}
    \item If $w<_{v^\mu}s_\beta w$, we have the chain
    \begin{equation}\pi^\lambda w < s_{\beta[\langle \lambda,\beta\rangle+\sigma]}\pi^\lambda w = \pi^{\lambda+\sigma \beta^\vee} s_\beta w <\pi^{\mu}w<\pi^\mu s_\beta w.\end{equation}
    The second inequality comes from $\pi^\mu w = s_{\beta[n+\sigma]}\pi^{\lambda+\sigma \beta^\vee}s_\beta w$ and Formula~\eqref{eq: leftaction}, the third comes from Equivalence~\eqref{eq: affinetorelative2}.
    \item Else $s_\beta w<_{v^\mu} w$, so by Equivalence~\eqref{eq: n_is_small1}, $s_\beta w <_{v^{\lambda+\sigma \beta^\vee}} w$ and we have the chain
    \begin{equation}\pi^\lambda w < s_{\beta[\langle \lambda,\beta\rangle+\sigma]}\pi^\lambda w = \pi^{\lambda+\sigma \beta^\vee} s_\beta w <\pi^{\lambda+\sigma \beta^\vee} w<\pi^\mu s_\beta w.\end{equation}
    Here the second inequality comes from Equivalence~\eqref{eq: affinetorelative2}. The third comes from $\pi^\mu s_\beta w = s_{\beta[n+\sigma]}\pi^{\lambda+\sigma \beta^\vee} w$ and Formula~\eqref{eq: leftaction}.
    \end{enumerate}
    Either way, for $n\in \langle \lambda,\beta\rangle + \sigma\mathbb Z_{>1}$, $s_{\beta[n]}\pi^\lambda w$ does not cover $\pi^\lambda w$. For $n\in -\sigma \mathbb Z_{>1}$ the argument is similar, because all the elements of $s_\beta \lambda - \sigma \mathbb Z_{>0} \beta^\vee$ are on the same side of $M_\beta$, in particular on the side of $\mu=s_\beta\lambda+n\beta^\vee$.
    \begin{enumerate}
        \item If $w<_{v^\mu}s_\beta w$, we have a chain
    \begin{equation}\pi^\lambda w < s_{\beta[-\sigma]}\pi^\lambda w = \pi^{s_\beta(\lambda+\sigma \beta^\vee)} s_\beta w <\pi^{\mu}w<\pi^\mu s_\beta w.\end{equation}
        \item Else $s_\beta w<_{v^\mu} w$ so $s_\beta w<_{v^{s_\beta(\lambda+\sigma\beta^\vee)}} w$ and we have a chain
    \begin{equation}\pi^\lambda w < s_{\beta[-\sigma]}\pi^\lambda w = \pi^{s_\beta(\lambda+\sigma \beta^\vee)} s_\beta w <\pi^{s_\beta(\lambda+\sigma \beta^\vee)} w<\pi^\mu s_\beta w.\end{equation}
    \end{enumerate}
    Hence the only possible covers (with varying coweights) are for $n \in \{-\sigma,\langle \lambda,\beta\rangle+\sigma\}$.
\end{proof}

\begin{Remark}
    To follow up on Remark~\ref{Remark: vectorial reflections}, note that by Equation~\eqref{eq : lefttorightmultiplication}, we have $s_{\beta[\sigma+\langle\lambda,\beta\rangle]}\pi^\lambda w = \pi^\lambda w s_{w^{-1}(\beta)[\sigma]}$ (where $\sigma=\sgn(\langle\lambda,\beta\rangle)$). Therefore Proposition~\ref{n_is_small} tells us that, if $\by$ covers $\bx$ in $W^a_+$, then $\by$ is obtained from $\bx$ applying an affinized reflection $s_{\tilde\beta[n]}$ either on the left or on the right, with $n\in \{-1,0,1\}$.
\end{Remark}
This is still far from a sufficient condition and many cases of potential covers can still be eliminated. We give another necessary condition for $\pi^\mu s_\beta w=s_{\beta[n]}\pi^\lambda w >\pi^\lambda w$ to be a cover, this is a generalisation of the chains produced in the proof of Theorem~\ref{non_reg_covers_same_orbit}:
\begin{Proposition}\label{covers_minimal_galleries}
    Let $\pi^\mu s_\beta w=s_{\beta[n]}\pi^\lambda w>\pi^\lambda w$ with $\mu^{++}\neq\lambda^{++}$. 
    Suppose that $s_\beta v^\mu$ is not on a minimal gallery from $w$ to $v^\lambda$. Then $\pi^\mu s_\beta w>\pi^\lambda w$ is not a cover. 
\end{Proposition}
\begin{proof}
    
We express the difference of $\varepsilon$-length using Formula~\eqref{eq: affinelengthexpressionepsilon}:

\begin{equation}  \ell^a_\varepsilon(\pi^\mu s_\beta w)- \ell^a_\varepsilon(\pi^\lambda w)=2\htt(\mu^{++}-\lambda^{++})+\varepsilon(\ell_{v^\mu}(s_\beta w)-\ell_{v^\lambda}(w)).\end{equation}

If there exists a reflection $r\in W$ such that $\ell_{v^\lambda}(rw)>\ell_{v^\lambda}(w)$ and $\ell_{v^\mu}(s_\beta rw)<\ell_{v^\mu}(s_\beta w)$ then using Formula~\eqref{eq: affinelengthexpressionepsilon} to compute the length $\ell^a_{\varepsilon}$, we have a chain
\begin{equation}\pi^\lambda w <\pi^\lambda rw <\pi^\mu s_\beta rw <\pi^\mu s_\beta w.\end{equation}
Since $\ell_{v}(rw)-\ell_{v}(w)=\ell(v^{-1}rw)-\ell(v^{-1}w)$ for $v,r,w\in W$, Lemma~\ref{Lemma : non_minimal_distance} guaranties the existence  of $r$, which proves the proposition.\end{proof}

In Figure~\ref{fig:galerieminimale} below, we give an example of a chain constructed this way in the $A_1$-affine case (with Cartan matrix $\begin{pmatrix}
    2 & -2 \\ -2 & 2
\end{pmatrix})$.

\begin{figure}[h!]

    \centerline{\includegraphics[width=\textwidth]{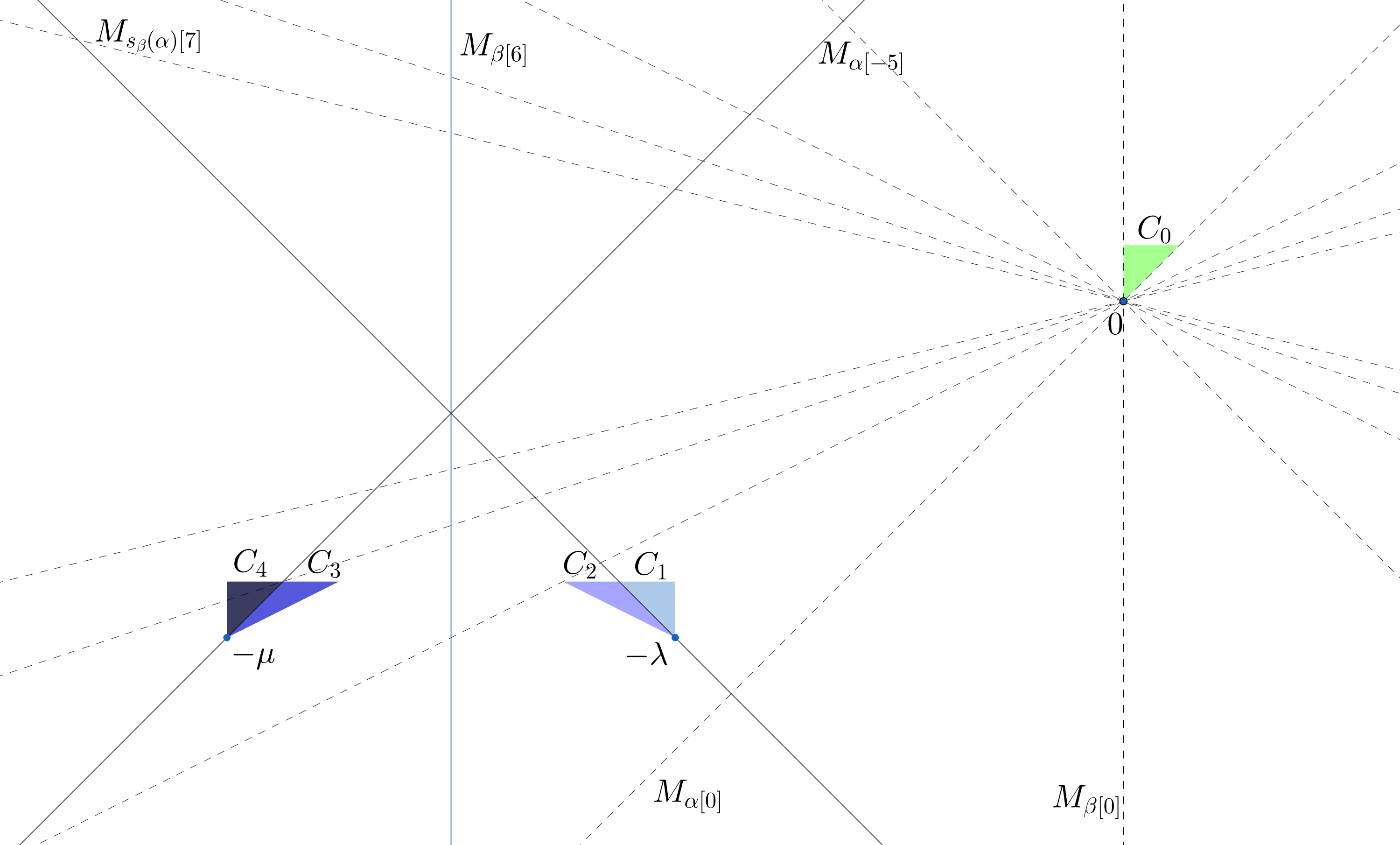}}
    \caption{Example of a chain constructed as in Proposition~\ref{covers_minimal_galleries}.}
    \label{fig:galerieminimale}

\end{figure}
 In this example, $\alpha$ and $\beta$ are the simple roots of a $A_1$-affinized root system, and we have chosen $\lambda$,$w$ and $\beta[n]$ such that $v^\lambda =s_\alpha$, $v^\mu=s_\alpha s_\beta$ and $w=s_\beta$. $\pi^\lambda w$ corresponds to the alcove $C_1$ in light blue, and its image $\pi^\mu s_\beta w$ by $s_{\beta[6]}$ corresponds to $C_4$. Since $r=s_\beta s_\alpha s_\beta$ verifies $d(v^\lambda,rw)=s_\alpha s_\beta s_\alpha>d(v^\lambda,w)=s_\alpha s_\beta$, and $d( s_\beta v^\mu,rw)=s_\beta<s_\beta s_\alpha =d(s_\beta v^\mu,w)$, there is a chain $\pi^\lambda w < \pi^\lambda r w <\pi^\mu s_\beta rw < \pi^\mu s_\beta w$ which corresponds to the sequence of alcoves $(C_1,C_2,C_3,C_4)$ on the figure.

\begin{Remark}

Let $v_0,w_0\in W$, $\mu_0\in Y$ and $\alpha_0\in \Phi$. To produce chains, note that Formula~\eqref{eq: leftaction} applied with the affinized reflection $s_{v_0(\alpha_0)[m+\langle \mu_0,\alpha_0\rangle]}$ to $\pi^{v_0(\mu_0)}w_0$ gives
\begin{equation}\label{eq : inequality condition}
    \forall m \in \mathbb Z \setminus \llbracket-\langle \mu_0,\alpha_0\rangle,0 \rrbracket,\; \pi^{v_0(\mu_0+m\alpha_0^\vee)}s_{v_0(\alpha_0)}w_0>\pi^{v_0(\mu_0)}w_0.
\end{equation} 
Moreover, applying the affinized reflection $s_{v_0(\alpha_0)[-m]}$ to $\pi^{v_0(\mu_0)}w_0$ instead we obtain 
\begin{equation} \label{eq : inequality conditionbis}
   \forall m \in \mathbb Z \setminus \llbracket-\langle \mu_0,\alpha_0\rangle,0 \rrbracket,\; \pi^{v_0s_{\alpha_0}(\mu_0+m\alpha_0^\vee)}s_{v_0(\alpha_0)}w_0>\pi^{v_0(\mu_0)}w_0 .
\end{equation}
For $m\in \rrbracket -\langle \mu_0,\alpha_0\rangle, 0\llbracket$ the inequalities are reversed. The limit cases $m\in \{ -\langle \mu_0,\alpha_0\rangle,0\}$ need to be treated more carefully, they depend on the sign of the root $v_0(\alpha_0)$ (because Formula~\eqref{eq: leftaction} holds for the affinized reflection $s_{v_0(\alpha_0)[0]}$ only if $v_0(\alpha_0)\in \Phi_+$), on the sign of $\langle \mu_0,\alpha_0\rangle$ and on the vectorial element $w_0$.
\end{Remark}
 \subsection{Another expression for the affinized length difference}\label{subsection: 3.2}

 Outside of the case of vectorial covers dealt with in Theorems~\ref{Theorem: constantcoweight} and~\ref{non_reg_covers_same_orbit}, if we write $\bx= \pi^{v\lambda}w$ with $\lambda \in Y^{++}, v,w\in W$ with $v$ of minimal length in $vW_\lambda$, by Proposition~\ref{n_is_small} the only covers are of the form $\by\in \{\pi^{v(\lambda+\beta^\vee)}s_{v(\beta)} w, \pi^{vs_\beta(\lambda+\beta^\vee)} s_{v(\beta)} w\}$ for some $\beta \in \Phi_+$, so the rest of this paper is dedicated to covers of this sort.
 \begin{Notation}\label{Notation: properly affine covers}
     From now on, unless stated otherwise, we use the following notation:
     \begin{enumerate}
         \item $\lambda\in Y^{++}$ is a dominant coweight
         \item $\beta\in\Phi_+$ is a positive root
         \item $v\in W^\lambda$ is the minimal representative of a $W_\lambda$-coset.
         \item $w\in W$ is any element of $W$.
         \item $\bx=\pi^{v(\lambda)}w$ and $\by\in \{\pi^{v(\lambda+\beta^\vee)}s_{v(\beta)}w,\; \pi^{vs_\beta(\lambda+\beta^\vee)}s_{v(\beta)}w\}$ are elements of $W^+$.
              \end{enumerate}

In particular, $\lambda$ is a dominant coweight. This choice is made in order to avoid the heavier notation $\lambda^{++}$. Recall from Notation~\ref{Notation : paraboliccosets} that $W^\lambda$ is the set of minimal coset representatives of $W/W_\lambda$, where $W_\lambda$ is the standard parabolic subgroup $\{w\in W\mid w(\lambda)=\lambda\}$. 
 \end{Notation}
 In this subsection, we give another expression for the length difference $\ell^a(\by)-\ell^a(\bx)$.

The two following lemmas give information on the vectorial chamber of $v(\lambda+\beta^\vee)$.

\begin{Lemma}\label{Lemma : l(su)=l(s)+l(u)}
Let $\lambda\in Y^{++}$ be a dominant coweight and let $\beta \in \Phi_+$ be a positive root such that $\lambda+\beta^\vee \in Y^+$. Let $u\in W$ be such that $\lambda+\beta^\vee$ belongs to the vectorial chamber $C^v_u$, that is to say $u^{-1}(\lambda+\beta^\vee)\in Y^{++}$. Then \begin{equation}\ell(s_\beta u)=\ell(s_\beta)+\ell(u).\end{equation}
\end{Lemma}
\begin{proof}
    Let $s_{\tau_1}\ldots  s_{\tau_n}$ be a reduced expression of $u$, so that $\ell(u)=n$ and $$\Inv(u^{-1})=\{\tau_1,s_{\tau_1}(\tau_2),\dots,s_{\tau_1}\ldots s_{\tau_{n-1}}(\tau_n)\}.$$ We show that $s_{\tau_{k+1}}\ldots s_{\tau_1}s_\beta >s_{\tau_k}\ldots s_{\tau_1}s_\beta$ for all $k\in \llbracket0, n-1\rrbracket$.
    
     For any $\alpha \in \Inv(u^{-1})$ we have $\langle \lambda+\beta^\vee,\alpha\rangle \leq 0$ (because $\lambda+\beta^\vee \in C^v_u$). Since $\lambda$ is dominant this implies $\langle \beta^\vee,\alpha \rangle \leq 0$.
     
 Let $k\in\llbracket 0,n-1\rrbracket$. Then $\langle \beta^\vee,s_{\tau_1}\ldots s_{\tau_k}(\tau_{k+1})\rangle \leq 0$, so:  $s_\beta(s_{\tau_1}\ldots s_{\tau_k}(\tau_{k+1}))=s_{\tau_1}\ldots s_{\tau_k}(\tau_{k+1})-\langle \beta^\vee,s_{\tau_1}\ldots s_{\tau_k}(\tau_{k+1})\rangle \beta$ is a positive root as a sum of positive roots. Thus $s_{\tau_{k+1}}\ldots s_{\tau_1}s_\beta >s_{\tau_k}\ldots s_{\tau_1}s_\beta$ for any $k\in\llbracket 0,n-1\rrbracket$ and therefore $\ell(s_\beta u)=\ell(u^{-1}s_\beta)=n+\ell(s_\beta)=\ell(s_\beta)+\ell(u)$.
\end{proof}

\begin{Lemma}\label{Lemma : l(vu)=l(v)+l(u)}
    Let $\lambda \in Y^{++}$ be a dominant coweight and let $\beta\in\Phi_+$ be a positive root such that $\lambda+\beta^\vee \in Y^+$. Let $v\in W^\lambda,\,w\in W$ and let $u$ denote the element $v^{\lambda+\beta^\vee}$.

    Then, if $\pi^{v(\lambda+\beta^\vee)}s_{v(\beta)}w$ (resp. $\pi^{vs_\beta(\lambda+\beta^\vee)}s_{v(\beta)}w$) covers $\bx=\pi^{v(\lambda)}w$, 
    \begin{equation}\ell(vu)=\ell(v)+\ell(u) \; \text{ (resp. } \ell(vs_\beta u)=\ell(v)+\ell(s_\beta u)\; \text{and } \ell(vs_\beta)=\ell(v)+\ell(s_\beta)\text{)}.\end{equation}
\end{Lemma}

\begin{proof}
    To simplify the notation, write $W_J$ for $W_{(\lambda+\beta^\vee)^{++}}$. Note that, with the notation of Definition~\ref{def: projparabolic}, $vu=\pr_{vuW_J}(v)$ since $u$ is the element of minimal length in $uW_J$.
    
    Suppose by contradiction that $\pi^{v(\lambda+\beta^\vee)}s_{v(\beta)}w$ covers $\bx$ with $\ell(vu)<\ell(v)+\ell(u)$ . Then $d^\mathbb N(1,vu)=\ell(vu)<d^\mathbb N (1,v)+d^\mathbb N (v,vu)=\ell(v)+\ell(u)$, so $v$ is not on a minimal gallery from $1$ to $vu$. Therefore by Lemma~\ref{Lemma : non_minimal_distance}, there is a reflection $r\in W$ such that $d(1,rvu)>d(1,vu)$ and $d(1,rv)<d(1,v)$, that is to say $rv<v$ and $rvu>vu$.
    
    By minimality of $u$, $r$ is not in $vuW_J (vu)^{-1}$: otherwise $rvu \in vuW_J$ verifies $d(v,rvu)=d(rv,vu)<d(v,vu)$, because foldings reduce the vectorial distance and $v,vu$ are on different sides of the wall $M_r$ associated to $r$. 
    
    Since $vu$ is the projection of $v$ on $vuW_J$ which is convex (see \cite[Lemma 2.10]{ronan1989lectures}), and since the wall $M_r$ separates $v$ and $vu$, any element of $vuW_J$ is on the same side of the wall $M_r$ as $vu$, so $rvu\tilde u > vu\tilde u $ for any $\tilde u \in W_J$. In particular, let $\tilde u \in W_J$ be such that $rvu\tilde u$ is the minimal coset representative of $rvuW_J$. Then by Proposition~\ref{affine_to_relative0}, since $rvu\tilde u > vu\tilde u$, we have \begin{equation}\pi^{rv(\lambda+\beta^\vee)}rs_{v(\beta)}w=\pi^{rvu\tilde u((\lambda+\beta^\vee)^{++})}rs_{v(\beta)}w<\pi^{vu\tilde u((\lambda+\beta^\vee)^{++})}s_{v(\beta)}w=\pi^{v(\lambda+\beta^\vee)}s_{v(\beta)}w.\end{equation}
    
    Therefore by Proposition~\ref{affine_to_relative0} for the left and right hand side inequalities and Formula~\eqref{eq : inequality condition} applied with $(\mu_0,\alpha_0,v_0,w_0,m)=(\lambda,\beta,rv,rw,1)$ for the middle one, we have a chain
    \begin{equation}\pi^{v(\lambda)}w < \pi^{rv(\lambda)}rw < \pi^{rv(\lambda+\beta^\vee)} s_{rv(\beta)}rw=\pi^{rv(\lambda+\beta^\vee)} rs_{v(\beta)}w< \pi^{v(\lambda+\beta^\vee)}s_{v(\beta)}w.\end{equation}
    Therefore if $\pi^{v(\lambda+\beta^\vee)}s_{v(\beta)}w$ covers $\bx$ then $\ell(vu)=\ell(v)+\ell(u)$.

    Now assume by contradiction that $\pi^{vs_\beta (\lambda+\beta^\vee)}s_{v(\beta)}w$ covers $\pi^\lambda w$ with $\ell(vs_\beta u)<\ell(v)+\ell(s_\beta u)$. Then, similarly there is a reflection $r\in W$ such that $rv<v$ and $rvs_\beta u\tilde u>vs_\beta u\tilde u$. By Proposition~\ref{affine_to_relative0} for the left and right hand side inequalities and Formula~\eqref{eq : inequality conditionbis} applied with $(\mu_0,\alpha_0,v_0,w_0,m)=(\lambda,\beta,rv,rw,1)$ for the middle one, we have a chain
    \begin{equation}\pi^{v(\lambda)}w < \pi^{rv(\lambda)}rw < \pi^{rvs_\beta(\lambda+\beta^\vee)} s_{rv(\beta)}rw=\pi^{rvs_\beta(\lambda+\beta^\vee)} rs_{v(\beta)}w< \pi^{vs_\beta(\lambda+\beta^\vee)}s_{v(\beta)}w.\end{equation}
    We deduce that, if $\pi^{vs_\beta(\lambda+\beta^\vee)}s_{v(\beta)}w$ covers $\bx$ then $\ell(vs_\beta u)=\ell(v)+\ell(s_\beta u)$. By Lemma~\ref{Lemma : l(su)=l(s)+l(u)} this is $\ell(v)+\ell(s_\beta)+\ell(u)$, by the triangle inequality we deduce that $\ell(v)+\ell(s_\beta)\geq \ell(vs_\beta)\geq \ell(vs_\beta u)-\ell(u)=\ell(v)+\ell(s_\beta)$ and we obtain the second equality in this case.
\end{proof}

\begin{Proposition}\label{Prop : length expression}
     Let $\lambda \in Y^{++}, v\in W^\lambda,w\in W$. Let $\beta \in \Phi_+$ be a positive root such that $\lambda+\beta^\vee \in Y^+$ and let $u$ denote $v^{\lambda+\beta^\vee}\in W^{(\lambda+\beta^\vee)^{++}}$.
     
     Suppose that $\by \in \{ \pi^{v(\lambda+\beta^\vee)}s_{v(\beta)}w, \pi^{vs_\beta(\lambda+\beta^\vee)}s_{v(\beta)}w\}$ covers $\bx = \pi^{v(\lambda)}w$. Then \begin{equation}\label{eq: length difference generalcase}\ell^a(\by)-\ell^a(\bx) = \Bigg(2\htt(\beta^\vee) -\ell(s_\beta)\Bigg)-2\Bigg(\ell(u)+\sum\limits_{\tau \in \Inv(u^{-1})} \langle \lambda+\beta^\vee,\tau\rangle \Bigg).\end{equation}
     
 \end{Proposition}

\begin{proof}
Let $W_J$ denote the standard parabolic subgroup $W_{(\lambda+\beta^\vee)^{++}}$. Recall that $u=v^{\lambda+\beta^\vee}$ is the minimal element of $W$ such that $u((\lambda+\beta^\vee)^{++})=\lambda+\beta^\vee$, so it is the minimal representative of the coset $uW_J$. By Proposition~\ref{Prop : affine_length_expression0} we have
\begin{align}\ell^a(\pi^{v(\lambda+\beta^\vee)} s_{v(\beta)}w)-\ell^a(\pi^{v(\lambda)} w)&=2\htt((\lambda+\beta^\vee)^{++})-2\htt(\lambda)+\ell_{v^{v(\lambda+\beta^\vee)}}(s_{v(\beta)}w)-\ell_v(w) \label{eq: lengthexpr0} \\
    \ell^a(\pi^{vs_\beta(\lambda+\beta^\vee)} s_{v(\beta)}w)-\ell^a(\pi^{v(\lambda)} w)&=2\htt((\lambda+\beta^\vee)^{++})-2\htt(\lambda)+\ell_{v^{vs_\beta(\lambda+\beta^\vee)}}(s_{v(\beta)}w)-\ell_v(w). \label{eq: lengthexpr01}\end{align}

    We unwrap these formulas with the help of previous results.
   \begin{enumerate} \item In the case $\by=\pi^{v(\lambda+\beta^\vee)}s_{v(\beta)}w$, let $\tilde u \in W_J$ be such that $vu\tilde u=(vu)^{J}=v^{v(\lambda+\beta^\vee)}$. The term $\ell_{v^{v(\lambda+\beta^\vee)}}(s_{v(\beta)}w)-\ell_v(w)$ rewrites as $\ell((u\tilde u)^{-1}s_\beta v^{-1}w)-\ell(vu\tilde u)-\ell(v^{-1}w)+\ell(v)$.
    Since $\by>\bx$ is a covering, by Proposition~\ref{covers_minimal_galleries}, $vs_\beta u\tilde u=s_{v(\beta)}(vu)^{J}$ is on a minimal gallery from $v$ to $w$, so $\ell(v^{-1}w)=\ell((vs_{\beta}u\tilde u)^{-1}w)+\ell(s_\beta u\tilde u)$. Moreover by Lemma~\ref{Lemma : l(su)=l(s)+l(u)}, $\ell(s_\beta u\tilde u)=\ell(s_\beta)+\ell(u\tilde u)$ and, by Lemma~\ref{Lemma : l(vu)=l(v)+l(u)}, $\ell(vu)=\ell(v)+\ell(u)$. Finally, by Equation~\eqref{eq : additive parabolic length}, since $u=u^J=v^{\lambda+\beta^\vee}$ and $vu\tilde u = (vu)^J=v^{v(\lambda+\beta^\vee)}$, we have $\ell(u\tilde u)=\ell(u)+\ell(\tilde u)$ and $\ell(vu)=\ell(vu \tilde u)+\ell(\tilde u)$. Thus
    \begin{align}
        \ell_{v^{v(\lambda+\beta^\vee)}}(s_{v(\beta)}w)-\ell_v(w) &= \ell((u\tilde u)^{-1}s_\beta v^{-1}w)-\ell(v^{-1}w)-\ell(vu\tilde u)+\ell(v) \nonumber \\
        &=-\ell(s_\beta u\tilde u)-\ell(vu)+\ell(\tilde u)+\ell(v) \nonumber\\
        &=-\ell(s_\beta)-\ell(u\tilde u)-\ell(u)+\ell(\tilde u) \nonumber\\
        &=-\ell(s_\beta)-2\ell(u). \label{eq: lengthexpr1}
    \end{align}
\item In the second case, let $\tilde u \in W_J$ be such that $vs_\beta u \tilde u=(vs_\beta u)^{J}=v^{vs_\beta(\lambda+\beta^\vee)}$, then $\ell_{v^{vs_\beta(\lambda+\beta^\vee)}}(s_{v(\beta)}w)-\ell_v(w)$ rewrites as $\ell((u\tilde u)^{-1}v^{-1}w)-\ell(vs_\beta u\tilde u)-\ell(v^{-1}w)+\ell(v)$. By Proposition~\ref{covers_minimal_galleries}, $\ell((u\tilde u)^{-1}v^{-1}w)=\ell(v^{-1}w)-\ell(u\tilde u)$. By Equation~\eqref{eq : additive parabolic length}, $\ell(u\tilde u)=\ell(u)+\ell(\tilde u)$ and $\ell(vs_\beta u\tilde u)=\ell(vs_\beta u)-\ell(\tilde u)$. By Lemma~\ref{Lemma : l(vu)=l(v)+l(u)} and~\ref{Lemma : l(su)=l(s)+l(u)}, $\ell(vs_\beta u)=\ell(v)+\ell(s_\beta u)=\ell(v)+\ell(s_\beta)+\ell(u)$.

Thus in this case: 
\begin{align}\ell_{v^{vs_\beta(\lambda+\beta^\vee)}}(s_{v(\beta)}w)-\ell_v(w) &= \ell((u\tilde u)^{-1}v^{-1}w)-\ell(vs_\beta u\tilde u)-\ell(v^{-1}w)+\ell(v) \nonumber \\
&= \ell(v^{-1}w)-\ell(u\tilde u)-(\ell(vs_\beta u)-\ell(\tilde u))-\ell(v^{-1}w)+\ell(v) \nonumber\\ 
&= -\ell(s_\beta)-2\ell(u). \label{eq: lengthexpr2}\end{align}

\item By Lemma~\eqref{Lemma : height of general coweight} we have \begin{align}
    2\htt((\lambda+\beta^\vee)^{++})&= 2(\htt(\lambda+\beta^\vee)-\sum\limits_{\tau\in \Inv(u^{-1})} \langle \lambda+\beta^\vee,\tau\rangle) \nonumber \\
    &=2(\htt(\lambda)+\htt(\beta^\vee)-\sum\limits_{\tau\in \Inv(u^{-1})} \langle \lambda+\beta^\vee,\tau\rangle). \label{eq: lengthexpr3}
\end{align}

\end{enumerate}
By reinjecting Expressions~\eqref{eq: lengthexpr1},~\eqref{eq: lengthexpr3} in Formula~\eqref{eq: lengthexpr0}, and Expressions~\eqref{eq: lengthexpr2},~\eqref{eq: lengthexpr3} in Formula~\eqref{eq: lengthexpr01} we obtain either way:

\begin{align*}
    \ell^a(\by)-\ell^a(\bx) &=2\htt(\lambda)+2\htt(\beta^\vee)-2\sum\limits_{\tau\in \Inv(u^{-1})} \langle \lambda+\beta^\vee,\tau\rangle-2\htt(\lambda)-\ell(s_\beta)-2\ell(u) \\
    &=\Bigg(2\htt(\beta^\vee)-\ell(s_\beta)\Bigg)-2\Bigg(\ell(u)+\sum\limits_{\tau\in \Inv(u^{-1})} \langle \lambda+\beta^\vee,\tau\rangle\Bigg).
\end{align*}
\end{proof}
Using Lemma~\ref{Lemma : height as sum}, it is easy to see that $2\htt(\beta^\vee)-\ell(s_\beta)$ is always positive and that, on the contrary, $\ell(u)+\sum\limits_{\tau\in \Inv(u^{-1})} \langle \lambda+\beta^\vee,\tau\rangle$ is always non-positive. Therefore, the length difference is equal to $1$ if and only if in the right-hand side of Equation~\eqref{eq: length difference generalcase}, the first term is equal to $1$ and the second term cancels out. This motivates the following definitions.
\begin{Definition}\label{Definition quantum}
  A coweight $\mu\in Y^+$ is \textbf{almost dominant} if and only if
        \begin{equation}
            \forall \tau\in\Phi_+,\; \langle \mu,\tau\rangle \geq -1.
        \end{equation}
 A root $\beta\in \Phi_+$ is \textbf{a quantum root} if and only if
        \begin{equation}
            \ell(s_\beta)=2\htt(\beta^\vee) -1.
        \end{equation}
\end{Definition}
The notion of quantum roots comes from the definition of quantum Bruhat graphs, (see~\cite[\S 4.1]{lenart2012Kirillov}). With Notation~\ref{Notation: properly affine covers}, in Subsection~\ref{subsection: almost dominance} we prove that, if $\by$ covers $\bx$ then $\lambda+\beta^\vee$ is almost dominant and we prove in Subsection~\ref{subsection: quantum roots} that $\beta$ needs to be a quantum root.

\begin{Remark}
    Note that if $\lambda+\beta^\vee$ is dominant, then the second term in the right-hand side of Expression~\eqref{eq: length difference generalcase} immediately cancels out, since in this case $u=1_W$. In the reductive case, $\Phi$ is finite and therefore if $\lambda$ is far enough in the fundamental chamber (meaning that $\langle \lambda,\alpha_i\rangle$ is large for all $i\in I$, we say that $\lambda$ is superregular), then $\lambda+\beta^\vee$ is always dominant. Accordingly, covers of $\pi^{v(\lambda)}w$ for $\lambda$ superregular are easier to classify (see \cite[Proposition 4.4]{lam2010quantum}, \cite[Proposition 4.4]{milicevic2021newton}).
\end{Remark}
\subsection{Almost-dominance in properly affine covers}\label{subsection: almost dominance}
In this section, we prove that the second term of the right-hand side of Expression~\eqref{eq: length difference generalcase} need to be zero when $\by$ covers $\bx$ (with Notation~\ref{Notation: properly affine covers}), through the following proposition:
\begin{Proposition}\label{Proposition : <lambda+beta,tau>}
    Let $\lambda\in Y^{++}$, $v\in W^\lambda$ and $w \in W$. Let $\beta\in \Phi_+$ be a positive root and suppose that $\pi^{v(\lambda+\beta^\vee)}s_{v(\beta)}w$ or $\pi^{vs_\beta(\lambda+\beta^\vee)}s_{v(\beta)}w$ covers $\pi^{v(\lambda)}w$. Then $\lambda+\beta^\vee$ is almost dominant, that is to say:
    \begin{equation}
        \forall \tau \in \Phi_+,\; \langle \lambda+\beta^\vee,\tau\rangle \geq -1.
    \end{equation}
\end{Proposition}
It is deduced from the following two technical lemmas, we give their proofs after the proof of Proposition~\ref{Proposition : <lambda+beta,tau>}.
\begin{Lemma}\label{SubLemma_2: <lambda+beta,tau>}
    Let $\lambda \in Y^{++},\, v\in W^\lambda,\, w\in W,\,\beta \in \Phi_+$. Suppose that there exists a pair $(\tau,n)\in \Phi_+\times \mathbb Z$ such that:
    \begin{enumerate}[(i)]
        \item\label{point i} $n>0$
        \item\label{point ii} $\langle \lambda+n\tau^\vee,\beta\rangle \geq -1$
        \item\label{point iii} $n<-\langle \lambda+\beta^\vee,\tau\rangle$.
    \end{enumerate}
    Then, $\pi^{v(\lambda+\beta^\vee)}s_{v(\beta)}w$ and $\pi^{vs_\beta(\lambda+\beta^\vee)}s_{v(\beta)}w$ do not cover $\pi^{v(\lambda)}w$. 
\end{Lemma}

\begin{Lemma}\label{SubLemma : <lambda+beta,tau>}
    Let $\lambda \in Y^{++}$ and $\beta\in\Phi_+$ be such that $\lambda+\beta^\vee$ lies in $Y^+$. Let $\tau\in \Phi_+$ be such that $\langle \lambda+\beta^\vee,\tau\rangle\leq-2$ and suppose that $\langle \tau^\vee,\beta\rangle\leq-2$. Then $\langle\lambda+\beta^\vee,s_\tau(\beta)\rangle\geq -1$.
\end{Lemma}

\begin{proof}[Proof of Proposition~\ref{Proposition : <lambda+beta,tau>}]
We prove the contrapositive:
Let $\tau \in \Phi_+$ be a positive root such that $\langle \lambda+\beta^\vee,\tau\rangle \leq -2 $, we will produce non-trivial chain from $\pi^{v(\lambda)}w$ to $\pi^{v(\lambda+\beta^\vee)}s_{v(\beta)}w$ and $\pi^{vs_\beta(\lambda+\beta^\vee)}s_{v(\beta)}w$. In particular since $\lambda$ is dominant, $\langle \beta^\vee,\tau\rangle \leq -2$.

The numbers $\langle \tau^\vee,\beta\rangle$ and $\langle \beta^\vee,\tau\rangle$ have the same sign (\cite[Lemma 1.1.10]{bardy1996systemes}), therefore we have that $\langle \tau^\vee,\beta \rangle\leq-1$.
    
    Suppose first that $\langle \tau^\vee,\beta\rangle \leq -2$, then $(\tau,-(\langle \lambda+\beta^\vee,\tau\rangle+1))$ is a pair which satisfy the conditions of Lemma~\ref{SubLemma_2: <lambda+beta,tau>}:
    
    \begin{enumerate}[(i)]
        \item Since $\langle \lambda+\beta^\vee,\tau\rangle\leq -2$, and $-(\langle \lambda+\beta^\vee,\tau\rangle +1)\geq 1 >0$.
        \item By Lemma~\ref{SubLemma : <lambda+beta,tau>}, $\langle \lambda+\beta^\vee,s_\tau(\beta)\rangle \geq -1$, thus \begin{align*}\langle \lambda-(\langle \lambda+ \beta^\vee,\tau\rangle+1)\tau^\vee,\beta\rangle &= \langle s_\tau(\lambda+\beta^\vee)-\beta^\vee-\tau^\vee,\beta\rangle \\ &=\langle \lambda+\beta^\vee,s_\tau(\beta)\rangle-2-\langle \tau^\vee,\beta\rangle \\ &\geq \langle \lambda+\beta^\vee,s_\tau(\beta)\rangle \geq -1.\end{align*}
        \item Clearly $-(\langle \beta^\vee,\tau\rangle +1)<-\langle \beta^\vee,\tau\rangle$.
    \end{enumerate}
Suppose now that $\langle \tau^\vee,\beta\rangle = -1$, we show that $(\tau,1)$ is a pair satisfying the conditions of Lemma~\ref{SubLemma_2: <lambda+beta,tau>}:

\begin{enumerate}[(i)]
    \item The first point is trivially verified.
    \item Since $\langle \tau^\vee,\beta\rangle = -1$ and $\lambda$ is dominant, $\langle \lambda+\tau^\vee,\beta\rangle \geq -1$.
    \item Since $\langle \lambda+\beta^\vee,\tau\rangle \leq -2$, $1< -\langle \lambda+\beta^\vee,\tau\rangle$.
\end{enumerate}
Hence, either way, if such a $\tau\in \Phi_+$ exists, then by Lemma~\ref{SubLemma_2: <lambda+beta,tau>} 
 $\pi^{v(\lambda+\beta^\vee)}s_{v(\beta)}w$ and $\pi^{vs_\beta(\lambda+\beta^\vee)}s_{v(\beta)}w$ do not cover $\pi^{v(\lambda)}w$.
\end{proof}

\begin{proof}[Proof of Lemma~\ref{SubLemma_2: <lambda+beta,tau>}]
    We use conditions (\ref{point i}),(\ref{point ii}), (\ref{point iii}) in the statement to produce chains from $\pi^{v(\lambda)}w$ to $\pi^{v(\lambda+\beta^\vee)}s_{v(\beta)}w$ and $\pi^{vs_\beta(\lambda+\beta^\vee)}s_{v(\beta)}w$. 
    
 Suppose first that Inequality (\ref{point ii}) is strict, then we show that we have the following chains:
        \begin{align}
             \pi^{v(\lambda)}w&<\pi^{vs_\tau(\lambda+n\tau^\vee)}s_{v(\tau)}w&<\pi^{vs_\tau(\lambda+\beta^\vee+n\tau^\vee)}s_{v(\tau)}s_{v(\beta)}w&<\pi^{v(\lambda+\beta^\vee)}s_{v(\beta)}w \label{eq: chains1} \\
             \pi^{v(\lambda)}w&<\pi^{v(\lambda+n\tau^\vee)}s_{v(\tau)}w &<\pi^{vs_{\beta}(\lambda+\beta^\vee+n\tau)}s_{v(\beta)}s_{v(\tau)}w &<\pi^{vs_{\beta}(\lambda+\beta^\vee)}s_{v(\beta)}w. \label{eq: chains2}
        \end{align}

    \begin{enumerate}[a)]
        \item By Condition (\ref{point i}), since $\lambda$ is dominant and $\tau$ is a positive root, using Formula~\eqref{eq : inequality condition} with $(\mu_0,\alpha_0,v_0,w_0,m)=(\lambda,\tau,v,w,n)$, we have \begin{equation}\label{eqsublemma1}\pi^{v(\lambda)}w<\pi^{v(\lambda+n\tau^\vee)}s_{v(\tau)}w.\end{equation} Using Formula~\eqref{eq : inequality conditionbis} with the same parameters gives \begin{equation}\label{eqsublemma1bis}\pi^{v(\lambda)}w <\pi^{vs_{\tau}(\lambda+n\tau^\vee)}s_{v(\tau)}w.\end{equation}
      
        \item We supposed that Inequality (\ref{point ii}) is strict, so $\langle \lambda+n\tau^\vee,\beta\rangle \geq 0$. Then by Formula~\eqref{eq : inequality condition} applied with $(\mu_0,\alpha_0,v_0,w_0,m)=(\lambda+n\tau^\vee,\beta,vs_\tau,s_{v(\tau)}w,1)$, we get
    \begin{equation}\label{eqsublemma2bis}
    \pi^{vs_\tau(\lambda+n\tau^\vee)}s_{v(\tau)}w<\pi^{vs_\tau(\lambda+\beta^\vee+n\tau^\vee)}s_{v(\tau)}s_{v(\beta)}w.    \end{equation}

    Moreover by Formula~\eqref{eq : inequality conditionbis} applied with $(\mu_0,\alpha_0,v_0,w_0,m)=(\lambda+n\tau^\vee,\beta,v,s_{v(\tau)}w,1)$,
    \begin{equation}\label{eqsublemma2}
        \pi^{v(\lambda+n\tau^\vee)}s_{v(\tau)}w <\pi^{vs_\beta(\lambda+\beta^\vee+n\tau^\vee)}s_{v(\beta)}s_{v(\tau)}w.
    \end{equation}
        \item Since $\langle \tau^\vee,\tau\rangle=2$, Condition (\ref{point iii}) is equivalent to $-n<- \langle \lambda+\beta^\vee+n\tau^\vee,\tau \rangle$, so, using Formula~\eqref{eq : inequality conditionbis} for $(\mu_0,\alpha_0,v_0,w_0,m)=(\lambda+\beta^\vee+n\tau^\vee,\tau,vs_\tau,s_{v(\tau)}s_{v(\beta)}w,-n)$, we get
        \begin{equation}\label{eqsublemma3}
            \pi^{vs_\tau(\lambda+\beta^\vee+n\tau^\vee)}s_{v(\tau)}s_{v(\beta)}w<\pi^{v(\lambda+\beta^\vee)}s_{v(\beta)}w.
        \end{equation}
        Moreover using Formula~\eqref{eq : inequality condition} for $(\mu_0,\alpha_0,v_0,w_0,m)=(\lambda+\beta^\vee+n\tau^\vee,\tau,vs_\beta,s_{v(\beta)}s_{v(\tau)}w,-n)$,
        \begin{equation}\label{eqsublemma3bis}
             \pi^{vs_{\beta}(\lambda+\beta^\vee+n\tau)}s_{v(\beta)}s_{v(\tau)}w <\pi^{vs_{\beta}(\lambda+\beta^\vee)}s_{v(\beta)}w.
        \end{equation}
           \end{enumerate} 
    Thus, if Inequality (\ref{point ii}) is strict, combining~\eqref{eqsublemma1}\eqref{eqsublemma2} and~\eqref{eqsublemma3bis} we obtain the chain~\eqref{eq: chains1}.
        Moreover combining~\eqref{eqsublemma1bis},~\eqref{eqsublemma2bis} and~\eqref{eqsublemma3} we obtain the chain~\eqref{eq: chains2}.
    This proves Lemma \ref{SubLemma_2: <lambda+beta,tau>} in this case.
    
    \medskip
    
    Suppose now that $\langle \lambda+n\tau^\vee,\beta\rangle = -1$. Then note that $\lambda+n\tau^\vee+\beta^\vee=s_\beta(\lambda+n\tau^\vee)$, and Formulas~\eqref{eq : inequality condition},~\eqref{eq : inequality conditionbis} can not be used for the middle inequalities of chains~\eqref{eq: chains1} and~\eqref{eq: chains2} anymore.

    \begin{enumerate}
        \item Case $\pi^{v(\lambda+\beta^\vee)}s_{v(\beta)}w$ if $\langle \lambda+n\tau^\vee,\beta\rangle=-1$.
        
    \begin{enumerate}[a)]
        \item If $vs_\tau(\beta) \in \Phi_+$ then we can apply Formula~\eqref{eq: leftaction} to the element $\pi^{vs_\tau(\lambda+\tau^\vee)}s_{v(\tau)}w$ and the positive affinized root $vs_\tau(\beta)[0]$, and since $\langle \lambda+n\tau^\vee,\beta\rangle = -1<0$, we still have
        \begin{equation*} \pi^{vs_\tau(\lambda+n\tau^\vee)}s_{v(\tau)}w<s_{vs_\tau(\beta)[0]}\pi^{vs_\tau(\lambda+\tau^\vee)}s_{v(\tau)}w=\pi^{vs_\tau(\lambda+\beta ^\vee+n\tau^\vee)}s_{v(\tau)}s_{v(\beta)}w\end{equation*}
        and the chain~\eqref{eq: chains1} still holds.

        \item If $vs_\tau(\beta)\in \Phi_-$, that is to say $s_\tau(\beta) \in \Inv(v)$, then by Proposition~\ref{affine_to_relative0} \begin{equation}\label{eqsublemma4}\pi^{v(\lambda)}w < \pi^{vs_{s_\tau(\beta)}(\lambda)}s_{vs_\tau(\beta)}w\end{equation} and by Formula~\eqref{eq : inequality conditionbis} applied with $(\mu_0,\alpha_0,v_0,w_0,m)=(\lambda,\tau,vs_{s_\tau(\beta)},s_{vs_\tau(\beta)}w, 1)$ we get \begin{equation}\label{eqsublemma4bis}\pi^{vs_{s_\tau(\beta)}(\lambda)}s_{vs_\tau(\beta)}w< \pi^{vs_\tau s_\beta(\lambda+n\tau^\vee)}s_{v(\tau)}s_{v(\beta)}w=\pi^{vs_\tau (\lambda+\beta^\vee+n\tau^\vee)}s_{v(\tau)}s_{v(\beta)}w.\end{equation} For the vectorial elements, we used the fact that $s_{vs_\tau(\beta)}=s_{v(\tau)}s_{v(\beta)}s_{v(\tau)}=vs_\tau s_\beta s_\tau v^{-1}$ and $s_{vs_{s_\tau(\beta)}(\tau)}=vs_\tau s_\beta s_\tau s_\beta s_\tau v^{-1}$, so $s_{vs_{s_\tau(\beta)}(\tau)}s_{vs_\tau(\beta)}=vs_\tau s_\beta v^{-1}=s_{v(\tau)}s_{v(\beta)}$.
         \end{enumerate}  
        Combining Inequalities~\eqref{eqsublemma4} and~\eqref{eqsublemma4bis} we obtain the chain
\begin{equation}\pi^{v(\lambda)}w<\pi^{vs_\tau s_\beta s_\tau (\lambda)}s_{v(\tau)}s_{v(\beta)} s_{v(\tau)}w <\pi^{v(s_\tau (\lambda+\beta^\vee+n\tau^\vee))}s_{v(\tau)}s_{v(\beta)}w<\pi^{v(\lambda+\beta^\vee)}s_{v(\beta)}w \end{equation} which replaces chain~\eqref{eq: chains1}.

    \item Case $\pi^{vs_\beta(\lambda+\beta^\vee)}s_{v(\beta)}w$ if $\langle \lambda+n\tau^\vee,\beta\rangle=-1$.
    \begin{enumerate}[a)]
    \item If $w^{-1}vs_\tau(\beta) \in \Phi_-$, by Formula~\eqref{eq: leftaction} applied to $\pi^{v(\lambda+n\tau^\vee)}s_{v(\tau)}w$ and the affinized root $v(\beta)[\langle \lambda+n\tau^\vee,\beta\rangle]$, since $\langle \lambda+n\tau^\vee,\beta\rangle <0$, \begin{equation*}\pi^{v(\lambda+n\tau^\vee)}s_{v(\tau)}w<\pi^{v(\lambda+n\tau^\vee)}s_{v(\beta)}s_{v(\tau)}w=\pi^{vs_\beta(\lambda+\beta^\vee+n\tau^\vee)}s_{v(\beta)}s_{v(\tau)}w\end{equation*} and the chain~\eqref{eq: chains2} still holds.
    \item If $w^{-1}vs_\tau(\beta) \in \Phi_+$, then using Formula~\eqref{eq: leftaction} with $\pi^{v(\lambda)}w$ and the affinized reflection $vs_\tau(\beta)[\langle \lambda,s_\tau(\beta)\rangle]$ (which is always possible because, if $\langle \lambda,s_\tau(\beta)\rangle=0$ then by minimality of $v$, $vs_\tau(\beta)\in\Phi_+$), we obtain \begin{equation}\label{eqsublemma5}\pi^{v(\lambda)}w<\pi^{v(\lambda)}s_{v(\tau)}s_{v(\beta)}s_{v(\tau)}w.\end{equation} Moreover, by Formula~\eqref{eq : inequality condition} applied with $(\mu_0,\alpha_0,v_0,w_0,m)=(\lambda,\tau,v,s_{vs_\tau(\beta)}w,n)$, we get    \begin{equation}\label{eqsublemma5bis}\pi^{v(\lambda)}s_{v(\tau)}s_{v(\beta)}s_{v(\tau)}w<\pi^{v(\lambda+n\tau^\vee)}s_{v(\beta)}s_{v(\tau)}w.\end{equation} Hence combining Inequalities~\eqref{eqsublemma5} and~\eqref{eqsublemma5bis} we obtain a chain
    \begin{equation}\pi^{v(\lambda)}w <\pi^{v(\lambda)}s_{v(\tau)}s_{v(\beta)}s_{v(\tau)}w <\pi^{vs_\beta(\lambda+\beta^\vee+n\tau^\vee)}s_{v(\beta)}s_{v(\tau)}w<\pi^{vs_\beta(\lambda+\beta^\vee)}s_{v(\beta)}w. \end{equation}
\end{enumerate}

    \end{enumerate}
Therefore, in all cases, if such a pair $(\tau,n)$ exists, then $\pi^{v(\lambda+\beta^\vee)}s_{v(\beta)}w$ and $\pi^{vs_\beta(\lambda+\beta^\vee)}s_{v(\beta)}w$ do not cover $\pi^{v(\lambda)}w$.
\end{proof}

\begin{proof}[Proof of Lemma~\ref{SubLemma : <lambda+beta,tau>}]

The proof relies on the assumption that $\lambda+\beta^\vee$ lies in the Tits cone, which is equivalent to saying that there is only a finite number of positive roots $\alpha$ such that $\langle \lambda+\beta^\vee,\alpha\rangle <0$.

We will show that $\langle \lambda+\beta^\vee, (s_\tau s_\beta)^n(\tau)\rangle \geq0$ for $n$ large enough implies $\langle\lambda+\beta^\vee,s_\tau(\beta)\rangle \geq-1$, which implies the lemma.
To shorten the computation, let us write $a=-\langle \beta^\vee,\tau\rangle$ and $a^\vee=-\langle \tau^\vee,\beta\rangle$. So the assumptions $\langle \lambda+\beta^\vee,\tau\rangle \leq -2$ and $\langle \tau^\vee,\beta\rangle\leq -2$ imply that $a\geq 2+\langle \lambda,\tau\rangle$ and $a^\vee \geq 2$.
In the basis $(\beta,\tau)$ of $\R \beta \oplus \R \tau$, the matrix of $s_\tau s_\beta$ is $M=\begin{pmatrix}
    -1 & a \\ -a^\vee & aa^\vee-1
\end{pmatrix}$. We have $\chi_M = X^2+(2-aa^\vee)X+1$ thus, since $aa^\vee \geq 4$, $M^2=(aa^\vee-2)M-I_2$. Write $M^n=\mu_n M+\nu_n I_2$ for $n\in \mathbb N$, then an easy computation shows that $\nu_n=-\mu_{n-1}$ and $\mu_{n+1}=(aa^\vee-2)\mu_n-\mu_{n-1}$. In particular since $aa^\vee-2 \geq 2$ and $\mu_0 = 0 < \mu_1$, by iteration $(\mu_n)_{n\in\mathbb N}$ is strictly increasing.

Let $x=\langle \lambda,\beta\rangle\geq0$ and $y=\langle \lambda,\tau\rangle \in \llbracket0,a-2\rrbracket$. Then
\begin{align}
    \langle \lambda+\beta^\vee,(s_\tau s_\beta)^n(\tau)\rangle & = \langle \lambda+\beta^\vee,a\mu_n\beta+((aa^\vee-1)\mu_n-\mu_{n-1})\tau\rangle \label{eq: sublemmatitscone1}\\
    & = (x+2)\mu_n a + ((aa^\vee-1)\mu_n -\mu_{n-1})(y-a). \nonumber 
\end{align}
Since $\lambda+\beta^\vee$ lies in the Tits cone, $\langle \lambda+\beta^\vee,(s_\tau s_\beta)^n(\tau)\rangle$ is non-negative for $n$ large enough. Moreover, $\mu_{n-1}<\mu_n$ for all $n\in \mathbb N$ and $a-y>0$. Therefore we deduce from Equation~\eqref{eq: sublemmatitscone1} that, for $n$ large enough,
\begin{equation*}(x+2)\mu_n a \geq (a-y)((aa^\vee-1)\mu_n-\mu_{n-1})>(a-y)\mu_n(aa^\vee-2).\end{equation*} Hence $$(x+2)>(a-y)(a^\vee-\frac{2}{a})=aa^\vee-a^\vee y-2+2\frac{y}{a}.$$
Therefore $\langle \lambda+\beta^\vee,s_\tau(\beta)\rangle = x+2+a^\vee y-aa^\vee>-2+2\frac{y}{a}$ and, since it is an integer, we deduce $\langle \lambda+\beta^\vee,s_\tau(\beta)\rangle \geq -1 \geq 1-a^\vee$, which proves the result.
\end{proof}

\begin{Corollary}\label{Corollary: simplification}
 Let $\lambda \in Y^{++}, v\in W^\lambda,w\in W$. Let $\beta \in \Phi_+$ be a positive root such that $\lambda+\beta^\vee \in Y^+$. Suppose that $\by \in \{ \pi^{v(\lambda+\beta^\vee)}s_{v(\beta)}w, \pi^{vs_\beta(\lambda+\beta^\vee)}s_{v(\beta)}w\}$ covers $\bx = \pi^{v(\lambda)}w$.
     
     Then \begin{align}\label{eq: length difference generalcase2}
         \ell^a(\by)-\ell^a(\bx)&=2\htt(\beta^\vee) -\ell(s_\beta). 
     \end{align}
\end{Corollary}
 \begin{proof}
   Let $u=v^{\lambda+\beta^\vee}\in W$. Then for any $\tau \in \Inv(u^{-1})$, by Lemma~\ref{Lemma : Minimal inversion set}, $\langle \lambda+\beta^\vee,\tau\rangle$ is negative. We deduce from Proposition~\ref{Proposition : <lambda+beta,tau>} that $\langle \lambda+\beta^\vee,\tau\rangle = -1$ for any such $\tau$. Therefore
   \begin{equation}\label{eq: lengthdifferencegeneralcase2proof}\sum\limits_{\tau \in \Inv(u^{-1})} \langle \lambda+\beta^\vee,\tau\rangle=-|\Inv(u^{-1})|=-\ell(u).\end{equation} We then directly obtain Expression~\eqref{eq: length difference generalcase2} from Expressions~\eqref{eq: length difference generalcase} and~\eqref{eq: lengthdifferencegeneralcase2proof}.
 \end{proof}

\subsection{Properly affine covers and quantum roots}\label{subsection: quantum roots}
We now prove in Proposition~\ref{Proposition : height is length} that, with Notation~\ref{Notation: properly affine covers}, if $\beta$ is not a quantum root, then $\by$ do not cover $\bx$. This is enough, together with Corollary~\ref{Corollary: simplification}, to conclude that $\ell^a(\by)-\ell^a(\bx)=1$.
There is a subtlety if the root $\beta$ lies in a subsystem of $\Phi$ of type $G_2$, we suppose that this is not the case in Lemma~\ref{Lemma : height is length} and Lemma~\ref{SubLemma : lengthexpression}, and we deal with the $G_2$ case in Lemma~\ref{Lemma : height is length G2}. Let us first give another characterization of quantum roots.
\begin{Lemma}\label{Lemma: quantum roots}
    A root $\beta\in \Phi_+$ is a quantum root if and only if $\langle \beta^\vee,\gamma\rangle =1$ for all $\gamma \in \Inv(s_\beta)\setminus\{\beta\}$.
\end{Lemma}
\begin{proof}
    Recall that a quantum root is a root $\beta\in \Phi_+$ such that $2\htt(\beta^\vee)=\ell(s_\beta)+1$. By Corollary~\ref{Lemma : height as sum}, this is equivalent to 
    \begin{equation}\label{equation : quantum height as sum}\sum\limits_{\gamma \in \Inv(s_\beta)}\langle \beta^\vee,\gamma\rangle = \ell(s_\beta)+1.\end{equation} 
    For any $\gamma \in \Inv(s_\beta)$, $\gamma$ is a positive root and $s_\beta(\gamma)=\gamma-\langle \beta^\vee,\gamma\rangle \beta$ is a negative root, therefore $\langle \beta^\vee,\gamma\rangle\geq 1$. Moreover, $\langle\beta^\vee,\beta\rangle =2$ and $|\Inv(s_\beta)|=\ell(s_\beta)$. Therefore Equation~\eqref{equation : quantum height as sum} is satisfied if and only if $\langle \beta^\vee,\gamma\rangle$ is exactly one for all $\gamma\in \Inv(s_\beta)\setminus\{\beta\}$.
\end{proof}

\begin{Lemma}\label{Lemma : height is length}
     Let $\lambda \in Y^{++}$, $v\in W^\lambda$, $w\in W$ and $\beta\in \Phi_+$. Let $\gamma \in \Inv(s_\beta)\setminus \{\beta\}$ be such that $\langle \beta^\vee,\gamma\rangle \geq2$ and suppose that $\beta \notin \Inv(s_\gamma)$. Then $\pi^{v(\lambda+\beta^\vee)}s_{v(\beta)}w$ and $\pi^{v(\lambda+\beta^\vee)}s_{v(\beta)}w$ do not cover $\pi^{vs_\beta(\lambda)}w$. 
\end{Lemma}

\begin{proof}
By the contrapositive of Proposition~\ref{Proposition : <lambda+beta,tau>}, we can suppose that $\langle \lambda+\beta^\vee,\tau\rangle \geq -1$ for any $\tau \in \Phi_+$.
    Let $\gamma$ be as in the statement and write $\alpha=s_\gamma(\beta)\in \Phi_+$. 
    We will construct non-trivial chains in the same fashion as in the proof of Lemma~\ref{SubLemma_2: <lambda+beta,tau>}. Beforehand, we show by computation that $\langle \lambda+\gamma^\vee,\alpha\rangle \geq -1$. If $\langle \gamma^\vee,\beta\rangle =1=-\langle \gamma^\vee,\alpha\rangle$ it is clear since $\lambda$ is dominant, else if $\langle \gamma^\vee,\beta\rangle \geq2$,
    \begin{align*}\langle \lambda+\gamma^\vee,\alpha\rangle &= 
        \langle \lambda + \beta^\vee-\alpha^\vee+(1-\langle \beta^\vee,\gamma \rangle)\gamma^\vee,\alpha\rangle \\
        &= \langle \lambda+\beta^\vee,\alpha\rangle+(1-\langle \beta^\vee,\gamma\rangle)\langle \gamma^\vee,\alpha \rangle -2 \\
        &=  \langle \lambda+\beta^\vee,\alpha\rangle+(\langle \beta^\vee,\gamma\rangle-1)\langle \gamma^\vee,\beta \rangle -2.\end{align*}
    Since $\langle \beta^\vee,\gamma\rangle \geq 2$ and $\langle \gamma^\vee,\beta\rangle \geq 2$, $(\langle \beta^\vee,\gamma\rangle-1)\langle \gamma^\vee,\beta\rangle \geq2$, and by assumption $\langle \lambda+\beta^\vee,\alpha\rangle \geq -1$. Thus, $\langle \lambda+\gamma^\vee,\alpha\rangle \geq -1$ either way.

We construct chains which are slight modifications of the ones constructed in the proof of Lemma~\ref{SubLemma_2: <lambda+beta,tau>}.
    \begin{enumerate}
        \item Suppose first that $\langle \lambda+\gamma^\vee,\alpha\rangle \geq 0$. Then we show that we have the following chains:
      \begin{align}\pi^{v(\lambda)} w &< &\pi^{v(\lambda+\gamma^\vee)}s_{v(\gamma)} w &<&\pi^{v(\lambda+\gamma^\vee+\alpha^\vee)} s_{v(\alpha)} s_{v(\gamma)} w &<& \pi^{v(\lambda+\beta^\vee)} s_{v(\beta)} w \label{eq : chains3}\\
      \pi^{v(\lambda)} w &< &\pi^{vs_\gamma(\lambda+\gamma^\vee)}s_{v(\gamma)} w &<&\pi^{vs_\gamma s_\alpha(\lambda+\gamma^\vee+\alpha^\vee)} s_{v(\gamma)} s_{v(\alpha)} w &<& \pi^{vs_\beta(\lambda+\beta^\vee)} s_{v(\beta)} w. \label{eq : chains4}\end{align}

      Indeed:
      \begin{enumerate}
        \item The coweight $\lambda$ is dominant and $\gamma \in \Phi_+$, so $\langle \lambda,\gamma\rangle \geq 0$ and Formula~\eqref{eq : inequality condition} (resp. ~\eqref{eq : inequality conditionbis}) applied with $(\mu_0,\alpha_0,v_0,w_0,m)=(\lambda,\gamma,v,w,1)$ proves the leftmost inequality in the chain~\eqref{eq : chains3} (resp.~\eqref{eq : chains4}). 
        \item Since $\langle \lambda+\gamma^\vee,\alpha\rangle \geq 0$, Formula~\eqref{eq : inequality condition} (resp. ~\eqref{eq : inequality conditionbis}) applied with $(\mu_0,\alpha_0,v_0,w_0,m)=(\lambda+\gamma^\vee,\alpha,v,s_{v(\gamma)}w,1)$ (resp. $(\mu_0,\alpha_0,v_0,w_0,m)=(\lambda+\gamma^\vee,\alpha,vs_\gamma,s_{v(\gamma)}w,1)$) proves the second inequality in the chain~\eqref{eq : chains3} (resp.~\eqref{eq : chains4}).
                
        \item Note that $\lambda+\beta^\vee = (\lambda+\gamma^\vee+\alpha^\vee)+(\langle\beta^\vee,\gamma\rangle-1)\gamma^\vee$. Moreover $0<\langle \beta^\vee,\gamma\rangle -1$ and $-\langle \lambda+\gamma^\vee+\alpha^\vee,\gamma\rangle=\langle \beta^\vee,\gamma\rangle-\langle \lambda,\gamma\rangle -2<\langle \beta^\vee,\gamma\rangle -1$. Therefore by applying Formula~\eqref{eq : inequality condition} (resp. ~\eqref{eq : inequality conditionbis}) to $(\mu_0,\alpha_0,v_0,w_0,m)=(\lambda+\gamma^\vee+\alpha^\vee,\gamma,v,s_{v(\alpha)}s_{v(\gamma)}w,\langle \beta^\vee,\gamma\rangle -1)$ (resp. $(\mu_0,\alpha_0,v_0,w_0,m)=(\lambda+\gamma^\vee+\alpha^\vee,\gamma,vs_\gamma s_\alpha,s_{v(\gamma)}s_{v(\alpha)}w,\langle \beta^\vee,\gamma\rangle -1)$) we obtain the rightmost inequality in the chain~\eqref{eq : chains3} (resp.~\eqref{eq : chains4}). 
    \end{enumerate}
    \item We now suppose that $\langle \lambda+\gamma^\vee,\alpha\rangle =-1$. Then $\lambda+\gamma^\vee+\alpha^\vee=s_\alpha(\lambda+\gamma^\vee)$ and the above chains do not always hold. We focus here on the case of $\pi^{v(\lambda+\beta^\vee)}s_{v(\beta)} w$.
    \begin{enumerate}
        \item If $v(\alpha)\in \Phi_+$, since $\langle \lambda+\gamma^\vee,\alpha\rangle <0$, the inequality $\pi^{vs_\alpha(\lambda+\gamma^\vee)}s_{v(\alpha)}s_{v(\gamma)}w>\pi^{v(\lambda+\gamma^\vee)}s_{v(\gamma)}w$ still holds, by Formula~\eqref{eq: leftaction} applied with $s_{v(\alpha)[0]}$. The rest of the chain~\eqref{eq : chains3} still holds and the whole chain remains correct.        
        \item If $v(\alpha)\in \Phi_-$, then $vs_\alpha <v$, and we have a chain
        \begin{equation}\pi^{v(\lambda)}w<\pi^{vs_\alpha(\lambda)}s_{v(\alpha)}w<\pi^{vs_\alpha(\lambda+\gamma^\vee)}s_{v(\alpha)}s_{v(\tau)}w=\pi^{v(\lambda+\gamma^\vee+\alpha^\vee)}s_{v(\alpha)}s_{v(\tau)}w<\pi^{v(\lambda+\beta^\vee)}s_{v(\beta)}w.\end{equation}
        The reflection used for the first inequality is $s_{-v(\alpha)[0]}$, and it holds by Formula~\eqref{eq: leftaction} because $\langle v(\lambda),-v(\alpha)\rangle = -\langle \lambda,\alpha\rangle<0$. Note that this is non-zero because $v$ is the minimal representative of $vW_\lambda$ and thus $vs_\alpha<v$ implies $s_\alpha \notin W_\lambda$ so $\langle \lambda,\alpha\rangle \neq 0$. For the second and the third inequalities we use Formula~\eqref{eq : inequality condition} with $(\mu_0,\alpha_0,v_0,w_0,m)$ equal to $(\lambda,\gamma,vs_\alpha,s_{v(\alpha)}w,1)$ and $(\lambda+\alpha^\vee+\gamma^\vee,\gamma,v,s_{v(\alpha)}s_{v(\gamma)}w,\langle \beta^\vee,\gamma\rangle-1)$ respectively.
    \end{enumerate}
    \item We suppose that $\langle \lambda+\gamma^\vee,\alpha\rangle = -1 $ and we deal with the case of $\pi^{vs_{v(\beta)}(\lambda+\beta^\vee)}s_{v(\beta)}w$. Then \begin{equation}\pi^{vs_\gamma s_\alpha(\lambda+\gamma^\vee+\alpha^\vee)}s_{v(\gamma)}s_{v(\alpha)}w=\pi^{vs_\gamma(\lambda+\gamma^\vee)}s_{vs_\gamma(\alpha)}s_{v(\gamma)}w = s_{vs_\gamma(\alpha)[\langle \lambda +\gamma^\vee,\alpha\rangle]}\pi^{vs_\gamma(\lambda+\gamma^\vee)}s_{v(\gamma)}w.\end{equation} 
    Moreover $(s_{v(\gamma)}w)^{-1}(vs_\gamma(\alpha))=w^{-1}v(\alpha)$. Thus, since $\langle \lambda+\gamma^\vee,\alpha\rangle<0$,
    \begin{enumerate}
        \item If $w^{-1}v(\alpha) \in \Phi_-$, by Formula~\eqref{eq: leftaction}, $\pi^{vs_\gamma(\lambda+\gamma^\vee)}s_{v(\gamma)}w<\pi^{vs_\gamma s_\alpha(\lambda+\gamma^\vee+\alpha^\vee)}s_{v(\gamma)}s_{v(\alpha)} w$ and the chain~\eqref{eq : chains4} still holds.
        \item If $w^{-1}v(\alpha) \in \Phi_+$, then, since $\langle \lambda,\alpha\rangle =\langle \gamma^\vee,\beta\rangle -1> 0$, by Formula~\eqref{eq: leftaction}, $\pi^{v(\lambda)}w<s_{v(\alpha)[\langle\lambda,\alpha\rangle]}\pi^{v(\lambda)}w=\pi^{v(\lambda)}s_{v(\alpha)}w$. Then by Formula~\eqref{eq : inequality conditionbis} applied with $(\mu_0,\alpha_0,v_0,w_0,m)=(\lambda,\gamma,v,s_{v(\alpha)}w,1)$ we have $$\pi^{v(\lambda)}s_{v(\gamma)}s_{v(\beta)}s_{v(\gamma)}w<\pi^{vs_\gamma(\lambda+\gamma^\vee)}s_{v(\beta)}s_{v(\gamma)}w=\pi^{vs_\gamma s_\alpha (\lambda+\gamma^\vee+\alpha^\vee)}s_{v(\gamma)}s_{v(\alpha)} w$$ and we have a chain
        \begin{equation}\pi^{v(\lambda)}w < \pi^{v(\lambda)}s_{v(\alpha)}w < \pi^{vs_\gamma s_\alpha(\lambda+\gamma^\vee+\alpha^\vee)} s_{v(\gamma)}s_{v(\alpha)}w<\pi^{vs_\beta(\lambda+\beta^\vee)}s_{v(\beta)}w.\end{equation}
    \end{enumerate}
    \end{enumerate}

\end{proof}

\begin{Lemma}\label{SubLemma : lengthexpression}
Let $\beta \in \Phi_+$ be a positive root and suppose that there exists $\gamma \in \Inv(s_\beta)\setminus\{\beta\}$ such that $\langle \beta^\vee,\gamma\rangle\geq2$ and $\langle \beta^\vee,\gamma\rangle\langle \gamma^\vee,\beta\rangle \neq 3$. Then $\gamma$ can be chosen such that $\beta \notin \Inv(s_\gamma)$. 
    
\end{Lemma}

\begin{proof}
Note that, by \cite[Lemma 1.1.10]{bardy1996systemes}, for any $\beta,\gamma\in \Phi$, $\langle \beta^\vee,\gamma\rangle$ and $\langle \gamma^\vee,\beta\rangle$ have the same sign, so if $\langle \beta^\vee,\gamma\rangle \geq 2$ and $\langle \beta^\vee,\gamma\rangle \langle \gamma^\vee,\beta\rangle \neq 3$, either $\langle \beta^\vee,\gamma\rangle \langle \gamma^\vee\rangle \geq 4$, either $\langle \beta^\vee,\gamma\rangle =2$ and $\langle \gamma^\vee,\beta\rangle =1$. 
We treat separately these cases:
\begin{enumerate}
    \item Let us first suppose that there exists $\gamma \in \Inv(s_\beta)$ such that $\langle \beta^\vee,\gamma\rangle = 2$ and $\langle \gamma^\vee,\beta\rangle =1$. Suppose that $\beta \in \Inv(s_\gamma)$, so $s_\gamma(\beta)=\beta-\gamma <0$, and $s_\beta(\gamma)=\gamma-2\beta <0$. Then we show that $\beta \notin \Inv(s_{\tilde \gamma})$ for $\tilde \gamma = -s_\beta(\gamma)$:
    \begin{align*}s_{\tilde \gamma}(\beta) =s_\beta s_\gamma s_\beta(\beta) = -s_\beta(\beta-\gamma) =\gamma-\beta=-s_\gamma(\beta)>0.\end{align*}
Moreover $s_{\beta}(\tilde \gamma)=-\gamma < 0$ and $\langle \beta^\vee,\tilde \gamma\rangle = \langle \beta^\vee,\gamma\rangle =2$ therefore, $\gamma$ can be chosen such that $\beta \notin \Inv(s_\gamma)$.
    
   \item Let us now suppose that there exists $\gamma \in \Inv(s_\beta)$ such that $\langle \beta^\vee,\gamma\rangle \geq 2$ and  $\langle \beta^\vee,\gamma\rangle \langle \gamma^\vee,\beta\rangle \geq 4$. Write $\beta=v_\beta(\beta_0)=s_{\alpha_1}\ldots s_{\alpha_n}(\beta_0)$ where the $\alpha_i$ and $\beta_0$ are simple roots, and suppose that $n$ is of minimal length amongst possible expressions of $\beta$. Therefore $s_{\alpha_1}\ldots s_{\alpha_n}s_{\beta_0}s_{\alpha_n}\ldots s_{\alpha_1}$ is a reduced expression of $s_\beta$ and
    \begin{equation}\Inv(s_{\beta})= \{s_{\alpha_1}\ldots s_{\alpha_{p-1}}(\alpha_p)\mid p\leq {n}\}\sqcup \{\beta\}\sqcup \{s_{\alpha_1}\ldots s_{\alpha_n}s_{\beta_0}s_{\alpha_{n}}\ldots s_{\alpha_{n+1-p}}(\alpha_{n-p})\mid p\leq n\}.\end{equation}
   Let $k$ be the smallest such that $\gamma_k=s_{\alpha_1}\ldots s_{\alpha_{k-1}}(\alpha_k)$ verifies $\langle \beta^\vee,\gamma_k\rangle \geq 2$ and $\langle \beta^\vee,\gamma_k\rangle \langle \gamma^\vee_k,\beta\rangle \geq 4$. 

    The expression $s_{\alpha_1}\ldots s_{\alpha_{k-1}}s_{\alpha_k}s_{\alpha_{k-1}}\ldots s_{\alpha_1}$ is an expression of $s_{\gamma_k}$, thus
    \begin{equation}\Inv(s_{\gamma_k})\subset \{s_{\alpha_1}\ldots s_{\alpha_{p-1}}(\alpha_p)\mid p\leq {k-1}\}\sqcup \{\gamma_k\}\sqcup \{s_{\alpha_1}\ldots s_{\alpha_k}s_{\alpha_{k-1}}\ldots s_{\alpha_{k+1-p}}(\alpha_{k-p})\mid p\leq k-1\}.\end{equation}
    Suppose by contradiction that $\beta \in \Inv(s_{\gamma_k})$. Since $v_\beta$ is of minimal length, $\beta$ is not in the first set, hence there is $p\in\llbracket 1,k-1\rrbracket$ such that $\beta=s_{\alpha_1}\ldots s_{\alpha_k}s_{\alpha_{k-1}}\ldots s_{\alpha_{k+1-p}}(\alpha_{k-p})$. 
    
    We show that $\gamma_{k-p}=s_{\alpha_1}\ldots s_{\alpha_{k-p-1}}(\alpha_{k-p})\in \Inv(s_\beta)$ verifies $\langle \beta^\vee,\gamma_{k-p}\rangle \geq 2$, which contradicts the minimality of $k$. Note that $\beta=-s_{\gamma_k}(\gamma_{k-p})$
    We compute:
    \begin{align*}\langle \beta^\vee,\gamma_{k-p}\rangle &= \langle -s_{\gamma_k}(\gamma_{k-p}^\vee),\gamma_{k-p}\rangle \\ &= -(2 - \langle \gamma_{k-p}^\vee,\gamma_k\rangle \langle \gamma^\vee_k,\gamma_{k-p}\rangle) \\ &= \langle \beta^\vee,\gamma_k\rangle \langle \gamma_k^\vee,\beta\rangle-2.
    \end{align*} 
    So since $\langle \beta^\vee,\gamma_k\rangle \langle \gamma_k^\vee,\beta\rangle \geq 4$, we get $\langle \beta^\vee,\gamma_{k-p}\rangle \geq 2$, and with a similar computation, we find that $\langle \gamma_{k-p}^\vee,\beta\rangle = \langle \beta^\vee,\gamma_{k-p}\rangle \geq 2$ as well, so $\langle \beta^\vee,\gamma_{k-p}\rangle \langle \gamma_{k-p}^\vee,\beta\rangle \geq 4$. This contradicts the minimality of $k$ and thus $\beta \notin \Inv(s_{\gamma_k})$.
    \end{enumerate}
\end{proof}

\begin{Lemma}\label{Lemma : height is length G2}
    Let $\lambda \in Y^{++}$, $v\in W^\lambda$ and $w\in W$. Let $\beta \in \Phi_+$ and let $\gamma\in \Inv(s_\beta)$ be such that $\beta \in \Inv(s_\gamma)$ and $\langle \beta^\vee,\gamma\rangle =3,\; \langle \gamma^\vee,\beta\rangle =1$.
    
    Then $\pi^{v(\lambda+\beta^\vee)}s_{v(\beta)}w$ and $\pi^{vs_\beta(\lambda+\beta^\vee)}s_{v(\beta)}w$ do not cover $\pi^{v(\lambda)}w$.
\end{Lemma}
\begin{proof}
     We show that, with the assumptions of the statement, $\beta$ and  $\gamma$ appear as positive roots of a root subsytem of $\Phi$ isomorphic to $G_2$, and we use this system to construct chains replacing the ones in the proof of Lemma~\ref{Lemma : height is length}.

    First, note that $-s_\gamma(\beta)$ lies in $\Inv(s_\beta)$ (so $s_\beta s_\gamma(\beta)$ is positive). Indeed, we can write, as in the proof of Lemma~\ref{SubLemma : lengthexpression}, $\beta=s_{\alpha_1}\ldots s_{\alpha_n}(\beta_0)$ for a minimal $n$, and $\gamma=s_{\alpha_1}\ldots s_{\alpha_{k-1}}(\alpha_k)$ for some $k \leq n$. Then, since $\beta \in \Inv(s_\gamma)$, $\beta$ is also of the form $s_{\alpha_1}\ldots s_{\alpha_k}s_{\alpha_{k-1}}\ldots s_{\alpha_{k-p+1}}(\alpha_{k-p})$ for some $p\leq k-1$, and thus $-s_\gamma(\beta)=s_{\alpha_1}\ldots s_{\alpha_{k-p-1}}(\alpha_{k-p})\in \Inv(s_\beta)$.
    Therefore we have the following positive roots, and their associated coroots (the notation will become clear afterwards):
    \begin{enumerate}
        \item $\theta_1:=-s_\gamma(\beta)=\gamma-\beta \in \Phi_+$, with associated coroot $\theta_1^\vee = -s_\gamma(\beta^\vee)=3\gamma^\vee-\beta^\vee$.
        \item $\tilde \beta:=-s_\beta(\gamma)=3\beta-\gamma\in \Phi_+$, with associated coroot $\tilde \beta^\vee=-s_\beta(\gamma^\vee)=\beta^\vee-\gamma^\vee$.
        \item $\tilde \gamma:= s_\beta s_\gamma(\beta) = 2\beta-\gamma \in \Phi_+$, with associated coroot $\tilde \gamma^\vee=s_\beta s_\gamma(\beta^\vee)=2\beta^\vee-3\gamma^\vee$.
    \end{enumerate}
    Let us also denote $\theta_2=s_{\theta_1}(\gamma)=3\beta-2\gamma$, with associated coroot $\theta_2^\vee=\beta^\vee-2\gamma^\vee$. Then one can check that $\{\theta_1,\theta_2\}$ form the positive simple roots of a $G_2$ root system (in the sense that $\langle \theta_1^\vee,\theta_2\rangle =-3$ and $\langle \theta_2^\vee,\theta_1\rangle =-1$), such that $\gamma=s_{\theta_1}(\theta_2)$, $\beta=s_{\theta_1}s_{\theta_2}(\theta_1)$, $\tilde \gamma = s_{\theta_2}(\theta_1)$ and $\tilde \beta =s_{\theta_2}s_{\theta_1}(\theta_2)$. However, $\theta_2$ may not be a positive root in $\Phi$, and we thus need to distinguish these two cases.

    Let us first suppose that $\theta_2$ lies in $\Phi_+$. Notice that \begin{equation}\label{eq: G2case1}\theta_1^\vee+\tilde \beta^\vee+\theta_2^\vee=(3\gamma^\vee-\beta^\vee)+(\beta^\vee-\gamma^\vee)+(\beta^\vee-2\gamma^\vee)=\beta^\vee,\end{equation} and \begin{equation}\label{eq: G2case2}s_{\theta_1}s_{\tilde \beta} s_{\theta_2}=s_{\theta_1}(s_{\theta_2}s_{\theta_1}s_{\theta_2}s_{\theta_1}s_{\theta_2})s_{\theta_2}=s_{\theta_1}s_{\theta_2}s_{\theta_1}s_{\theta_2}s_{\theta_1}=s_{\theta_2}s_{\tilde \beta} s_{\theta_1}=s_\beta.\end{equation}

    Moreover, we have \begin{align}\langle \theta_2^\vee,\tilde \beta\rangle &= \langle \beta^\vee-2\gamma^\vee,3\beta-\gamma\rangle =1>0 \label{eqg21}\\ \langle \theta_2^\vee+\tilde \beta^\vee,\theta_1\rangle &=\langle 2\beta^\vee-3\gamma^\vee,\gamma-\beta\rangle=-1.\label{eqg22}\end{align}

    \begin{enumerate}
        \item Suppose first that $\langle \lambda,\theta_1\rangle >0$. Since $\lambda$ is dominant and by Computation~\eqref{eqg21},\eqref{eqg22}, $-\langle \lambda+\theta_2^\vee,\tilde \beta\rangle<0$ and  $-\langle\lambda+\theta_2^\vee+\tilde \beta^\vee,\theta_1 \rangle\leq 0$. Using Formula~\eqref{eq : inequality condition} (resp. ~\eqref{eq : inequality conditionbis}) with $(\mu_0,\alpha_0,m)=(\lambda,\theta_2,1)$ for the first inequality, $(\lambda+\theta_2^\vee,\tilde \beta,1)$ for the second and $(\lambda+\theta_2^\vee+\tilde \beta^\vee,\theta_1,1)$ for the third (recall Formulas~\eqref{eq: G2case1},~\eqref{eq: G2case2}), we obtain the chain~\eqref{chain7} (resp.~\eqref{chain8})
        \begin{align}
            \pi^{v(\lambda)}w &< \pi^{v(\lambda+\theta_2^\vee)}s_{v(\theta_2)}w &<& \pi^{v(\lambda+\theta_2^\vee+\tilde \beta^\vee)}s_{v(\tilde\beta)}s_{v(\theta_2)}w &<& \pi^{v(\lambda+\beta^\vee)}s_{v(\beta)}w \label{chain7}\\
            \pi^{v(\lambda)}w &< \pi^{vs_{\theta_2}(\lambda+\theta_2^\vee)}s_{v(\theta_2)}w &<& \pi^{vs_{\theta_2}s_{\tilde \beta}(\lambda+\theta_2^\vee+\tilde \beta^\vee)}s_{v(\theta_2)}s_{v(\tilde\beta)}w &<& \pi^{vs_\beta(\lambda+\beta^\vee)}s_{v(\beta)}w. \label{chain8}
        \end{align}
        \item If $\langle\lambda,\theta_1\rangle=0$, then $-\langle \lambda+\theta_2^\vee+\tilde \beta^\vee,\theta_1\rangle = 1$ so the last inequality in the chains~\eqref{chain7} and~\eqref{chain8} do not always hold, we have the following case distinction, which we already encountered in Lemma~\ref{SubLemma_2: <lambda+beta,tau>} and Lemma~\ref{Lemma : height is length}:
        \begin{enumerate}
            \item If $v(\theta_1)\in \Phi_+$, the chain~\eqref{chain7} still holds, else $vs_{\theta_1}<v$, $\lambda+\beta^\vee=s_{\theta_1}(\lambda+\theta_2^\vee+\tilde\beta)$ and we instead have the chain
            \begin{equation}\pi^{v(\lambda)}w < \pi^{vs_{\theta_1}(\lambda)}s_{v(\theta_1)}w<\pi^{vs_{\theta_1}(\lambda+\theta_2^\vee)}s_{v(\theta_1)}s_{v(\theta_2)}w < \pi^{vs_{\theta_1}(\lambda+\theta_2^\vee+\tilde \beta^\vee)}s_{v(\theta_1)}s_{v(\tilde\beta)}s_{v(\theta_2)}w \end{equation}where the last term is actually equal to $\pi^{v(\lambda+\beta^\vee)}s_{v(\beta)}w$.
            \item If $w^{-1}v(\theta_1)\in \Phi_-$, then since $\langle \lambda+\theta_2^\vee+\tilde\beta^\vee,\theta_1\rangle <0$, by Formula~\eqref{eq: leftaction} applied with the affinized root $vs_{\theta_2}s_{\tilde\beta}(\theta_1)[\langle \lambda+\theta_2^\vee+\tilde\beta^\vee,\theta_1\rangle]$, the third inequality of Chain~\eqref{chain8} still holds, and thus the whole chain remains correct. Otherwise if $w^{-1}v(\theta_1)\in \Phi_+$ we instead have the chain
            \begin{equation}\pi^{v(\lambda)}w <\pi^{v(\lambda)}s_{v(\theta_1)}w<  \pi^{vs_{\theta_2}(\lambda+\theta_2^\vee)}s_{v(\theta_2)}s_{v(\theta_1)}w <\pi^{vs_{\theta_2}s_{\tilde\beta}(\lambda+\theta_2^\vee+\tilde\beta^\vee)}s_{v(\theta_2)}s_{v(\tilde\beta)}s_{v(\theta_1)}w\end{equation} where the last term is actually equal to $\pi^{vs_\beta(\lambda+\beta^\vee)}s_{v(\beta)}w$ since $\lambda+\theta_2^\vee+\tilde\beta^\vee = s_{\theta_1}(\lambda+\beta^\vee)$.
        \end{enumerate}
    
    \end{enumerate}
    We now turn to the case of $\theta_2 \in \Phi_-$. Notice that $\beta^\vee = -\theta_2^\vee+\tilde\gamma^\vee +\gamma^\vee$ and $s_\beta = s_\gamma s_{\tilde \gamma}  s_{\theta_2}=s_{\theta_2}s_{\tilde \gamma}s_\gamma$. Moreover, $\langle -\theta_2^\vee,\tilde \gamma \rangle = \langle 2\gamma^\vee-\beta^\vee,2\beta-\gamma\rangle = -1$ and $\langle -\theta_2^\vee+\tilde \gamma^\vee,\gamma\rangle = \langle \beta^\vee-\gamma^\vee,\gamma\rangle =1$. Therefore, since $\lambda$ is dominant and $-\theta_2$ is a positive root:
    \begin{enumerate}
        \item If $\langle \lambda,\tilde \gamma\rangle >0$, then using Formula~\eqref{eq : inequality condition} (resp. ~\eqref{eq : inequality conditionbis}) with $(\mu_0,\alpha_0,m)=(\lambda,-\theta_2,1)$ for the first inequality, $(\lambda-\theta_2^\vee,\tilde \gamma,1)$ for the second and $(\lambda-\theta_2^\vee+\tilde \gamma^\vee,\gamma,1)$ for the third, we obtain the chain~\eqref{chain9} (resp.~\eqref{chain10})
        \begin{align}
            \pi^{v(\lambda)}w &< \pi^{v(\lambda-\theta_2^\vee)}s_{v(\theta_2)}w &< \pi^{v(\lambda-\theta_2^\vee+\tilde \gamma^\vee)}s_{v(\tilde\gamma)}s_{v(\theta_2)}w &< \pi^{v(\lambda+\beta^\vee)}s_{v(\beta)}w \label{chain9}\\
            \pi^{v(\lambda)}w &< \pi^{vs_{\theta_2}(\lambda-\theta_2^\vee)}s_{v(\theta_2)}w &< \pi^{vs_{\theta_2}s_{\tilde \gamma}(\lambda-\theta_2^\vee+\tilde \gamma^\vee)}s_{v(\theta_2)}s_{v(\tilde\gamma)}w &< \pi^{vs_\beta(\lambda+\beta^\vee)}s_{v(\beta)}w. \label{chain10}
        \end{align}
        \item Suppose now that $\langle \lambda,\tilde \gamma\rangle = 0$, so $\lambda-\theta_2^\vee+\tilde \gamma = s_{\tilde \gamma}(\lambda-\theta_2^\vee)$. Then
        \begin{enumerate}
            \item If $v(\tilde \gamma)\in\Phi_+$, the chain~\eqref{chain9} still holds. Else, $vs_{\tilde \gamma}<v$ and we instead have the chain
            \begin{equation}\pi^{v(\lambda)}w < \pi^{vs_{\tilde \gamma}(\lambda)}s_{v(\tilde \gamma)}w < \pi^{vs_{\tilde \gamma}(\lambda-\theta_2^\vee)}s_{v(\tilde\gamma)}s_{v(\theta_2)}w < \pi^{v(\lambda+\beta^\vee)}s_{v(\beta)}w\end{equation}
            where the first inequality comes from Proposition~\ref{affine_to_relative0} and the two others from Formula~\eqref{eq : inequality condition}.
            \item If $w^{-1}v(\tilde \gamma)\in \Phi_-$, then the chain~\eqref{chain10} still holds. Else $w^{-1}v(\tilde \gamma)\in \Phi_+$ and we instead have the chain
            \begin{equation}\pi^{v(\lambda)}w<\pi^{v(\lambda)}s_{v(\tilde \gamma)}w < \pi^{vs_{\theta_2}(\lambda-\theta_2^\vee)}s_{v(\theta_2)}s_{v(\tilde \gamma)}w < \pi^{vs_\beta(\lambda+\beta^\vee)}s_{v(\beta)}w\end{equation} where the first inequality is deduced from Formula~\eqref{eq: leftaction} used with the affinized root $v(\tilde \gamma)[\langle \lambda,\tilde \gamma\rangle]$, and the two others from Formula~\eqref{eq : inequality conditionbis} as for the chain~\eqref{chain10}.
        \end{enumerate}
    \end{enumerate}
  \end{proof}

  \begin{Proposition}\label{Proposition : height is length}
     Let $\lambda \in Y^{++}$, $v\in W^\lambda$ and $w\in W$. Let $\beta\in \Phi_+$ and suppose that $\pi^{v(\lambda+\beta^\vee)}s_{v(\beta)}w$ or $\pi^{vs_\beta(\lambda+\beta^\vee)}s_{v(\beta)}w$ cover $\pi^{v(\lambda)}w$. Then $\beta$ is a quantum root.
\end{Proposition}
\begin{proof}
    We prove the converse, suppose that $\beta$ is not a quantum root. By Lemma~\ref{Lemma: quantum roots}, there is $\gamma\in \Inv(s_\beta)\setminus\{\beta\}$ such that $\langle \beta^\vee,\gamma\rangle \geq 2$. If $\beta\notin \Inv(s_\gamma)$ we apply Lemma~\ref{Lemma : height is length}. We can also apply it in case $\langle \beta^\vee,\gamma\rangle \langle \gamma^\vee,\beta\rangle\neq3$ by Lemma~\ref{SubLemma : lengthexpression}. Finally if $\beta\in \Inv(s_\gamma)$ and $\langle\beta^\vee,\gamma\rangle \langle \gamma^\vee,\beta\rangle =3$ we apply Lemma~\ref{Lemma : height is length G2}. Therefore if $\beta$ is not a quantum root then $\pi^{v(\lambda+\beta^\vee)}s_{v(\beta)}w$ and $\pi^{vs_\beta(\lambda+\beta^\vee)}s_{v(\beta)}w$ do not cover $\pi^{v(\lambda)}w$.
\end{proof}

\subsection{Conclusion}\label{subsection: conclusion}
We now have everything to prove Theorem~\ref{Theorem : maintheoremintro}:
\setcounter{Theoremintro}{0}
\begin{Theoremintro}\label{Theorem: covers general case}
Suppose that $\by,\bx\in W^a_+$ are such that $\bx\leq \by$. Then 
\begin{equation}\bx\lhd \by\iff \ell^a(\by)=\ell^a(\bx)+1.\end{equation}
\end{Theoremintro}
\begin{proof}
    If $\bx\leq \by$ with $\ell^a(\by)=\ell^a(\bx)+1$, then by strict compatibility of $\ell^a$ (Theorem~\ref{Theorem Muthiah Orr}), $\by$ covers $\bx$. Conversely suppose that $\by$ covers $\bx$. If $\by$ and $\bx$ have same dominance class then Theorem~\ref{non_reg_covers_same_orbit} implies that $\ell^a(\by)=\ell^a(\bx)+1$. Else, if $\pr^{Y^+}(\by)\notin W\cdot\pr^{Y^+}(\bx)$, by Proposition~\ref{n_is_small}, $\by$ is of the form $\pi^{v(\lambda+\beta^\vee)}s_{v(\beta)}w$ or $\pi^{vs_\beta(\lambda+\beta^\vee)}s_{v(\beta)}w$, for $\bx=\pi^{v(\lambda)}w$ with $\lambda \in Y^{++}$, $v\in W^\lambda$, $w\in W$ and $\beta \in \Phi_+$. Then, by Corollary~\ref{Corollary: simplification}, we have $\ell^a(\by)-\ell^a(\bx)=2\htt(\beta^\vee)-\ell(s_\beta)$. Moreover, by Proposition~\ref{Proposition : height is length}, $\beta$ is a quantum root and therefore in this case as well:
    $$\ell^a(\by)-\ell^a(\bx)=1.$$

\end{proof}
Along the way, we have obtained a classification of covers, which we summarize in Proposition~\ref{Proposition classification general}. This is to be compared with~\cite[Proposition 4.5]{schremmer2023affine}.

\begin{Proposition}\label{Proposition classification general}
    Let $\bx=\pi^{v(\lambda)}w\in W^a_+$ with $\lambda\in Y^{++}$, $v\in W^\lambda$ and $w\in W$. Let $J\subset S$ be the set of simple reflections such that $W_\lambda=W_J$, recall Notation~\ref{Notation : paraboliccosets} and Definition~\ref{Definition quantum}. Then covers of $\bx$ are the elements of the following form. 
    \begin{enumerate}
        \item $\pi^{v(\lambda)}s_{v(\beta)}w=\bx s_{w^{-1}v(\beta)[0]}$ for $\beta \in \Phi$ such that $\ell(s_\beta v^{-1}w)=\ell(v^{-1}w)+1$.
        \item $\pi^{vs_\beta(\lambda)}s_{v(\beta)} w=s_{v(\beta)[0]}\bx$ for $\beta\in \Phi_+$ such that:
        \begin{enumerate}
            \item $\langle \lambda,\beta\rangle \neq 0$
            \item $\ell(vs_\beta)=\ell(v)-1$
            \item If $u$ denotes $vs_\beta$ and $u_J$ the maximal $W_J$-suffix of $u$, then $vu_J^{-1}$  is on a minimal gallery from $v$ to $w$.
        \end{enumerate}
        \item $\pi^{v(\lambda+\beta^\vee)}s_{v(\beta)}w=s_{v(\beta)[\langle \lambda,\beta\rangle+1]}\bx=\bx s_{w^{-1}v(\beta)[1]}$ for $\beta\in \Phi_+$ such that:
        \begin{enumerate}
            \item $\beta$ is a quantum root
            \item $\lambda+\beta^\vee$ is an almost dominant coweight
            \item For $u=v^{\lambda+\beta^\vee}$, $v$ is on a minimal gallery from $1$ to $vu$, that is to say $\ell(vu)=\ell(v)+\ell(u)$
            \item For $\tilde v=v^{v(\lambda+\beta^\vee)}$, $s_{v(\beta)}\tilde v$ is on a minimal gallery from $v$ to $w$.
        \end{enumerate}
        \item $\pi^{vs_\beta(\lambda+\beta^\vee)}s_{v(\beta)}w=s_{v(\beta)[-1]}\bx$ for $\beta\in \Phi_+$ such that:
        \begin{enumerate}
            \item $\beta$ is a quantum root
            \item $\lambda+\beta^\vee$ is an almost dominant coweight
            \item For $u=v^{\lambda+\beta^\vee}$, $v$ is on a minimal gallery from $1$ to $vs_\beta u$
            \item For $\tilde v=v^{vs_\beta(\lambda+\beta^\vee)}$, $s_{v(\beta)}\tilde v$ is on a minimal gallery from $v$ to $w$.
        \end{enumerate}
    \end{enumerate}
\end{Proposition}
In particular, suppose that $\lambda\in Y^{++}$ is regular and is such that $\lambda+\beta^\vee$ is also regular for any quantum root $\beta\in \Phi_+$, we then say that $\lambda$ is superregular. Proposition~\ref{Proposition classification general} can be simplified for superregular coweights. This is to be compared with~\cite[Proposition 4.4]{lam2010quantum} and~\cite[Theorem 2]{welch2022classification}.
\begin{Proposition}
    Let $\bx=\pi^{v(\lambda)}w\in W^a_+$ with $\lambda\in Y^{++}$ a superregular coweight and $v,w\in W$. Then covers of $\bx$ are the elements of the following form.

    \begin{enumerate}
        \item $\bx s_{\beta[0]}=\pi^{v(\lambda)}ws_\beta$ for $\beta \in \Phi_+$ such that $\ell(v^{-1}ws_\beta)=\ell(v^{-1}w)+1$.
        \item $s_{\beta[0]}\bx=\pi^{s_\beta v(\lambda)}s_{\beta}w$ for $\beta \in \Phi_+$ such that $\ell(s_\beta v)=\ell(v)-1$.
        \item $\bx s_{w^{-1}v(\beta)[1]}=\pi^{v(\lambda+\beta^\vee)}s_{v(\beta)}w$ for $\beta\in \Phi_+$ a quantum root such that $\ell(v^{-1}w)=\ell(s_\beta)+\ell(s_\beta v^{-1}w)$ (otherwise said $s_\beta v^{-1}w \leq_R v^{-1}w$).
        \item $s_{v(\beta)[-1]}\bx=\pi^{vs_\beta(\lambda+\beta^\vee)}s_{v(\beta)}w$ for $\beta\in \Phi_+$ a quantum root such that $\ell(vs_\beta)=\ell(v)+\ell(s_\beta)$ (otherwise said $s_\beta \leq_R vs_\beta$).
    \end{enumerate}
\end{Proposition}
For Kac-Moody root systems, the existence of superregular coweights is not clear a priori. However in an upcoming joint work with Hébert we prove that any Kac-Moody root system admits a finite number of quantum roots, which ensures the existence of superregular coweights. We also use this finiteness to deduce that any element of $W^+_a$ admits a finite number of covers; in particular intervals in $W^+_a$ are finite.
\bibliography{covers.bib}

\newcommand{\etalchar}[1]{$^{#1}$}
\begin{thebibliography}{BPGR16}

\bibitem[Bar96]{bardy1996systemes}
Nicole Bardy.
\newblock Syst\`emes de racines infinis.
\newblock {\em M\'em. Soc. Math. Fr. (N.S.)}, (65):vi+188, 1996.

\bibitem[BB05]{bjorner2005combinatorics}
Anders Bj\"{o}rner and Francesco Brenti.
\newblock {\em Combinatorics of {C}oxeter groups}, volume 231 of {\em Graduate Texts in Mathematics}.
\newblock Springer, New York, 2005.

\bibitem[BKP16]{braverman2016iwahori}
Alexander Braverman, David Kazhdan, and Manish~M. Patnaik.
\newblock Iwahori-{H}ecke algebras for {$p$}-adic loop groups.
\newblock {\em Invent. Math.}, 204(2):347--442, 2016.

\bibitem[BPGR16]{bardy2016iwahori}
Nicole Bardy-Panse, St\'ephane Gaussent, and Guy Rousseau.
\newblock Iwahori-{H}ecke algebras for {K}ac-{M}oody groups over local fields.
\newblock {\em Pacific J. Math.}, 285(1):1--61, 2016.

\bibitem[BPHR22]{twinmasures}
Nicole Bardy-Panse, Auguste Hebert, and Guy Rousseau.
\newblock Twin masures associated with {K}ac-{M}oody groups over {L}aurent polynomials.
\newblock {\em arXiv preprint arXiv:2210.07603}, 2022.

\bibitem[CFG21]{chaput2021parametrization}
Pierre-Emmanuel Chaput, Lucas Fresse, and Thomas Gobet.
\newblock {Parametrization, structure and Bruhat order of certain spherical quotients}.
\newblock {\em {Representation Theory. An Electronic Journal of the American Mathematical Society}}, 25:935--974, October 2021.

\bibitem[IM65]{iwahori1965bruhat}
N.~Iwahori and H.~Matsumoto.
\newblock On some {B}ruhat decomposition and the structure of the {H}ecke rings of {${ p}$}-adic {C}hevalley groups.
\newblock {\em Inst. Hautes \'Etudes Sci. Publ. Math.}, (25):5--48, 1965.

\bibitem[KL80]{kazhdan1980schubert}
David Kazhdan and George Lusztig.
\newblock Schubert varieties and {P}oincar\'{e} duality.
\newblock In {\em Geometry of the {L}aplace operator ({P}roc. {S}ympos. {P}ure {M}ath., {U}niv. {H}awaii, {H}onolulu, {H}awaii, 1979)}, volume XXXVI of {\em Proc. Sympos. Pure Math.}, pages 185--203. Amer. Math. Soc., Providence, RI, 1980.

\bibitem[Kum02]{kumar2002kac}
Shrawan Kumar.
\newblock {\em Kac-{M}oody groups, their flag varieties and representation theory}, volume 204 of {\em Progress in Mathematics}.
\newblock Birkh\"auser Boston, Inc., Boston, MA, 2002.

\bibitem[LNS{\etalchar{+}}12]{lenart2012Kirillov}
Cristian Lenart, Satoshi Naito, Daisuke Sagaki, Anne Schilling, and Mark Shimozono.
\newblock A uniform model for kirillov-reshetikhin crystals i: Lifting the parabolic quantum bruhat graph.
\newblock {\em International Mathematics Research Notices}, 2015, 11 2012.

\bibitem[LS10]{lam2010quantum}
Thomas Lam and Mark Shimozono.
\newblock Quantum cohomology of g/p and homology of affine grassmannian.
\newblock {\em Acta Mathematica}, 204, 03 2010.

\bibitem[Mil21]{milicevic2021newton}
Elizabeth Milićević.
\newblock {Maximal Newton Points and the Quantum Bruhat Graph}.
\newblock {\em Michigan Mathematical Journal}, 70(3):451 -- 502, 2021.

\bibitem[MO18]{muthiah2018walk}
Dinakar Muthiah and Daniel Orr.
\newblock Walk algebras, distinguished subexpressions, and point counting in {K}ac-{M}oody flag varieties.
\newblock In {\em Representations of {L}ie algebras, quantum groups and related topics}, volume 713 of {\em Contemp. Math.}, pages 187--203. Amer. Math. Soc., Providence, RI, 2018.

\bibitem[MO19]{muthiah2019bruhat}
Dinakar Muthiah and Daniel Orr.
\newblock On the double-affine {B}ruhat order: the {$\varepsilon=1$} conjecture and classification of covers in {ADE} type.
\newblock {\em Algebr. Comb.}, 2(2):197--216, 2019.

\bibitem[Mut18]{muthiah2018iwahori}
Dinakar Muthiah.
\newblock On {I}wahori-{H}ecke algebras for {$p$}-adic loop groups: double coset basis and {B}ruhat order.
\newblock {\em Amer. J. Math.}, 140(1):221--244, 2018.

\bibitem[Mut19]{muthiah2019double}
Dinakar Muthiah.
\newblock Double-affine {K}azhdan-{L}usztig polynomials via masures.
\newblock {\em arXiv preprint arXiv:1910.13694}, 2019.

\bibitem[R{\'e}m02]{remy2002groupes}
Bertrand R{\'e}my.
\newblock Groupes de {K}ac-{M}oody d\'eploy\'es et presque d\'eploy\'es.
\newblock {\em Ast\'erisque}, (277):viii+348, 2002.

\bibitem[Ron89]{ronan1989lectures}
Mark Ronan.
\newblock Lectures on buildings, volume 7 of perspectives in mathematics.
\newblock {\em Academic Press Inc., Boston, MA}, 11:12, 1989.

\bibitem[Sch23]{schremmer2023affine}
Felix Schremmer.
\newblock Affine bruhat order and demazure products.
\newblock 2023.

\bibitem[Wel22]{welch2022classification}
Amanda Welch.
\newblock Classification of cocovers in the double affine {Bruhat} order.
\newblock {\em Electron. J. Comb.}, 29(4):research paper p4.7, 19, 2022.

\end{thebibliography}

\end{document}